\documentclass[12pt,reqno]{amsart}
\usepackage[headings]{fullpage}
\usepackage{amssymb,amsmath,amscd,bbm,tikz}
\usepackage{graphicx}
\usepackage{texdraw}
\usepackage{pb-diagram}
\usepackage[all,cmtip]{xy}
\usepackage{url}
\usepackage[bookmarks=true,%
    colorlinks=true,%
    linkcolor=blue,%
    citecolor=blue,%
    filecolor=blue,%
    menucolor=blue,%
    urlcolor=blue,%
    breaklinks=true]{hyperref}
\usepackage{slashed}    
\usepackage{tikz-cd}
\usetikzlibrary{decorations.pathmorphing}
\tikzset{snake it/.style={decorate, decoration=snake}}
\usetikzlibrary{patterns}
\usepackage{stmaryrd}
\usepackage{tabularray}

\usepackage{todonotes}
\usepackage{enumerate}
\usepackage{mathtools}

\usepackage{arydshln}

\usepackage{verbatim} 

\newtheorem{theorem}{Theorem}[section]
\theoremstyle{definition}
\newtheorem{proposition}[theorem]{Proposition}
\newtheorem{lemma}[theorem]{Lemma}
\newtheorem{definition}[theorem]{Definition}
\newtheorem{remark}[theorem]{Remark}
\newtheorem{corollary}[theorem]{Corollary}

\newtheorem{example}[theorem]{Example}

\def\SG#1{\textcolor[rgb]{1.00,0.00,0.00}{[SG:#1]}}
\def\MS#1{\textcolor[rgb]{1.00,0.00,0.00}{[MS:#1]}}
\def\CW#1{\textcolor[rgb]{1.00,0.00,0.00}{[CW:#1]}}
\def\red#1{\textcolor[rgb]{1.00,0.00,0.00}{[#1]}} 
\def\green#1{\textcolor[rgb]{0.00,0.70,0.30}{[#1]}} 

\def\BN{\mathbb N}
\def\BA{\mathbb A}
\def\BB{\mathbb B}
\def\BZ{\mathbb Z}
\def\BQ{\mathbb Q}
\def\BR{\mathbb R}
\def\BC{\mathbb C}
\def\BK{\mathbb K}
\def\BW{\mathbb W}
\def\BH{\mathbb H}
\def\BT{\mathbb T}
\def\calF{\mathcal F}
\def\calW{\mathcal W}
\def\calA{\mathcal A}
\def\calI{\mathcal I}
\def\calK{\mathcal K}
\def\calD{\mathcal D}
\def\calC{\mathcal C}
\def\calE{\mathcal E}
\def\calT{\mathcal T}
\def\calP{\mathcal P}
\def\calS{\mathcal S}
\def\calB{\mathcal B}
\def\calG{\mathcal G}
\def\calX{\mathcal X}
\def\calM{\mathcal M}
\def\calL{\mathcal L}
\def\calV{\mathcal V}
\def\calU{\mathcal U}
\def\calH{\mathcal H}
\def\calJ{\mathcal J}
\def\calN{\mathcal N}
\def\calZ{\mathcal Z}
\def\s{\sigma}
\def\la{\langle}
\def\ra{\rangle}
\def\we{\wedge}
\def\ti{\widetilde}
\def\SL{\mathrm{SL}}
\def\longto{\longrightarrow}
\def\pt{\partial}
\def\ID{I_{\Delta}}
\newcommand\Res[1]{\,\underset{#1}{\mathrm{Res}\,}}
\def\Vol{\mathrm{Vol}}
\def\a{\alpha}
\def\b{\beta}
\def\g{\gamma}
\def\d{\delta}
\def\ve{\varepsilon}
\def\th{\theta}
\def\coeff{\mathrm{coeff}}
\def\PSL{\mathrm{PSL}}
\def\Jac{\mathrm{Jac}}
\def\diagonal{\mathrm{diag}}
\def\AK{\mathrm{AK}}
\def\vlon{v_{\mathrm{lon}}}
\def\Re{\mathrm{Re}}
\def\Im{\mathrm{Im}}
\def\sgn{\mathrm{sgn}}
\def\be{\begin{equation}}
\def\ee{\end{equation}}
\def\ID{I_{\Delta}}
\def\IKD{I^{\Delta}}
\def\Ipre{I^{\mathrm{pre}}}
\def\Ibal{I^{\mathrm{bal}}}
\def\Abar{\overline{A}}
\def\Bbar{\overline{B}}
\def\Cbar{\overline{C}}
\def\emu{e_{\mu}}
\def\elambda{e_{\lambda}}
\def\cxymatrix#1{\xy*[c]\xybox{\xymatrix#1}\endxy}
\def\myu{\mathsf{u}}
\def\myv{\mathsf{v}}
\def\myfun{\eta}
\def\fourier{\mathsf{F}}
\def\mypos{\mathsf{q}}
\def\mymom{\mathsf{p}}
\def\mymu{\mu}
\def\myh{\mathsf{h}}
\def\myz{\mathsf{z}}
\def\myH{\mathsf{H}}
\def\myZ{\mathsf{Z}}
\def\myi{\mathsf{J}}
\def\Bor{\mathcal{B}}
\def\Lap{\mathcal{L}}
\newcommand{\im}{\mathsf{i}}
\newcommand{\fad}{\operatorname{\Phi}_{\mathsf{b}}}
\newcommand{\poc}[2]{\left(#1\right)_{#2}}
\def\bb{\mathsf{b}}
\def\cb{c_{\mathsf{b}}}
\def\ind{\mathrm{ind}}
\def\phih{\widehat{\phi}}
\def\psih{\widehat{\psi}}
\def\Bh{\widehat{B}}
\def\vth{\vartheta}
\def\KLV{\mathrm{KLV}}
\def\z{\zeta}
\def\DJ{\mathrm{DJ}}
\def\DH{\mathrm{DH}}
\def\DK{\mathrm{DJ}}
\def\J{\mathrm{J}}
\def\H{\mathrm{H}}
\def\Zhat{\widehat{\BZ[q]}}
\def\GL{\mathrm{GL}}
\def\e{\mathbf e}  
\def\hol{\mathrm{hol}}
\def\mod{\mathrm{mod}}
\def\diag{\mathrm{diag}}
\def\Av{\mathrm{Av}}
\def\Om{\Omega}
\def\rind{\rho}
\def\sma#1#2#3#4{\bigl(\smallmatrix#1&#2\\#3&#4\endsmallmatrix\bigr)} 
\def\mat#1#2#3#4{\begin{pmatrix}#1&#2\\#3&#4\end{pmatrix}} 
\def\MM{\overset{\longrightarrow}{M}}
\def\comp{\mathrm{comp}}
\def\Sol{\mathrm{Sol}}
\def\Ker{\mathrm{Ker}}
\usepackage{accents}
\def\vM{\accentset{\longrightarrow}{M}}
\def\hb{\hbar}
\def\vphi{\varphi}
\def\geom{\mathrm{geom}}
\def\mbfQ{\mathbf Q}
\def\tz{\hat z}
\def\tdelta{\hat \delta}
\def\llangle{\left\langle\!\left\langle}
\def\rrangle{\right\rangle\!\right\rangle}

\newcommand{\om}{\omega}
\newcommand{\prin}[2]{\left[#2\right]_{#1}}
\newcommand{\hev}[2]{\operatorname{\theta}_{#1}\!\left(#2\right)}

\newcommand{\bea}{\begin{equation}\begin{aligned}}
\newcommand{\eea}{\end{aligned}\end{equation}}
\renewcommand{\=}{\;=\;}
\newcommand{\eps}{\varepsilon}

\newcommand{\tq}{\widetilde q}
\newcommand{\tx}{\widetilde x}
\newcommand{\R}{\mathbb{R}}
\newcommand{\Q}{\mathbb{Q}}
\newcommand{\C}{\mathbb{C}}
\newcommand{\Hh}{\mathbb{H}}
\newcommand{\Z}{\mathbb{Z}}
\newcommand{\Li}{\operatorname{Li}}

\newcommand{\hroot}{\hbar^{\frac{1}{2}}}
\newcommand{\hsqrt}{\hbar^{\frac{1}{2}}}
\newcommand{\ezh}{e^{z\hbar^{\frac 12}}}

\def\CW#1{$\spadesuit\spadesuit$ \textcolor{red}{CW: \emph{#1}}
  $\spadesuit\spadesuit$}
\def\VF#1{$\spadesuit\spadesuit$ \textcolor{red}{VF: \emph{#1}}
  $\spadesuit\spadesuit$}

\makeatletter

\renewcommand\thepart{\@Roman\c@part}%
\renewcommand\part{%
   \if@noskipsec \leavevmode \fi
   \par
   \addvspace{6.7ex}%
   \@afterindentfalse
   \secdef\@part\@spart}
\def\@part[#1]#2{%
    \ifnum \c@secnumdepth >\m@ne
      \refstepcounter{part}%
      \addcontentsline{toc}{part}{Part~\thepart.\ #1}%
    \else
      \addcontentsline{toc}{part}{#1}%
    \fi
    {\parindent \z@ \raggedright
     \interlinepenalty \@M
     \normalfont
     \ifnum \c@secnumdepth >\m@ne
       \centering\large\scshape \partname~\thepart.%
       \hspace{1ex}%
     \fi%
     \large\scshape #2%
     \markboth{}{}\par}%
    \nobreak
    \vskip 4.7ex
    \@afterheading}
  \def\@spart#1{
  \refstepcounter{part}%
  \addcontentsline{toc}{part}{#1}%
    {\parindent \z@ \raggedright
     \interlinepenalty \@M
     \normalfont
     \centering\large\scshape #1\par}%
     \nobreak
     \vskip 4.7ex
     \@afterheading}
\renewcommand*\l@part[2]{%
  \ifnum \c@tocdepth >-2\relax
    \addpenalty\@secpenalty
    \addvspace{0.75em \@plus\p@}%
    \begingroup
      \parindent \z@ \rightskip \@pnumwidth
      \parfillskip -\@pnumwidth
      {\leavevmode
       \normalsize \bfseries #1\hfil \hb@xt@\@pnumwidth{\hss #2}}\par
       \nobreak
       \if@compatibility
         \global\@nobreaktrue
         \everypar{\global\@nobreakfalse\everypar{}}%
      \fi
    \endgroup
  \fi}

\def\l@subsection{\@tocline{2}{0pt}{2pc}{6pc}{}}
\makeatother

\begin{document}
\title[Summability for State Integrals of hyperbolic knots]{
 Summability for State Integrals of hyperbolic knots}
\author{Veronica Fantini}
\address{Institut des Hautes Études Scientifiques \\
          \newline
         {\tt \url{https://sites.google.com/view/vfantini/home-page}}}
\email{fantini@ihes.fr}
\author{Campbell Wheeler}
\address{Institut des Hautes Études Scientifiques \\
          \newline
         {\tt \url{https://www.ihes.fr/~wheeler/}}}
\email{wheeler@ihes.fr}

\date{\today }
\dedicatory{\it{To Maxim Kontsevich, on the occasion of his 60th birthday}}

\begin{abstract}
  We prove conjectures of Garoufalidis-Gu-Mariño~\cite{GGM:I} that perturbative series associated with the hyperbolic knots $4_1$ and $5_2$ are resurgent and Borel summable. In the process, we give an algorithm that can be used to explicitly compute the Borel--Laplace resummation as a combination of state integrals of Andersen--Kashaev~\cite{AK:I}.
  This gives a complete description of the resurgent structure in these examples and allows for explicit computations of Stokes constants.
\end{abstract}

\maketitle
{
\footnotesize
\tableofcontents
}

\section{Introduction}

The volume of a hyperbolic knot $K\subset S^3$ has a natural perturbative deformation constructed from the quantum dilogarithm~\cite{Hikami:stateint}. This deformation is a formal power series
\be
  \Phi_{K}(\hbar)
  \=
  \exp\Big(\frac{\mathrm{Vol}(S^3\backslash K)}{\hbar}\Big)\frac{\hbar^{3/2}}{\sqrt{\delta}}(1+A_1\hbar+\cdots)
\ee
where $\delta$ and $A_k$ are elements in the trace field of $K$. This formal power series $\Phi_{K}$ is conjectured to agree to all orders with the asymptotic series of the Kashaev invariant of $K$---expected to exist as part of Kashaev's volume conjecture~\cite{Kashaev:VC}. Independent of Kashaev's volume conjecture, the series $\Phi_{K}$ can be explicitly described in terms of Neumann--Zagier data~\cite{DimGar:QC} and proved to be a topological invariant~\cite{GSW}. This paper is concerned with the analytic properties of $\Phi_{K}$.

\medskip

It was conjecture by Garoufalidis~\cite{Gar:resCS} that $\Phi_{K}$ is resurgent, which means that its Borel transform has endless analytic continuation~\cite{EcalleI}. The singularities of the Borel transform are expected to be located at the values of the Chern--Simons functional at parabolic $\SL_{2}(\BC)$-flat connections of the knot complement. While there have been many examples of resurgent series associated to non-hyperbolic knots~\cite{Zagier:Vas,CG:KZser}, there has been no proof of resurgence of the asymptotic series $\Phi_{K}$ of any hyperbolic knot.

\medskip

A step further in the study of the analytic properties of $\Phi_{K}$ is to address its summability. This requires the existence of analytic functions with prescribed asymptotics. By refining Hikami's original approach~\cite{Hikami:stateint}, Andersen--Kashaev~\cite{AK:I} defined a convergent integral of products of Faddeev's quantum dilogarithm associated to certain triangulations. These integrals are called state integrals as they give a continuous analogue of state sums. Combining ideas from quantum modularity~\cite{GZ:RQMOD} and numerical Borel--Padé--Laplace resummation, Garoufalidis--Gu--Mariño~\cite{GGM:I} gave precise conjectures for the Borel--Laplace resummation of $\Phi_{K}$ for the two simplest hyperbolic knots $4_1$ and $5_2$. The main result of this paper proves their conjectures for these examples, and it can be summarised in the following theorem.

\begin{theorem}\label{thm:4152}
The series $\Phi_{4_1}$ and $\Phi_{5_2}$ are Borel--Laplace summable and their resummations are equal to combinations of state integrals.
\end{theorem}



The method that we use to prove these conjectures seems easily generalisable and will likely lead to many resurgence and summability results for a wide class of asymptotic series associated to $q$-hypergeometric functions. This would have a variety of applications in a number of areas such as topological strings~\cite{GM:tsres}, quantum $K$-theory~\cite{GivLee:QK}, and complex Chern--Simons theory~\cite{GGM:I}. Hopefully, this can also lend insights to the infinite dimensional approach considered in~\cite{KS:floer}.

\subsection{Deforming the volume of a three-manifold}
In this paper, we will consider a two parameter family of examples of asymptotic series similar to those that would come from a knot. These examples were considered in~\cite{GK:qser}. In particular, for some $A,B\in\BZ_{>0}$ with~$A\neq B$, we will consider the function
\be\label{eq:V}
  V:\Sigma\rightarrow\BC/\BZ
  \qquad\text{such that}\qquad
  V(z,m)\=B\frac{\Li_{2}(\e(z))}{(2\pi i)^2}+\frac{B}{24}+\frac{A}{2}z(z+1)+mz\,,
\ee
where $\e(z)=\exp(2\pi iz)$ and $\Sigma$ is the associated Riemann surface. In particular, $\Sigma$ is defined by taking points $p=(z,m)$ where $z\in\BC\backslash(\BZ-i\BR)$ and $m\in\BZ$ with the equivalences 
\be\label{eq:surface}
\begin{aligned}
(k+0+iy,m)&\sim(k-0+iy,m+B)\qquad\text{for }y<0\,,\;k\in\BZ\,,\\
(z,m+A)&\sim(z+1,m)\,.
\end{aligned}
\ee
Given the second equivalence relation, we can always represent points $p\in\Sigma$ as $p=(z,m)$ with $m=0,\cdots,A-1$. The critical points of $V$ are in bijective correspondence with the roots of the polynomial
\be
  P(x)\=(-x)^A-(1-x)^B\,.
\ee
Indeed, to each root $P(x_0)=0$ the point
\be\label{eq:crit.pts}
  p_0
  \=
  (z_0,m_0)
  \=
  \Big(\frac{\log(x_0)}{2\pi i},\frac{B}{2\pi i}\log(1-x_0)-A\frac{\log(x_0)}{2\pi i}-\frac{A}{2}\Big)
\ee
is a critical point of $V$. These roots are all distinct and therefore the critical points are non-degenerate (see Appendix~\ref{app:crit.pts}). Notice that all ambiguities in the logarithms lead to equivalent points in $\Sigma$. The function $V$ computes the volumes of the $4_1$ and $5_2$ knots with the parameters $(A,B)=(1,2)$ and $(2,3)$, respectively. Indeed, we find that $(z_1,m_1)=(-1/6,0)$ and $(z_2,m_2)=(-5/6,0)$ are the two critical points when $A=1$ and $B=2$ and these have volumes
\be\label{eq:41.crit.val}
  V(-1/6,0)\=0.051418\cdots i\,,
  \qquad
  V(-5/6,0)\=-0.051418\cdots i\,.
\ee
We will denote the collection $(A,B,p_0)$ by $\Xi$. Let $\mathcal{V}\subset\BC$ denote the set of critical values, which includes $\BZ$. Denote the set of rays $\BR_{\geq0}\frac{2\pi i}{\mathcal{V}-V_0}$ by $\mathcal{S}(V_{0})\subseteq\BC$. These form a peacock pattern so that $\BC\backslash\mathcal{S}(V_{0})$ consists of a countable set of cones, see Figure~\ref{fig:41.stokes}.

\medskip

We are interested in a perturbative expansion of certain integrals at the critical points of the function $V$. These integrals involve the formal series $\Psi(z,\hbar)\in\BC\llbracket\hbar\rrbracket$ defined by
\be
  \Psi(z,\hbar)
  \=
  \mu_8\exp\Big(-\frac{(2\pi i)^2}{24\hbar}-\frac{\hbar}{24}-\sum_{k=0}^{\infty}\frac{B_{k}}{k!}\Li_{2-k}(\e(z))\hbar^{k-1}\Big)\,,
\ee
where $\mu_8^8=1$ is an eighth root of unity and $B_k$ denotes the $k$-th Bernoulli number. Explicitly, for each $\Xi$ the formal series is defined by
\be\label{eq:form.int}
  \Phi_{\Xi}(\hbar)\=\int \Psi(z,\hbar)^B\exp\Big(-\frac{A}{2\hbar}(2\pi i)^2z\Big(z+1-\frac{\hbar}{2\pi i}\Big)+(2\pi i)^2\frac{m_0z}{\hbar}\Big)dz\,,
\ee
where we take the formal Gaussian integration around the critical point $z_0$. More precisely, consider the series
\be
\sum_{k=0}^{\infty}\sum_{\ell=-\lfloor k/3\rfloor}^{\infty}a_{k,\ell}w^k\hbar^\ell\in\BC\llbracket\hbar\,,w\rrbracket
\ee
defined through the following equation
\be
\begin{aligned}
  &\exp\Big(-\frac{V(z_0,m_0)}{\hbar}-\frac{\delta}{\hbar} w^2\Big)
  \sum_{k=0}^{\infty}\sum_{\ell=-\lfloor k/3\rfloor}^{\infty}a_{k,\ell}w^k\hbar^\ell\\
  &=\;
  \Psi(z_0+w,\hbar)^B\exp\Big(-\frac{A}{2\hbar}(2\pi i)^2(z_0+w)\Big(z_0+w+1-\frac{\hbar}{2\pi i}\Big)-(2\pi i)^2\frac{m_0(z_0+w)}{\hbar}\Big)\,,
\end{aligned}
\ee
for some constant $\delta$. 
Then, formal Gaussian integration gives
\be\label{eq:def.phi}
  \Phi_{\Xi}(\hbar)\=
  \sqrt{\frac{-2\pi i\hbar}{\delta}}
  \exp\Big(-\frac{V(z_0,m_0)}{\hbar}\Big)
  \sum_{k=0}^{\infty}\sum_{\ell=-\lfloor 2k/3\rfloor}^{\infty}a_{2k,\ell}\frac{(2k-1)!!}{\delta^k}\hbar^{\ell+k}.
\ee
Given that $B_{k}\Li_{2-k}(x)=O((k!)^2(2\pi|\log(x)|)^{-k})$ and exponentiation preserves the property of being Gevery-1~\cite[Thm.~5.55]{diverg-resurg-i}, it is easy to see that these series are Gevery-1, i.e. the coefficient of $\hbar^k$ grows like $A\,k!\,C^k$ as $k\to\infty$. We will prove that $\Phi_{\Xi}$ is in fact Borel--Laplace summable in Theorem~\ref{thm:main_intro}. Importantly, for $(A,B)=(1,2)$ and $(2,3)$, the series $\Phi_{\Xi}$ give the series $\Phi_{4_1}$ and $\Phi_{5_2}$ respectively. This was explicitly illustrated in~\cite[Sec.~7]{GSW} but known in previous work such as~\cite[Sec.~1.3]{GK:qser}.
\subsection{State integrals}
We are interested in the resurgent and summability properties of the series $\Phi_{\Xi}(\hbar)$, which was defined in terms of the quantum dilogarithm $\Psi(z,\hbar)$ in Equation~\eqref{eq:def.phi}. The latter was shown to be resurgent and Borel--Laplace summable to a function given in terms of Faddeev's quantum dilogarithm $\Phi(z;\tau)$ (see~\cite{GK:resQD,andersen.mero}). The function $\Phi(z;\tau)$ is a meromorphic function of $(z,\tau)\in\BC\times(\BC\backslash\BR_{\leq0})$ with poles located on a cone $z\in\BZ_{\geq0}+\tau\BZ_{\geq0}$ and zeros located on the opposite cone $z\in\BZ_{<0}+\tau\BZ_{<0}$. In addition, it has a variety of descriptions
\be\label{eq:fad.expressions}
  \frac{(q\e(z);q)_{\infty}}{(\e(z/\tau);\tq)_{\infty}}
  \=
  \Phi(z;\tau)
  \=
  \exp\Big(
    \int_{i\sqrt{\tau}\BR+\varepsilon\sqrt{\tau}}\frac{\e((z+1+\tau)w/\tau)}{(\e(w)-1)(\e(w/\tau)-1)}
    \frac{dw}{w}
  \Big)\,,
\ee
where $(x;q)_{k}=\prod_{j=0}^{k-1}(1-q^jx)$, the first equality holds for $\tau\in\BH$, and the second for $\Re(-\sqrt{\tau}-1/\sqrt{\tau})<\Re(z/\sqrt{\tau})<0$. We refer to~\cite{AK:I,Faddeev} and Appendix~\ref{app:Faddeev} for general properties of the function~$\Phi(z;\tau)$. We can now consider integrals of the same form as Equation~\eqref{eq:form.int} replacing\footnote{The relation between the variables is $\hbar=-2\pi i/\tau$.} $\Psi(z,\hbar)$ by $\Phi(z\tau;\tau)$: for $(A,B)\in\BZ_{>0}^2$ with $A\neq B$ and $m,\ell\in\BZ$ we define
\be\label{eq:state_int}
\calI_{m,\ell}(\tau)
\;:=\;
\mu_8^B\tq^{-B/24}q^{B/24}
\int_{\calJ_{\ell,\tau}}\Phi((z-\ell)\tau;\tau)^B \e\Big(\frac{A}{2} z(z\tau+\tau+1)+mz\tau\Big) dz\,,
\ee
where $\calJ_{\ell,\tau}:=(\tfrac{i}{\sqrt{\tau}}e^{-iA\epsilon}\BR_{\geq0}-\frac{1}{2}+\ell)\cup(\tfrac{i}{\sqrt{\tau}}e^{-i(A-B)\epsilon}\BR_{\leq0}-\frac{1}{2}+\ell)$ for some small $\epsilon>0$. The fact that $\Phi(z;\tau)$ is meromorphic for $\tau\in\BC\backslash\BR_{\leq0}$ implies that $\calI_{m,\ell}(\tau)$ is holomorphic for $\tau\in\BC\backslash\BR_{\leq0}$.
These integrals are not independent. Indeed, using the quasi-periodicity of Faddeev's quantum dilogarithm, $\Phi(z-\tau;\tau)=(1-\e(z))\Phi(z;\tau)$,
one easily finds the relations
\be\label{eq:state_integral_1}
\sum_{k=0}^B\binom{B}{k}(-1)^kq^{-\ell k}\calI_{m+ k,\ell}(\tau)
\=
\calI_{m,\ell+1}(\tau)
\=(-1)^A q^{A+m}\calI_{m+ A,\ell}(\tau)\,.
\ee
Therefore, over the field $\BQ(q)$ there are $\max\{A,B\}$ independent integrals. The Andersen--Kashaev state integrals for $4_1$ and $5_2$ given in~\cite{AK:I} correspond to $m=\ell=0$. 

The conjectural description of the resurgent structure of the series $\Phi_{K}$ requires the introduction of all of the $\max\{A,B\}$ integrals~\cite{GGM:I}. What was not previously clear was how to construct the explicit combination of state integrals to give the Borel--Laplace resummation of one of the asymptotic series of interest. In Section~\ref{sec:alg}, we describe an algorithm to compute this combination using the geometry of the function~$V$ and the asymptotics of Faddeev's quantum dilogarithm. 

\medskip

To understand the asymptotics of $\Phi(z\tau;\tau)$ for large $\tau$, we need to consider an unusual but useful principle branch of the dilogarithm function. For $\theta\in(0,2\pi)$ define the domain $\BC_{\theta}=\BC\backslash\big((\BZ_{\geq0}+e^{i\theta}\BR_{\leq0})\cup(\BZ_{<0}+ e^{i\theta}\BR_{\geq0})\big)$ depicted in Figure~\ref{fig:prin.bran.fad}. 
\begin{figure}[ht]
\begin{tikzpicture}[scale=0.5]
\draw[<->,thick] (0,-3) -- (0,3);
\draw[<->,thick] (-7.5,0) -- (7.5,0);
\foreach \x in {-6,-4,-2,0,2,4,6}{\filldraw[red](\x,0) circle (2pt);};
\draw[red] (0,0) -- (3,-3);
\draw[red] (2,0) -- (2+3,-3);
\draw[red] (4,0) -- (4+3,-3);
\draw[red] (6,0) -- (6+1,-1);
\draw[red] (-2,0) -- (-2-3,3);
\draw[red] (-4,0) -- (-4-3,3);
\draw[red] (-6,0) -- (-6-1,1);
\end{tikzpicture}
\caption{Domain of $\mathrm{D}_{3\pi/4}(z)$ given by $\BC\backslash\big((\BZ_{\geq0}+(i-1)\BR_{\leq0})\cup(\BZ_{<0}+(i-1)\BR_{\geq0})\big)$.}
\label{fig:prin.bran.fad}
\end{figure}

\noindent Define $\mathrm{D}_{\theta}:\BC_{\theta}\rightarrow\BC$ to be
\be
  \mathrm{D}_{\theta}(z)
  \=
  \int_{z}^{e^{i\theta/2}\infty}
  \int_{w}^{e^{i\theta/2}\infty}
  \frac{\e(\z)}{1-\e(\z)}d\z\,dw\,,
\ee
where the contours are contained in $\BC_{\theta}$. This branch of the dilogarithm allows for us to describe the asymptotics of $\Phi(z\tau;\tau)$.
\begin{theorem}\label{thm:fad.asymp}
Suppose that $\tau\in\BC\backslash\BR_{\leq0}$ and $\varepsilon\in\BR_{>0}$. Then as $|\tau|\rightarrow \infty$ with fixed argument and $z$ is bounded away from $\BC\backslash\BC_{\theta}$ by $\varepsilon$, there exists a constant $C>0$, independent of $\varepsilon$, such that
\be
    \Big|\e(-\mathrm{D}_{\theta}(z)\tau)\Phi(z\tau;\tau)
    -
    \e\Big(\frac{1}{2}\mathrm{D}_{\theta}'(z)
    +
    \sum_{k=2}^{K}\frac{B_k}{k!}\frac{(2\pi i)^{k-2}}{\tau^{k-1}}\Li_{2-k}(\e(z))\Big)
    \Big|
    <
    C\,\varepsilon^{-K}\,K!\,|\tau|^{-K}\,,
\ee
with $\theta=\arg(-1/\tau)\in(0,2\pi)$.
\end{theorem}
\noindent The proof of this theorem is described in Appendix~\ref{app:Faddeev} (see Theorem~\ref{thm:fad.asy.app}).
\subsection{Proving Borel--Laplace summability}\label{sec:BL-intro}
Starting with a Gevrey-1 formal power series, the process of resurgence requires convergence and then analytic continuation of the series's Borel transform~\cite{EcalleI,diverg-resurg-i}. 
Recall that the Borel transform $\calB$ is defined as the formal inverse of the Laplace transform
\be\label{eq:laplace}
\calL^\vartheta(\phi):=\int_0^{e^{i\vartheta}\infty} e^{-\z/\hbar} \phi(\z) \, d\z\,,
\ee
where $\vartheta\in [0,2\pi)$. 
More precisely, the Borel transform maps Gevrey-1 formal series to convergent power series in a neighbourhood of the origin 
\be\label{eq:borel}
\calB\colon\sum_{n\geq 1} a_n \hbar^n\in\BC[\![\hbar]\!]\quad\mapsto\quad \sum_{n\geq 1} \frac{a_n}{n!} \zeta^{n-1}\in\BC\{\zeta\}\,,
\ee 
where $|a_n|\leq C A^n n!$ for some constants $A,C>0$.

When the analytic continuation of the Borel transform has good growth conditions at infinity, its Laplace transform in the direction $\vartheta$ is well defined and gives an analytic function, which is called the Borel--Laplace sum of the series $s_\vartheta:=\calL^\vartheta\circ\calB$. As the argument $\vartheta$ varies, the resummation $s_\vartheta$ can jump at {\em Stokes rays} defined by rays from the origin with arguments agreeing with that of singularities in Borel plane. 
Following \'Ecalle's formalism~\cite{EcalleI,diverg-resurg-i}, the Stokes automorphism $\mathfrak{S}_{\vartheta}$ associated to each Stokes ray is defined by comparing two later Borel--Laplace sums  
\be
s_{\vartheta_+}\=s_{\vartheta_-}\mathfrak{S}_{\vartheta}
\ee
where $\vartheta_\pm=\vartheta\pm\epsilon$ for some small $\epsilon$. When the difference of the two lateral Borel--Laplace sum of $\varphi$ is given by a transseries like 
\be
s_{\vartheta_+}\varphi-s_{\vartheta_-}\varphi\=\sum_{\arg(\omega)=\vartheta} S_\omega e^{-\omega\tau}s_{\vartheta_-}(\varphi_\omega)\,
\ee
where $S_\omega\in\BC$, and $\varphi_\omega\in\BC[\![\tau^{-1}]\!]$ are the secondary resurgent series arising at the singularity~$\omega$ (see for example~\cite[Sec.~5]{GGM:I}). Then the Stokes automorphism allows the computation of the Stokes constants $S_\omega$. Indeed, by factorising $\mathfrak{S}_{\vartheta}$ into automorphisms associated to each singularity $\omega$ in various orders, we can compute the so called alien derivatives~\cite[Def.~6.63]{diverg-resurg-i}. Hence we can compute the Stokes constants $S_\omega$ from $\mathfrak{S}_\vartheta$.  
Moreover, it is enough to compute the Stokes automorphisms between two distinct direction; given two such directions $\vartheta_1, \vartheta_2$, we can write
\be
  \mathfrak{S}_{\vartheta_1,\vartheta_2}
  \=
  \prod_{\vartheta\in(\vartheta_1,\vartheta_2)}\mathfrak{S}_{\vartheta}\,.
\ee
Therefore, factorising $\mathfrak{S}_{\vartheta_1,\vartheta_2}$ into $\mathfrak{S}_{\vartheta}$ and the previous remarks allow us to compute the individual Stokes constants. This formalism is closely related to the analytic wall-crossing structure of Kontsevich--Soibelman~\cite{kontsevich2022analyticity}.

When the Borel--Laplace sum exists, it has uniform asymptotics as $|\tau|\to\infty$ for~$\arg(\tau)$ contained in an interval of length greater than or equal to $\pi$. Conversely, an analytic function with such asymptotic properties is the resummation of its asymptotics~\cite{nevanlinna}. Indeed, the analytic continuation of the Borel transform can be explicitly computed using Mordell's formula for the inverse Laplace transform~\cite{watson.reg}. 
Using this method and Theorem~\ref{thm:fad.asymp} we can deduce that Faddeev's quantum dilogarithm is the resummation of its asymptotics.
\begin{corollary}\label{cor:fad.resum}
The Borel--Laplace sum of $\Psi(z;-2\pi i/\tau)$ in the direction $\vartheta$ is given by $\mu_8\tq^{-1/24}q^{1/24}\Phi(z\tau;\tau)$
\begin{itemize}
  \item for $\arg(z+1)+\tfrac{\pi}{2}<\vartheta<\arg(z)+\tfrac{\pi}{2}$ with $\Im(z)\geq0$,
  \item and for $\arg(z)-\tfrac{\pi}{2}<\vartheta<\arg(z+1)-\tfrac{\pi}{2}$ with $\Im(z)\leq0$.
\end{itemize}
\end{corollary}
\begin{proof}
By definition, if $z\in\BC\backslash(\BR_{\leq-1}\cup\BR_{\geq0})$ then $z\in\BC_\pi$. Then the first time the branch cuts of $D_{\theta}$ intersect the point $z$ are located at $\theta=\arg(\pm(z+1))$ and $\theta=\arg(\pm z)$. From Theorem~\ref{thm:fad.asymp}, the function $\Phi(z\tau;\tau)$ therefore has uniform asymptotics for $\arg(z+1)<\theta<\arg(z)+\pi$ or $\arg(z)-\pi<\theta<\arg(z+1)$ depending on the sign of $\Im(z)$.
\end{proof}
The same method can also be used to prove that exponential integrals along their steepest descent contours are the Borel--Laplace sum of their asymptotics (a different proof is given in~\cite{borel_reg}). Thus, given that $\Phi_\Xi$ is defined by formal Gaussian integration, one might expect that a similar argument would work for the state integrals in Equation~\eqref{eq:state_int}. However, the state integrals' contours are not defined on $\Sigma$ or on the Riemann surface of the dilogarithm. Moreover asymptotically, Faddeev's quantum dilogarithm chooses branch cuts of the dilogarithm. Therefore, trying to deform state integrals into thimble integrals cannot always be done. To remedy this, we make the crucial observation that the integrands of different $\calI_{m,\ell}$ can have the same leading asymptotics depending on the branching of $D_{\th}$. This allows us to patch together a variety of integrals that follow the contour of steepest descent up to a finite collection of extra integrals (we will call {\em tails}). The asymptotics of these finitely many integrals can again be represented as a collection of state integrals up to a finite collection of new tails. This process can be repeated and proved to eventually terminate, which happens when the tails have a (relatively) exponentially small contribution. This leads to our following main result.
\begin{theorem}\label{thm:main_intro}
The asymptotic series $\Phi_{\Xi}(-2\pi i/\tau)$ are Borel--Laplace summable with Stokes rays on the countable set of rays $\tau\in\mathcal{S}(V_{0})$ with two accumulation points at $\arg(\tau)=\pm\frac{\pi}{2}$. Moreover, there is an algorithm to compute the Borel--Laplace resummation away from Stokes rays, which gives rise to a $\BZ[q^{\pm}]$-linear combination of state integrals $\calI_{m,\ell}(\tau)$ for $\Re(\tau)>0$ and $\calI_{m,\ell}^{-}(\tau)$ for $\Re(\tau)<0$.\footnote{See Remark~\ref{rem:neg.rt} or Equation~\eqref{eq:neg.si} for a definition.}
\end{theorem}
\begin{remark}\label{rem:neg.rt}
All of the proofs and statements work perfectly well for $\Re(\tau)<0$. To avoid cluttering the notation we only give the details for $\Re(\tau)>0$. To deal with the case $\Re(\tau)<0$ one would need to use the additional function
\be
  \Phi^{-}(z;\tau)
  \=
  \Phi(-z+1+\tau;-\tau)^{-1}
  \=
  \frac{(q\e(-z);q)_{\infty}}{(\e(z/\tau);\tq)_{\infty}}\,.
\ee
The function $\mu_8\tq^{-1/24}q^{1/24}\Phi^{-}(z\tau;\tau)$ gives the resummation of $\Psi(z;-2\pi i/\tau)$ for some $\vartheta\in (-\pi/2, \pi/2)$. The function $\Phi^-$ comes with its own state integrals
\be\label{eq:neg.si}
\calI_{m,\ell}^{-}(\tau)
\;:=\;
\mu_8^B\tq^{-B/24}q^{B/24}
\int_{\calJ_{\ell,\tau}^{-}}\Phi^{-}((z-\ell)\tau;\tau)^B \e\Big(\frac{A}{2} z(z\tau+\tau+1)+mz\tau\Big) dz\,,
\ee
where $\calJ_{\ell,\tau}^{-}:=(\tfrac{i}{\sqrt{-\tau}}e^{iA\epsilon}\BR_{\geq0}+\frac{1}{2}+\ell)\cup(\tfrac{i}{\sqrt{-\tau}}e^{+ i(A-B)\epsilon}\BR_{\leq0}+\frac{1}{2}+\ell)$ for some small $\epsilon>0$.
\end{remark}

%
We can use Theorem~\ref{thm:main_intro} to describe the full resurgent structure for our examples. First, notice that there are $\max\{A,B\}$ critical points. One can also see that there are exactly $\max\{A,B\}^2$ independent state integrals and asymptotic series. Indeed, consider the family of asymptotic series
\be
  \Phi_{\Xi,j}(\hbar)
  \=
  \int_{\calC_{p_0}} \Psi(z,\hbar)^B\exp\Big(-\frac{A}{2\hbar}(2\pi i)^2z\Big(z+1-\frac{\hbar}{2\pi i}\Big)+(2\pi i)^2\frac{m_0z}{\hbar}+2\pi izj\Big)dz\,,
\ee
and collect all these asymptotic series into one matrix
\be\label{eq:phimat}
  (\widehat{\Phi}_j(\hbar))_{\ell,p_0}
  \=
  \Phi_{(A,B,p_0),j+\ell}(\hbar)\,.
\ee
Then from Theorem~\ref{thm:main_intro}, we know that this whole matrix of asymptotic series is resummable for $\vartheta\in\arg(\BC\backslash(\cup_{p_0}\calS(V_0)))$ to a matrix we denote $\widehat{\calI}_{j,\vartheta}(\tau)$. We show that this matrix is invertible in Corollary~\ref{cor:phi.is.invert}. Moreover, the theorem implies that the entries of this matrix consist of finite $\BZ[q]$-integral combinations of state integrals. Therefore, for $\vartheta_1,\vartheta_2\in(\tfrac{\pi}{2},\tfrac{3\pi}{2})\cap\arg(\BC\backslash(\cup_{p_0}\calS(V_0)))$ we find that
\be
  \widehat{\calI}_{j,\vartheta_2}(\tau)
  \=
  \widehat{\calI}_{j,\vartheta_1}(\tau)
  \mathfrak{S}_{\vartheta_1,\vartheta_2}(q)
\quad\text{where}\quad
  \mathfrak{S}_{\vartheta_1,\vartheta_2}(q)
  \in
  \GL_{\max\{A,B\}}(\BZ[q^{\pm}])\,.
\ee
Moreover, for $\vartheta_1<\vartheta_2<\vartheta_3$ we have
\be
  \mathfrak{S}_{\vartheta_1,\vartheta_2}(q)
  \mathfrak{S}_{\vartheta_2,\vartheta_3}(q)
  \=
  \mathfrak{S}_{\vartheta_1,\vartheta_3}(q)\,.
\ee
 In Section~\ref{sec:ex}, we will compute a couple of Stokes constants explicitly using the algorithm of Section~\ref{sec:alg}. However, as will be clear to the reader this is not an efficient method of computation for $\vartheta$ near $\pm\tfrac{\pi}{2}$. The method is extremely effective when $\vartheta$ is near $0,\pi$ as the algorithm terminates almost immediately. Therefore, armed with Theorem~\ref{thm:main_intro} and an easy computation when $\vartheta$ is near $0,\pi$, we can compute the other Stokes constants using the methods of~\cite{GGM:I}. This involves computing the Stokes automorphism for a small $\epsilon>0$
\be
\begin{aligned}
  \mathsf{S}_+(q)
  &\=
  \mathfrak{S}_{\pi+\epsilon,\epsilon}(q)\in\GL_{\max\{A,B\}}(\BZ[q^{-1}][\![q]\!])\,,\\
  \mathsf{S}_-(q)
  &\=
  \mathfrak{S}_{\epsilon,-\pi+\epsilon}(q)\in\GL_{\max\{A,B\}}(\BZ[q][\![q^{-1}]\!])\,.
\end{aligned}
\ee
These series then store all the information of the Stokes constants. This allows for a complete computation of the Stokes constants with proofs.

This paper is organised as follows. In Section~\ref{sec:geometry} we define the geometric setting, namely a relative homology to which both the class of thimbles and of state integrals contours belong. Section~\ref{sec:asymp_state_int} is dedicated to the proof of our main results. More precisely, in Section~\ref{sec:alg} we illustrated the algorithm that leads to the proof of Theorem~\ref{thm:main_intro} in Section~\ref{sec:proof_main}. Finally in Section~\ref{sec:ex}, we illustrate the algorithm in two examples: the case $(A,B)=(1,2)$ corresponding to the $4_1$ knot in Section~\ref{sec:4_1}, and the case $(A,B)=(4,1)$ in Section~\ref{sec:A4B1}. As a result of the decomposition into state integrals, we verify the numerical computations of the first Stokes constants~\cite{GGM:I,Wh:thesis} in both examples. There are two appendices.   

\section*{Acknowledgements}

The authors wish to thank Jørgen Andersen, Bertrand Eynard, Stavros Garoufalidis, Jie Gu, Maxim Kontsevich, Marcos Mariño. The work of V.F. and C.W. has been carried out at Institut des Hautes Études Scientifiques and we thank the institute for fantastic working conditions.
This paper is a result of the ERC-SyG project, Recursive and Exact New Quantum Theory (ReNewQuantum), which received funding from the European Research Council (ERC) under the European Union's Horizon 2020 research and innovation program under grant agreement No 810573. C.W. has been supported by the Huawei Young Talents Program.

\section{From thimbles to state integrals}\label{sec:geometry}

To begin with, we must understand the homology theory that governs the thimbles associated to the function $V$ and the state integrals $\calI_{m,\ell}$. All of the topology is concentrated near $z\in\BR$ in the coordinates used in Equation~\eqref{eq:V}. There are logarithmic branch cuts at $z\in\BZ$ thus the first homology of $\Sigma$ is not finitely generated. However, for a given $\vartheta\in (-\pi/2,\pi/2)\cup(\pi/2,3\pi/2)$, we are only interested in a finitely generated subgroup. These groups can be described and depend on $\vartheta$ but allows for a decomposition of the thimbles associated to $V$ into state integral contours.

\subsection{Existence of thimbles}

To effectively understand the structure of thimbles associated to $V$, we will use two descriptions of $\Sigma$ and $V$. In particular, we choose branch cuts in the upper half plane to define a function $\Lambda$ (see Figure~\ref{fig:branchcut}).
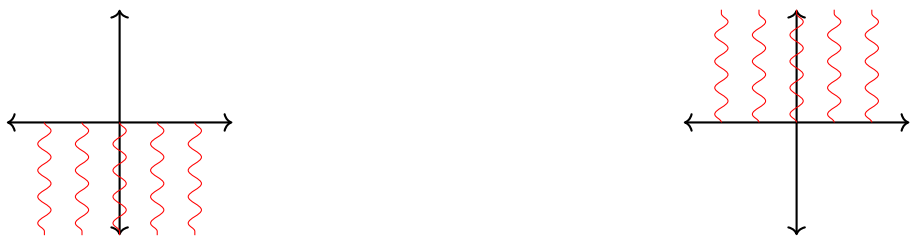
\begin{figure}[ht]
\begin{tikzpicture}
\draw[<->,thick] (-3,-1.5) -- (-3,1.5);
\draw[<->,thick] (-4.5,0) -- (-1.5,0);
\draw[red,snake it] (-3,0) -- (-3,-1.5);
\draw[red,snake it] (-2,0) -- (-2,-1.5);
\draw[red,snake it] (-3.5,0) -- (-3.5,-1.5);
\draw[red,snake it] (-2.5,0) -- (-2.5,-1.5);
\draw[red,snake it] (-4,0) -- (-4,-1.5);
.. controls (-3.555,-2) and (-1.5,-1) .. (-1.5,0);
\draw[<->,thick] (6,-1.5) -- (6,1.5);
\draw[<->,thick] (4.5,0) -- (7.5,0);
\draw[red,snake it] (5,0) -- (5,1.5);
\draw[red,snake it] (7,0) -- (7,1.5);
\draw[red,snake it] (6.5,0) -- (6.5,1.5);
\draw[red,snake it] (6,0) -- (6,1.5);
\draw[red,snake it] (5.5,0) -- (5.5,1.5);
\end{tikzpicture}
\caption{The two conventions for the branch cuts of the function $V$.}\label{fig:branchcut}
\end{figure}

\noindent Taking the principle branch of $\Li_2$ we find that for
\be\label{eq:V_Lambda}
\begin{aligned}
  V(z_+,m_+)
  &\=
  B\frac{\Li_{2}(\e(z_+))}{(2\pi i)^2}+\frac{B}{24}+\frac{A}{2}z_+(z_++1)+m_+z_+\,,\\
  \Lambda(z_-,m_-)
  &\=
  -\frac{B\Li_{2}(\e(-z_-))}{(2\pi i)^2}
  +\frac{B}{12}
  -\frac{B}{2}\Big(z_--\frac{1}{2}\Big)^2
  +\frac{A}{2}z_-(z_-+1)+m_-z_-\,,
\end{aligned}
\ee
and $p=(z_+,m_+)\in\Sigma$ we have\footnote{In fact we have more generally, $V(z,m)=\Lambda(z,m+Bk)-Bk(k+1)/2\in\BC$ for $k=\lfloor \Re(z)\rfloor$.}
\be
  V(p)\=V(z_+,m_+)\=\Lambda(z_-,m_-)\in\BC/\BZ\,,
\ee
where $z_-=z_-(p)=z_+(p)$ and $m_-=m_-(p)=m_+(p)+B\lfloor z_+(p)\rfloor$.
It will be useful to regard the surface $\Sigma$ as being given by $A$-copies of the complex plane, glued together along the branch cuts at each integer point $k\in\BZ$. In particular, allowing only $m\in\{0,\cdots,A-1\}$ as opposed to $m\in\BZ$. This can always be done using the second relation in Equation~\eqref{eq:surface}. We have the derivative $V':\Sigma\rightarrow\BC$ given by
\be\label{eq:tangent}
\begin{aligned}
  V'(z_+,m_+)
  &\=
  \frac{-B\log(1-\e(z_+))}{2\pi i}+Az_++\frac{A}{2}+m_+\,,\\
  \Lambda'(z_-,m_-)
  &\=
  \frac{-B\log(1-\e(-z_-))}{2\pi i}+(A-B)z_-+\frac{A}{2}+\frac{B}{2}+m_-\,.
\end{aligned}
\ee
Fix a critical point $p_0=(z_0,m_0)$ satisfying $V'(p_0)=0$, as given in Equation~\eqref{eq:crit.pts}. Denote the critical value of $p_0$ as $V_{0}=V(p_0)$. We are interested in the level sets 
\be
\Re((V(p)-V_0) e^{-i\vartheta})\in\cos(2\pi\vartheta)\BZ
\ee
as we vary $\vartheta\in(-\pi/2,\pi/2)\cup(\pi/2,3\pi/2)$. 
\begin{definition}\label{def:GGamma_p}
If $p\in\Sigma$ such that $V(p)\in V_0-i e^{i\vartheta}\lambda+\BZ$ for some constant $\lambda>0$, then let~$\Gamma_{p,\vartheta}$ denote the connected component of $V_0-i e^{i\vartheta}\BR_{\geq\lambda}+\BZ$ containing~$p$. For the critical point $p_0$, let $\mathcal{C}_{p_0,\vartheta}$ denote the connected component of $V_0-i e^{i\vartheta}\BR_{\geq0}+\BZ$ containing $p_0$.
\end{definition}
\begin{lemma}\label{lem:thimble}
Suppose that $ e^{i\vartheta}\in\BC\backslash\mathcal{S}(V_{0})$ and $p\in\Sigma$ such that $V(p)\in V_0-i e^{i\vartheta}\BR_{\geq0}+\BZ$. Then the set $\Gamma_{p,\vartheta}$ is a smooth curve with a parametrisation $\gamma:\BR_{\geq0}\rightarrow\Sigma$ such that
\be\label{eq:tangent.lim}
  \lim_{t\to\infty}
  \gamma'(t)
  \=
  i\sqrt{i e^{i\vartheta}}
  \quad\text{or}\quad
  -i\sqrt{i e^{i\vartheta}\frac{|(A-B)|}{A-B}}\,.
\ee
\end{lemma}
\begin{proof}
Firstly, notice that
\be
  \frac{\Li_{2}(\e(z))}{(2\pi i)^2}+\frac{1}{24}
  \=
  -\frac{\Li_{2}(1-\e(z))}{(2\pi i)^2}
  -\frac{\Li_{1}(\e(z))}{2\pi i}z
\ee
and therefore, for small enough $|z|$ there is a constant $C_0>0$ such that
\be
  \Big|\frac{\Li_{2}(\e(z))}{(2\pi i)^2}+\frac{1}{24}\Big|
  <
  C_0|z\log(-iz)|\,.
\ee
Therefore, there exists constants $C_1,C_2>0$ such that, for $k\in\BZ$ and small enough $|z|$ (independent of $k$), the representative of $V$ with the smallest absolute value\footnote{That means we choose a representative with $|\Re(V(z+k))|\leq1/2$.} satisfies
\be\label{eq:size.V}
\begin{aligned}
  &|V(z+k,m)|
  \=
  \bigg|\frac{B\Li_{2}(\e(z))}{(2\pi i)^2}+\frac{B}{24}
  +\frac{A}{2}(z+k)(z+k+1)+mz\bigg|\\
  &<
  BC_0|z\log(z)|
  +
  A|z|^2+|Ak+m+\frac{A}{2}||z|
  <
  C_1|z\log(-iz)|+C_2(1+|k|)|z|\,,
\end{aligned}
\ee
and similarly
\be\label{eq:size.lam}
  |\Lambda(z+k,m)|
  <
  C_1|z\log(iz)|+C_2(1+|k|)|z|\,.
\ee
Consider the set\footnote{Here we restrict to $m=0,\dots,A-1$ to have a finite number of conditions for each point $x$.}
\be
  \Sigma_\epsilon\=\bigcup_{\pm}\Big\{x\in\Sigma\,\Big|\,\Im(z_{\pm}(p))\in\pm\BR_{\geq0}\text{ and }\forall k\in\BZ\quad|z_{\pm}(x)-k|\geq\frac{\epsilon}{1+|k|}\Big\}\,.
\ee
Given $ e^{i\vartheta}\notin\mathcal{S}(V_{0})$ the set
\be
  V_0-i e^{i\vartheta}\BR_{\geq0}+\BZ
\ee
is disjoint from $0\in\BC/\BZ$. From Equations~\eqref{eq:size.V}~and~\eqref{eq:size.lam} for any $\delta>0$, small enough $\epsilon=\epsilon(\delta)$, and $x\in\Sigma\backslash\Sigma_{\epsilon}$ we have $|V(x)|<\delta$. Therefore, there exists $\epsilon>0$ such that $\Gamma_{p,\vartheta}\subset\Sigma_\epsilon$. 

On the set $\Sigma_{\epsilon}$ we can approximate the derivative of $V$. To do this we notice that for a fixed $\alpha\in(0,1)$ there exists constants $C_{3},C_{4}>0$ such that for $p\in\Sigma_{\epsilon}$ with $\Im(z_{+}(p))\geq0$ and $k\in\BZ$ with $\Re(z_+(p)-k)\in(-1/2, 1/2]$ we have bounds
\be
  \Big|\frac{\log(1-\e(z))}{2\pi i}\Big|
  \leq
  \frac{C_{3}}{|z-k|^{\alpha}}
  \leq
  \frac{C_{3}(1+|k|)^{\alpha}}{\epsilon^{\alpha}}
  <\frac{C_{4}|z|^{\alpha}}{\epsilon^{\alpha}}\,.
\ee
Therefore, there is a constant $C_{5}$ such that
\be
  |V'(z)-Az|
  <
  |z|\Big(\frac{C_{5}}{|z|^{1-\alpha}}+\frac{A}{2|z|}+\frac{m}{|z|}\Big)\,.
\ee
A similar bound works for $\Im(z_-(p))\leq 0$ replacing $A$ by $A-B$. Therefore, for any $\delta_1>0$ there exists an $M>0$ such that for $p\in\Sigma_{\epsilon}$ with $\Im(z_+(p))\geq0$ and $|z_+(p)|>M$ 
\be
  |V'(p)-Az_+(p)|<|z_+(p)|\delta_1\,,
\ee
and for $\Im(z_-(p))\leq0$ and $|z_-(p)|>M$ we find that
\be
  |V'(p)-(A-B)z_-(p)|<|z_-(p)|\delta_1\,.
\ee
Given the set of $p\in\Sigma_{\epsilon}$ with $|z_{\pm}(p)|\leq M$ is compact, we see that the intersection with $\Gamma_{p,\vartheta}$ is also compact. The remaining portion of $\Gamma_{p,\vartheta}$ therefore has $|z_\pm(p)|>M$. Assuming that $M$ is big enough, $|z_+(p)|>M$, and $\Im(z_+(p))>0$ we see that
\be
  |\Im(V(p)-V(p_0)-\tfrac{A}{2}z_+(p)(z_+(p)+1)-mz_+(p))|
  \leq
  C_{6}\,,
\ee
for some constant $C_6>0$.
Given that $\Im(V(p)-V(p_0))\Im(i e^{i\vartheta})^{-1}<0$, this implies that $\Re(z_+(p))\Im(i e^{i\vartheta})<0$. Therefore, the quadrant containing $z_+(p)$ is determined by the sign of $\Im(i e^{i\vartheta})$. Moreover, the curve $\Gamma_{p,\vartheta}$ has tangent at a point $p$ given by
\be\label{eq:tangent-GGamma_p}
  \frac{ e^{i\vartheta}}{iV'(p)}\,.
\ee
Therefore for $|z_+(p)|>M$, there is a constant $C_{7}$ independent of $\delta_1$ such that for all $\Im(z_+(p))\geq 0$
\be\label{eq:tangent-outside-ball}
  \Big|\frac{e^{i\vartheta}|z_+(p)|}{i V'(p)}-\frac{e^{i\vartheta}|z_+(p)|}{i Az_+(p)}\Big|
  <\delta_1C_7\,.
\ee
Therefore, the tangent direction always points towards a line parallel to $i\sqrt{i e^{i\vartheta}}$, which is in the same quadrant as $z_+(p)$. Similarly for the lower half plane we find that the tangent direction always points towards a line parallel to $-i\sqrt{i(A-B) e^{i\vartheta}}$, which is in the same quadrant as $z_-(p)$. Therefore, in the limit, the tangent is described by Equation~\eqref{eq:tangent.lim}.
\end{proof}
\begin{corollary}\label{cor:thimble}
If $e^{i\vartheta}\in\BC\backslash\mathcal{S}(V_{0})$, then the set $\mathcal{C}_{p_0,\vartheta}$ is a smooth curve with a parametrisation $\gamma:\BR\rightarrow\Sigma$ such that the limit of $\gamma'$ at $\pm\infty$ is given in Equation~\eqref{eq:tangent.lim}.
\end{corollary}
This describes the limiting behaviour of the thimbles associated to the finite set of critical points of $V$. We see that outside a compact set the behaviour is trivial and tends to a straight line towards infinity.
This is illustrated in the following example.
\begin{example}
As an example of Corollary~\ref{cor:thimble}, we can consider the case $(A,B)=(1,2)$. As we saw in the introduction, the critical points were located at $(-5/6,0)$ and $(-1/6,0)$ in Equation~\eqref{eq:41.crit.val}. Taking $\vartheta=1.6650$ we plot the set $\calC_{(-5/6,0),\vartheta}$ in Figure~\ref{fig:thimb.41}. We can see that the thimble wraps around the branch points at $-2$ and $-1$ and then tends to infinity along lines parallel to those predicted by Equation~\eqref{eq:tangent.lim}.
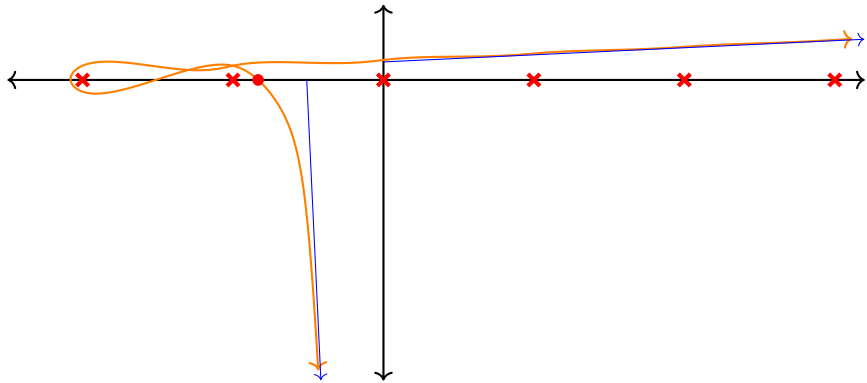
\begin{figure}[ht]
\begin{tikzpicture}[scale=2]
\draw[<->,thick] (-2.5,0) -- (3.2,0);
\draw[<->,thick] (0,-2) -- (0,0.5);
\draw[orange,thick,<->] (-0.43333, -1.9326)--(-0.44333, -1.7687)--(-0.45333, -1.6119)--(-0.46333, -1.4632)--(-0.47333, -1.3237)--(-0.48333, -1.1942)--(-0.49333, -1.0754)--(-0.50333, -0.96772)--(-0.51333, -0.87121)--(-0.52333, -0.78545)--(-0.53333, -0.70972)--(-0.54333, -0.64306)--(-0.55333, -0.58439)--(-0.56333, -0.53264)--(-0.57333, -0.48678)--(-0.58333, -0.44592)--(-0.59333, -0.40928)--(-0.60333, -0.37619)--(-0.61333, -0.34612)--(-0.62333, -0.31860)--(-0.63333, -0.29327)--(-0.64333, -0.26982)--(-0.65333, -0.24799)--(-0.66333, -0.22758)--(-0.67333, -0.20840)--(-0.68333, -0.19032)--(-0.69333, -0.17322)--(-0.70333, -0.15699)--(-0.71333, -0.14154)--(-0.72333, -0.12681)--(-0.73333, -0.11273)--(-0.74333, -0.099259)--(-0.75333, -0.086346)--(-0.76333, -0.073958)--(-0.77333, -0.062065)--(-0.78333, -0.050641)--(-0.79333, -0.039667)--(-0.80333, -0.029127)--(-0.81333, -0.019008)--(-0.82333, -0.0093009)--(-0.83333, 0)--(-0.84333, 0.0088985)--(-0.85333, 0.017396)--(-0.86333, 0.025491)--(-0.87333, 0.033181)--(-0.88333, 0.040462)--(-0.89333, 0.047330)--(-0.90333, 0.053778)--(-0.91333, 0.059802)--(-0.92333, 0.065397)--(-0.93333, 0.070561)--(-0.94333, 0.075292)--(-0.95333, 0.079593)--(-0.96333, 0.083470)--(-0.97333, 0.086930)--(-0.98333, 0.089986)--(-0.99333, 0.092652)--(-1.0033, 0.094944)--(-1.0133, 0.096881)--(-1.0233, 0.098481)--(-1.0333, 0.099765)--(-1.0433, 0.10075)--(-1.0533, 0.10145)--(-1.0633, 0.10190)--(-1.0733, 0.10209)--(-1.0833, 0.10206)--(-1.0933, 0.10181)--(-1.1033, 0.10136)--(-1.1133, 0.10073)--(-1.1233, 0.099910)--(-1.1333, 0.098927)--(-1.1433, 0.097789)--(-1.1533, 0.096503)--(-1.1633, 0.095080)--(-1.1733, 0.093526)--(-1.1833, 0.091850)--(-1.1933, 0.090059)--(-1.2033, 0.088160)--(-1.2133, 0.086158)--(-1.2233, 0.084060)--(-1.2333, 0.081871)--(-1.2433, 0.079596)--(-1.2533, 0.077241)--(-1.2633, 0.074810)--(-1.2733, 0.072308)--(-1.2833, 0.069739)--(-1.2933, 0.067106)--(-1.3033, 0.064415)--(-1.3133, 0.061669)--(-1.3233, 0.058870)--(-1.3333, 0.056024)--(-1.3433, 0.053133)--(-1.3533, 0.050200)--(-1.3633, 0.047228)--(-1.3733, 0.044221)--(-1.3833, 0.041182)--(-1.3933, 0.038113)--(-1.4033, 0.035017)--(-1.4133, 0.031897)--(-1.4233, 0.028756)--(-1.4333, 0.025596)--(-1.4433, 0.022420)--(-1.4533, 0.019230)--(-1.4633, 0.016030)--(-1.4733, 0.012821)--(-1.4833, 0.0096063)--(-1.4933, 0.0063882)--(-1.5033, 0.0031691)--(-1.5233, -0.0032617)--(-1.5333, -0.0064685)--(-1.5433, -0.0096661)--(-1.5533, -0.012852)--(-1.5633, -0.016024)--(-1.5733, -0.019179)--(-1.5833, -0.022314)--(-1.5933, -0.025428)--(-1.6033, -0.028516)--(-1.6133, -0.031577)--(-1.6233, -0.034607)--(-1.6333, -0.037604)--(-1.6433, -0.040564)--(-1.6533, -0.043484)--(-1.6633, -0.046362)--(-1.6733, -0.049194)--(-1.6833, -0.051976)--(-1.6933, -0.054705)--(-1.7033, -0.057377)--(-1.7133, -0.059988)--(-1.7233, -0.062535)--(-1.7333, -0.065012)--(-1.7433, -0.067416)--(-1.7533, -0.069742)--(-1.7633, -0.071984)--(-1.7733, -0.074138)--(-1.7833, -0.076197)--(-1.7933, -0.078156)--(-1.8033, -0.080008)--(-1.8133, -0.081747)--(-1.8233, -0.083365)--(-1.8333, -0.084854)--(-1.8433, -0.086205)--(-1.8533, -0.087409)--(-1.8633, -0.088456)--(-1.8733, -0.089335)--(-1.8833, -0.090032)--(-1.8933, -0.090535)--(-1.9033, -0.090829)--(-1.9133, -0.090896)--(-1.9233, -0.090717)--(-1.9333, -0.090270)--(-1.9433, -0.089532)--(-1.9533, -0.088475)--(-1.9633, -0.087065)--(-1.9733, -0.085264)--(-1.9833, -0.083029)--(-1.9933, -0.080306)--(-2.0033, -0.077028)--(-2.0133, -0.073115)--(-2.0233, -0.068459)--(-2.0333, -0.062912)--(-2.0433, -0.056253)--(-2.0533, -0.048106)--(-2.0633, -0.037701)--(-2.0733, -0.022747)--(-2.0793, -0.0057982)--(-2.0793, 0.016623)--(-2.0733, 0.033969)--(-2.0633, 0.049578)--(-2.0533, 0.060644)--(-2.0433, 0.069475)--(-2.0333, 0.076853)--(-2.0233, 0.083160)--(-2.0133, 0.088626)--(-2.0033, 0.093398)--(-1.9933, 0.097586)--(-1.9833, 0.10127)--(-1.9733, 0.10451)--(-1.9633, 0.10736)--(-1.9533, 0.10986)--(-1.9433, 0.11205)--(-1.9333, 0.11396)--(-1.9233, 0.11560)--(-1.9133, 0.11701)--(-1.9033, 0.11821)--(-1.8933, 0.11921)--(-1.8833, 0.12002)--(-1.8733, 0.12066)--(-1.8633, 0.12115)--(-1.8533, 0.12149)--(-1.8433, 0.12169)--(-1.8333, 0.12177)--(-1.8233, 0.12172)--(-1.8133, 0.12157)--(-1.8033, 0.12132)--(-1.7933, 0.12097)--(-1.7833, 0.12052)--(-1.7733, 0.12000)--(-1.7633, 0.11939)--(-1.7533, 0.11872)--(-1.7433, 0.11797)--(-1.7333, 0.11716)--(-1.7233, 0.11629)--(-1.7133, 0.11536)--(-1.7033, 0.11438)--(-1.6933, 0.11336)--(-1.6833, 0.11229)--(-1.6733, 0.11118)--(-1.6633, 0.11003)--(-1.6533, 0.10885)--(-1.6433, 0.10763)--(-1.6333, 0.10639)--(-1.6233, 0.10512)--(-1.6133, 0.10383)--(-1.6033, 0.10253)--(-1.5933, 0.10120)--(-1.5833, 0.099861)--(-1.5733, 0.098509)--(-1.5633, 0.097149)--(-1.5533, 0.095783)--(-1.5433, 0.094413)--(-1.5333, 0.093042)--(-1.5233, 0.091672)--(-1.5133, 0.090306)--(-1.5033, 0.088946)--(-1.4933, 0.087595)--(-1.4833, 0.086255)--(-1.4733, 0.084929)--(-1.4633, 0.083620)--(-1.4533, 0.082329)--(-1.4433, 0.081060)--(-1.4333, 0.079815)--(-1.4233, 0.078597)--(-1.4133, 0.077409)--(-1.4033, 0.076253)--(-1.3933, 0.075132)--(-1.3833, 0.074049)--(-1.3733, 0.073007)--(-1.3633, 0.072009)--(-1.3533, 0.071058)--(-1.3433, 0.070157)--(-1.3333, 0.069311)--(-1.3233, 0.068521)--(-1.3133, 0.067792)--(-1.3033, 0.067127)--(-1.2933, 0.066530)--(-1.2833, 0.066005)--(-1.2733, 0.065556)--(-1.2633, 0.065187)--(-1.2533, 0.064903)--(-1.2433, 0.064707)--(-1.2333, 0.064605)--(-1.2233, 0.064601)--(-1.2133, 0.064699)--(-1.2033, 0.064906)--(-1.1933, 0.065226)--(-1.1833, 0.065663)--(-1.1733, 0.066223)--(-1.1633, 0.066910)--(-1.1533, 0.067729)--(-1.1433, 0.068683)--(-1.1333, 0.069775)--(-1.1233, 0.071008)--(-1.1133, 0.072381)--(-1.1033, 0.073893)--(-1.0933, 0.075540)--(-1.0833, 0.077314)--(-1.0733, 0.079204)--(-1.0633, 0.081195)--(-1.0533, 0.083267)--(-1.0433, 0.085395)--(-1.0333, 0.087555)--(-1.0233, 0.089719)--(-1.0133, 0.091859)--(-1.0033, 0.093951)--(-0.99333, 0.095976)--(-0.98333, 0.097917)--(-0.97333, 0.099764)--(-0.96333, 0.10151)--(-0.95333, 0.10315)--(-0.94333, 0.10469)--(-0.93333, 0.10612)--(-0.92333, 0.10745)--(-0.91333, 0.10868)--(-0.90333, 0.10982)--(-0.89333, 0.11087)--(-0.88333, 0.11183)--(-0.87333, 0.11271)--(-0.86333, 0.11351)--(-0.85333, 0.11424)--(-0.84333, 0.11490)--(-0.83333, 0.11549)--(-0.82333, 0.11602)--(-0.81333, 0.11650)--(-0.80333, 0.11691)--(-0.79333, 0.11728)--(-0.78333, 0.11759)--(-0.77333, 0.11786)--(-0.76333, 0.11808)--(-0.75333, 0.11826)--(-0.74333, 0.11840)--(-0.73333, 0.11850)--(-0.72333, 0.11857)--(-0.71333, 0.11861)--(-0.70333, 0.11861)--(-0.69333, 0.11859)--(-0.68333, 0.11854)--(-0.67333, 0.11846)--(-0.66333, 0.11836)--(-0.65333, 0.11824)--(-0.64333, 0.11811)--(-0.63333, 0.11795)--(-0.62333, 0.11777)--(-0.61333, 0.11758)--(-0.60333, 0.11738)--(-0.59333, 0.11717)--(-0.58333, 0.11694)--(-0.57333, 0.11671)--(-0.56333, 0.11646)--(-0.55333, 0.11622)--(-0.54333, 0.11596)--(-0.53333, 0.11571)--(-0.52333, 0.11545)--(-0.51333, 0.11519)--(-0.50333, 0.11493)--(-0.49333, 0.11468)--(-0.48333, 0.11443)--(-0.47333, 0.11418)--(-0.46333, 0.11394)--(-0.45333, 0.11370)--(-0.44333, 0.11348)--(-0.43333, 0.11326)--(-0.42333, 0.11306)--(-0.41333, 0.11287)--(-0.40333, 0.11269)--(-0.39333, 0.11253)--(-0.38333, 0.11238)--(-0.37333, 0.11226)--(-0.36333, 0.11215)--(-0.35333, 0.11207)--(-0.34333, 0.11200)--(-0.33333, 0.11197)--(-0.32333, 0.11195)--(-0.31333, 0.11197)--(-0.30333, 0.11201)--(-0.29333, 0.11209)--(-0.28333, 0.11219)--(-0.27333, 0.11233)--(-0.26333, 0.11251)--(-0.25333, 0.11272)--(-0.24333, 0.11297)--(-0.23333, 0.11326)--(-0.22333, 0.11359)--(-0.21333, 0.11397)--(-0.20333, 0.11439)--(-0.19333, 0.11485)--(-0.18333, 0.11536)--(-0.17333, 0.11592)--(-0.16333, 0.11653)--(-0.15333, 0.11719)--(-0.14333, 0.11791)--(-0.13333, 0.11867)--(-0.12333, 0.11948)--(-0.11333, 0.12034)--(-0.10333, 0.12125)--(-0.093333, 0.12221)--(-0.083333, 0.12321)--(-0.073333, 0.12425)--(-0.063333, 0.12532)--(-0.053333, 0.12643)--(-0.043333, 0.12756)--(-0.033333, 0.12870)--(-0.023333, 0.12986)--(-0.013333, 0.13101)--(-0.0033333, 0.13217)--(0.0066667, 0.13331)--(0.016667, 0.13443)--(0.026667, 0.13553)--(0.036667, 0.13660)--(0.046667, 0.13764)--(0.056667, 0.13865)--(0.066667, 0.13962)--(0.076667, 0.14055)--(0.086667, 0.14144)--(0.096667, 0.14229)--(0.10667, 0.14310)--(0.11667, 0.14387)--(0.12667, 0.14461)--(0.13667, 0.14530)--(0.14667, 0.14596)--(0.15667, 0.14659)--(0.16667, 0.14718)--(0.17667, 0.14773)--(0.18667, 0.14826)--(0.19667, 0.14875)--(0.20667, 0.14921)--(0.21667, 0.14965)--(0.22667, 0.15006)--(0.23667, 0.15044)--(0.24667, 0.15080)--(0.25667, 0.15114)--(0.26667, 0.15145)--(0.27667, 0.15174)--(0.28667, 0.15202)--(0.29667, 0.15227)--(0.30667, 0.15251)--(0.31667, 0.15273)--(0.32667, 0.15293)--(0.33667, 0.15312)--(0.34667, 0.15330)--(0.35667, 0.15347)--(0.36667, 0.15362)--(0.37667, 0.15376)--(0.38667, 0.15390)--(0.39667, 0.15402)--(0.40667, 0.15414)--(0.41667, 0.15425)--(0.42667, 0.15436)--(0.43667, 0.15445)--(0.44667, 0.15455)--(0.45667, 0.15464)--(0.46667, 0.15473)--(0.47667, 0.15482)--(0.48667, 0.15491)--(0.49667, 0.15499)--(0.50667, 0.15508)--(0.51667, 0.15517)--(0.52667, 0.15527)--(0.53667, 0.15536)--(0.54667, 0.15546)--(0.55667, 0.15557)--(0.56667, 0.15568)--(0.57667, 0.15579)--(0.58667, 0.15592)--(0.59667, 0.15605)--(0.60667, 0.15619)--(0.61667, 0.15634)--(0.62667, 0.15651)--(0.63667, 0.15668)--(0.64667, 0.15687)--(0.65667, 0.15707)--(0.66667, 0.15728)--(0.67667, 0.15751)--(0.68667, 0.15775)--(0.69667, 0.15801)--(0.70667, 0.15829)--(0.71667, 0.15858)--(0.72667, 0.15890)--(0.73667, 0.15923)--(0.74667, 0.15958)--(0.75667, 0.15996)--(0.76667, 0.16035)--(0.77667, 0.16077)--(0.78667, 0.16121)--(0.79667, 0.16167)--(0.80667, 0.16216)--(0.81667, 0.16267)--(0.82667, 0.16321)--(0.83667, 0.16377)--(0.84667, 0.16435)--(0.85667, 0.16496)--(0.86667, 0.16559)--(0.87667, 0.16625)--(0.88667, 0.16693)--(0.89667, 0.16763)--(0.90667, 0.16835)--(0.91667, 0.16909)--(0.92667, 0.16985)--(0.93667, 0.17062)--(0.94667, 0.17140)--(0.95667, 0.17220)--(0.96667, 0.17300)--(0.97667, 0.17381)--(0.98667, 0.17462)--(0.99667, 0.17543)--(1.0067, 0.17624)--(1.0167, 0.17704)--(1.0267, 0.17784)--(1.0367, 0.17862)--(1.0467, 0.17939)--(1.0567, 0.18015)--(1.0667, 0.18089)--(1.0767, 0.18162)--(1.0867, 0.18233)--(1.0967, 0.18302)--(1.1067, 0.18369)--(1.1167, 0.18434)--(1.1267, 0.18498)--(1.1367, 0.18559)--(1.1467, 0.18618)--(1.1567, 0.18676)--(1.1667, 0.18731)--(1.1767, 0.18785)--(1.1867, 0.18837)--(1.1967, 0.18887)--(1.2067, 0.18936)--(1.2167, 0.18983)--(1.2267, 0.19028)--(1.2367, 0.19072)--(1.2467, 0.19114)--(1.2567, 0.19154)--(1.2667, 0.19194)--(1.2767, 0.19232)--(1.2867, 0.19269)--(1.2967, 0.19304)--(1.3067, 0.19339)--(1.3167, 0.19372)--(1.3267, 0.19405)--(1.3367, 0.19436)--(1.3467, 0.19467)--(1.3567, 0.19497)--(1.3667, 0.19526)--(1.3767, 0.19554)--(1.3867, 0.19582)--(1.3967, 0.19609)--(1.4067, 0.19636)--(1.4167, 0.19662)--(1.4267, 0.19688)--(1.4367, 0.19713)--(1.4467, 0.19739)--(1.4567, 0.19764)--(1.4667, 0.19788)--(1.4767, 0.19813)--(1.4867, 0.19838)--(1.4967, 0.19862)--(1.5067, 0.19887)--(1.5167, 0.19912)--(1.5267, 0.19936)--(1.5367, 0.19961)--(1.5467, 0.19987)--(1.5567, 0.20012)--(1.5667, 0.20038)--(1.5767, 0.20065)--(1.5867, 0.20092)--(1.5967, 0.20119)--(1.6067, 0.20147)--(1.6167, 0.20176)--(1.6267, 0.20205)--(1.6367, 0.20235)--(1.6467, 0.20266)--(1.6567, 0.20298)--(1.6667, 0.20330)--(1.6767, 0.20363)--(1.6867, 0.20398)--(1.6967, 0.20433)--(1.7067, 0.20469)--(1.7167, 0.20507)--(1.7267, 0.20546)--(1.7367, 0.20585)--(1.7467, 0.20626)--(1.7567, 0.20669)--(1.7667, 0.20712)--(1.7767, 0.20757)--(1.7867, 0.20803)--(1.7967, 0.20850)--(1.8067, 0.20899)--(1.8167, 0.20949)--(1.8267, 0.21001)--(1.8367, 0.21053)--(1.8467, 0.21107)--(1.8567, 0.21163)--(1.8667, 0.21219)--(1.8767, 0.21277)--(1.8867, 0.21336)--(1.8967, 0.21396)--(1.9067, 0.21457)--(1.9167, 0.21520)--(1.9267, 0.21583)--(1.9367, 0.21646)--(1.9467, 0.21711)--(1.9567, 0.21775)--(1.9667, 0.21841)--(1.9767, 0.21906)--(1.9867, 0.21972)--(1.9967, 0.22038)--(2.0067, 0.22103)--(2.0167, 0.22169)--(2.0267, 0.22234)--(2.0367, 0.22299)--(2.0467, 0.22363)--(2.0567, 0.22426)--(2.0667, 0.22489)--(2.0767, 0.22551)--(2.0867, 0.22612)--(2.0967, 0.22672)--(2.1067, 0.22731)--(2.1167, 0.22789)--(2.1267, 0.22846)--(2.1367, 0.22903)--(2.1467, 0.22958)--(2.1567, 0.23011)--(2.1667, 0.23064)--(2.1767, 0.23116)--(2.1867, 0.23167)--(2.1967, 0.23216)--(2.2067, 0.23265)--(2.2167, 0.23313)--(2.2267, 0.23359)--(2.2367, 0.23405)--(2.2467, 0.23450)--(2.2567, 0.23493)--(2.2667, 0.23536)--(2.2767, 0.23579)--(2.2867, 0.23620)--(2.2967, 0.23660)--(2.3067, 0.23700)--(2.3167, 0.23739)--(2.3267, 0.23778)--(2.3367, 0.23816)--(2.3467, 0.23853)--(2.3567, 0.23890)--(2.3667, 0.23926)--(2.3767, 0.23962)--(2.3867, 0.23997)--(2.3967, 0.24032)--(2.4067, 0.24066)--(2.4167, 0.24101)--(2.4267, 0.24135)--(2.4367, 0.24168)--(2.4467, 0.24202)--(2.4567, 0.24235)--(2.4667, 0.24268)--(2.4767, 0.24302)--(2.4867, 0.24335)--(2.4967, 0.24368)--(2.5067, 0.24401)--(2.5167, 0.24434)--(2.5267, 0.24467)--(2.5367, 0.24501)--(2.5467, 0.24534)--(2.5567, 0.24568)--(2.5667, 0.24602)--(2.5767, 0.24637)--(2.5867, 0.24671)--(2.5967, 0.24706)--(2.6067, 0.24741)--(2.6167, 0.24777)--(2.6267, 0.24813)--(2.6367, 0.24850)--(2.6467, 0.24887)--(2.6567, 0.24925)--(2.6667, 0.24963)--(2.6767, 0.25002)--(2.6867, 0.25041)--(2.6967, 0.25081)--(2.7067, 0.25122)--(2.7167, 0.25164)--(2.7267, 0.25206)--(2.7367, 0.25249)--(2.7467, 0.25292)--(2.7567, 0.25337)--(2.7667, 0.25382)--(2.7767, 0.25428)--(2.7867, 0.25475)--(2.7967, 0.25522)--(2.8067, 0.25571)--(2.8167, 0.25620)--(2.8267, 0.25670)--(2.8367, 0.25720)--(2.8467, 0.25772)--(2.8567, 0.25824)--(2.8667, 0.25877)--(2.8767, 0.25930)--(2.8867, 0.25985)--(2.8967, 0.26039)--(2.9067, 0.26095)--(2.9167, 0.26151)--(2.9267, 0.26207)--(2.9367, 0.26264)--(2.9467, 0.26321)--(2.9567, 0.26379)--(2.9667, 0.26436)--(2.9767, 0.26494)--(2.9867, 0.26552)--(2.9967, 0.26610)--(3.0067, 0.26668)--(3.0167, 0.26726)--(3.0267, 0.26783)--(3.0367, 0.26841)--(3.0467, 0.26898)--(3.0567, 0.26955)--(3.0667, 0.27011)--(3.0767, 0.27067)--(3.0867, 0.27123)--(3.0967, 0.27178)--(3.1067, 0.27232)--(3.1167, 0.27286);
\filldraw[red] (-5/6,0) circle (1pt);
\draw[ultra thick,red,xshift=3cm] (-0.04,-0.04)--(0.04,0.04);
\draw[ultra thick,red,xshift=3cm] (0.04,-0.04)--(-0.04,0.04);
\draw[ultra thick,red,xshift=2cm] (-0.04,-0.04)--(0.04,0.04);
\draw[ultra thick,red,xshift=2cm] (0.04,-0.04)--(-0.04,0.04);
\draw[ultra thick,red,xshift=1cm] (-0.04,-0.04)--(0.04,0.04);
\draw[ultra thick,red,xshift=1cm] (0.04,-0.04)--(-0.04,0.04);
\draw[ultra thick,red] (-0.04,-0.04)--(0.04,0.04);
\draw[ultra thick,red] (0.04,-0.04)--(-0.04,0.04);
\draw[ultra thick,red,xshift=-1cm] (-0.04,-0.04)--(0.04,0.04);
\draw[ultra thick,red,xshift=-1cm] (0.04,-0.04)--(-0.04,0.04);
\draw[ultra thick,red,xshift=-2cm] (-0.04,-0.04)--(0.04,0.04);
\draw[ultra thick,red,xshift=-2cm] (0.04,-0.04)--(-0.04,0.04);
\draw[blue,yshift=0.12cm,->,opacity=0.5] (0,0)--(3.1964,0.15074);
\draw[blue,xshift=-0.51cm,->,opacity=0.5] (0,0)--(0.094213,-1.9978);
\end{tikzpicture}
\caption{This figure depicts in \textcolor{orange}{orange} the analytic continuation of the set $\calC_{(-5/6,0),\vartheta}$ for $\vartheta=1.6650$ and $(A,B)=(1,2)$. The \textcolor{blue}{blue} lines are parallel to the lines $\sqrt{-i e^{i\vartheta}}$ and $\sqrt{i e^{i\vartheta}}$. The \textcolor{red}{red} crosses are the branch points of $V$.
}\label{fig:thimb.41}
\end{figure}
\end{example}
Lemma~\ref{lem:thimble} gives a complete description of the thimbles outside a compact set (depending on $\vartheta$). Next we study the behaviour of the thimbles inside the compact set given by Corollary~\ref{cor:thimble} (specifically the constant $M$ from Lemma~\ref{lem:thimble}). This will be described in the following sections.

\subsection{A relative homology for state integrals contours}
In this section, we define the homology $H_\bullet(\Sigma,D_{\infty,\theta})$ relative to a region $D_{\infty,\theta}$. The first relative homology group $H_1(\Sigma,D_{\infty,\theta})$ contains lifts of the state integral contours $\calJ_{\ell,\tau}$ to $\Sigma$ and the thimbles $\calC_{p_0,\vartheta}$. First, we construct a finite dimensional homology $H_\bullet(X_{\epsilon,M},D_{\infty,\theta})$ defined for a compact surface $X_{\epsilon,M}\subset\Sigma$. Then, $H_\bullet(\Sigma,D_{\infty,\theta})$ will be defined via the limits $\epsilon\to0$ and $M\to\infty$.
\subsubsection{The surface $X_{\epsilon,M}$}\label{sec:X_eps_M}
We choose the set of coordinates $z_{m,\pm}$ for $m=0,\ldots, A-1$. For given constants $M,\epsilon>0$, we define the surface $X_{\epsilon, M}$ by gluing different local coordinate charts. Let $m\in\{0,\dots,A-1\}$, on the $m$-sheet of $\Sigma$ we define 
\begin{align}
{\calU}_{\epsilon,M,m}^{\pm}&\;:=\;\left\lbrace p\in\Sigma\colon |\Re(z_{m,\pm}(p))|\leq M\,,\; \Im(\pm z_{m,\pm}(p))\geq-\frac{\epsilon}{M} \right\rbrace\,, \\
{\calV}_{\epsilon,m}^{\pm}&\;:=\;\bigcup_{k\in\BZ}\Big\{p\in\Sigma \colon |z_{m,\pm}(p)-k|\geq\frac{\epsilon}{1+|k|}\Big\}\,,\\
{\calW}_{\epsilon,M,m}^{\pm}&\;:=\;\left\lbrace p\in\Sigma\colon \Im(\pm z_{m,\pm}(p))\geq\frac{\epsilon}{M} \right\rbrace\,.
\end{align}
Then, we define the surface
\be
  X_{\epsilon, M}
  \=
  \Sigma_{\epsilon,M}^{+}\cup\Sigma_{\epsilon,M}^{-}\subset\Sigma\,,
  \qquad\text{where}\qquad 
  \Sigma_{\epsilon,M}^{\pm}
  \;:=\;\bigcup_{m=1}^A{\calU}_{\epsilon,M,m}^{\pm}\cap{\calV}_{\epsilon,m}^{\pm}\cup{\calW}_{\epsilon,M,m}^{\pm}\,.
\ee
We represent one of the sheets defining $\Sigma_{\epsilon,M}^{+}$ in Figure~\ref{fig:sigma_eps_M}.
\begin{figure}[ht]
\begin{tikzpicture}[scale=0.9]
\filldraw[fill=pink!30,color=pink!30] (-4,0) rectangle (8,2.5);
\filldraw[fill=pink!30,color=pink!30] (-4,-0.2) rectangle (8,0);
\foreach \x in {0,...,147}
\draw[red,xshift=\x*0.1 cm] (-6.5,0.3)--(-4.3,2.5);
\foreach \x in {0,...,21}
\draw[red] (8.3+\x*0.1,0.3)--(10.5,2.5-\x*0.1);
\foreach \x in {0,...,21}
\draw[red] (-6.5,2.5-\x*0.1)--(-6.5+\x*0.1,2.5);
\filldraw[color=pink!30,pattern=north east lines,pattern color=red] (-5,-3) rectangle (-5.5,-3.5);
\filldraw[fill=pink!30,color=pink!30] (-2,-3) rectangle (-2.5,-3.5);
\filldraw[fill=white,thick] (0.5,-3.25) circle (0.25);
\node[right,font=\tiny] at (-5,-3.25) {$\calW^+_{\epsilon,M,m}$};
\node[right,font=\tiny] at (-2,-3.25) {$\calU^+_{\epsilon,M,m}$};
\node[right,font=\tiny] at (0.75,-3.25) {$\BC\setminus \calV^+_{\epsilon,m}$};
\filldraw[fill=white,thick](2,0) circle (0.6);
\filldraw[fill=white,thick](0,0) circle (0.5);
\filldraw[fill=white,thick](4,0) circle (0.5);
\filldraw[fill=white,thick](-4,0) circle (0.3);
\filldraw[fill=white,thick](8,0) circle (0.3);
\filldraw[fill=white,thick](-6,0) circle (0.2);
\filldraw[fill=white,thick](10,0) circle (0.2);
\filldraw[fill=white,thick](6,0) circle (0.4);
\filldraw[fill=white,thick](-2,0) circle (0.4);
\draw[red,snake it] (2,0) -- (2,-2.5);
\draw[red,snake it] (-2,0) -- (-2,-2.5);
\draw[red,snake it] (0,0) -- (0,-2.5);
\draw[red,snake it] (4,0) -- (4,-2.5);
\draw[red,snake it] (-4,0) -- (-4,-2.5);
\draw[red,snake it] (6,0) -- (6,-2.5);
\draw[red,snake it] (8,0) -- (8,-2.5);
\draw[red,snake it] (-6,0) -- (-6,-2.5);
\draw[red,snake it] (10,0) -- (10,-2.5);
\foreach \x in {-6,-4,-2,0,2,4,6,8,10}
\filldraw (\x,0) circle (2.5pt);
\draw[->] (-7,0) -- (10.5,0);
\draw[->] (2,-2.5)--(2,3);
\end{tikzpicture}
\caption{The $m$-sheet of $\Sigma_{\epsilon,M}^+$ in the $z_{m,+}$ coordinates.}\label{fig:sigma_eps_M}
\end{figure}
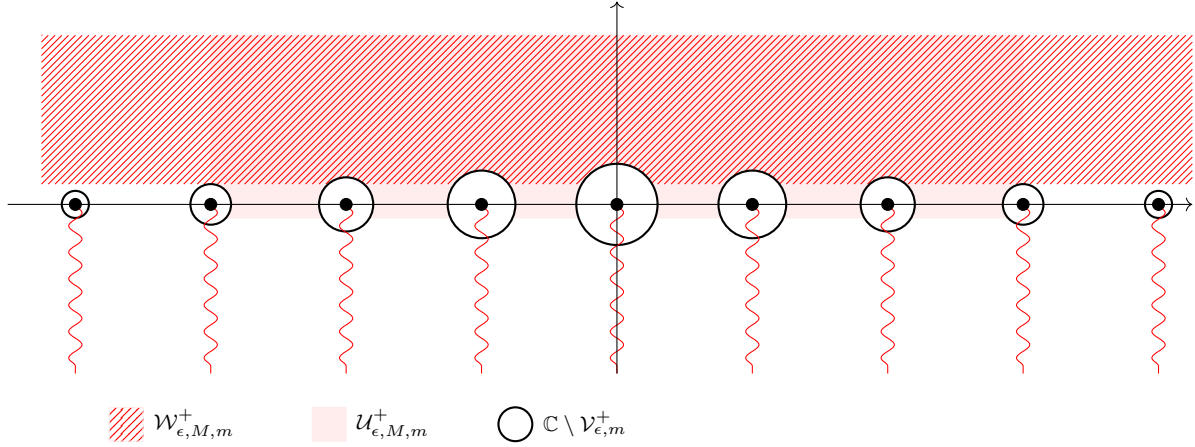
\subsubsection{Behaviour at infinity}
The dilogarithm $\Li_{2}(\e(z))$ is exponentially small as $\Im(z)\to\infty$. Therefore, the behaviour of the thimbles at infinity is governed by the quadratic functions
\be
\begin{aligned}
  &P_{+}:\Sigma_{\epsilon,M}^{+}\rightarrow\BC\,,
  &\qquad\text{such that}\qquad&
  P_{+}(p)\=\frac{A}{2}\,z_{+}(p)^2\,,\\
  &P_{-}:\Sigma_{\epsilon,M}^{-}\rightarrow\BC\,,
  &\qquad\text{such that}\qquad&
  P_{-}(p)\=\frac{A-B}{2}\,z_{-}(p)^2\,.
\end{aligned}
\ee
We define the sets $D_{L,\theta}^{\pm}\subseteq P_{\pm}^{-1}\big(\{w\in\BC\;|\;-\Im(we^{-i\theta})>L\}\big)$ to be the connected components containing points $p$ such that $z_{+}(p)=e^{i\theta/2}e^{-iA\epsilon}\in D_{L,\theta}^{+}$ and $z_{-}(p)=-e^{i\theta/2}e^{-i(A-B)\epsilon}\in D_{L,\theta}^{-}$ respectively. For large enough $L$, these sets have arguments with $\theta\in(0,2\pi)$
\be\label{eq:D.angles}
\begin{aligned}
  \arg(D_{L,\theta}^{+})
  &\=
  \Big(\max\big\{0,\tfrac{\theta}{2}-\tfrac{\pi}{2}\big\},\min\big\{\pi,\tfrac{\theta}{2}\big\}\Big)\,,\\
  \arg(D_{L,\theta}^{-})
  &\=
  \Big(\max\big\{\pi,\tfrac{\theta}{2}+\tfrac{3\pi}{4}-\tfrac{(A-B)\pi}{4|A-B|}\big\},\min\big\{2\pi,\tfrac{\theta}{2}+\tfrac{5\pi}{4}-\tfrac{(A-B)\pi}{4|A-B|}\big\}\Big)\,.
\end{aligned}
\ee
The regions are depicted in blue for $\theta\in(\tfrac{\pi}{2},\tfrac{3\pi}{2})$ in Figure~\ref{fig:good.region.I} and Figure~\ref{fig:good.region.II}.
We denote with $D_{L,\theta}$ the union of the regions $D_{L,\theta}^+$ and $D_{L,\theta}^{-}$.

\subsubsection{The relative homology} 
%
%
\begin{proposition}\label{prop:relative-homology} 
For $\epsilon,M,L>0$, and $\theta\in(0,2\pi)$, the relative homology group $H_1(X_{\epsilon,M},D_{L,\theta})$ is finite dimensional. 
\end{proposition}
\begin{proof}
We consider the Mayer--Vietoris sequence:
\[
\begin{tikzcd}
 0\arrow[r]& H_1(\Sigma_{\epsilon,M}^{+}\,;D_{L,\theta}^{+})\oplus H_1(\Sigma_{\epsilon,M}^{-}\,;D_{L,\theta}^{-})\arrow[r]&H_1(X_{\epsilon,M}\,,D_{L,\theta})\arrow[dll,controls={(3.5,-0.5) and (-3.5,0.5)}] \\
H_0(\Sigma_{\epsilon,M}^{+}\cap\Sigma_{\epsilon,M}^{-})\arrow[r]&H_0(\Sigma_{\epsilon,M}^{+}\,,D_{L,\theta}^{+})\oplus H_0(\Sigma_{\epsilon,M}^{-}\,,D_{L,\theta}^{-})\arrow[r]&H_0(X_{\epsilon,M}\,,D_{L,\theta})\arrow[r]&0
\end{tikzcd}
\]
Notice that $H_n(\Sigma_{\epsilon,M}^{-}\,,D_{L,\theta}^{-})\cong H_n(\Sigma_{\epsilon,M}^{+}\,,D_{L,\theta}^{+})\cong \{0\}$ for every $n\ge 0$, because $\Sigma_{\epsilon,M}^{\pm}$ are homeomorphic to a disjoint union of discs and $D_{L,\theta,\pm}$ are connected; therefore $H_1(X_{\epsilon,M}\,,D_{L,\theta})\cong H_0(\Sigma_{\epsilon,M}^{+}\cap\Sigma_{\epsilon,M}^{-})$.

\medskip

In addition, $\Sigma_{\epsilon,M}^{-}\cap\Sigma_{\epsilon,M}^{+}$ is homeomorphic to the disjoint union of finitely many discs (see Figure~\ref{fig:intersection}), hence $H_0(\Sigma_{\epsilon,M}^{-}\cap\Sigma_{\epsilon,M}^{+})$ is generated by a subset of the class of points $[p_{m,\ell}]$ labeled by $m=0,\ldots, A-1$ and $\ell=-M,\dots,M$. Thus, we deduce that $H_1(X_{\epsilon,M}\,,D_{L,\theta})$ is finite dimensional too. 
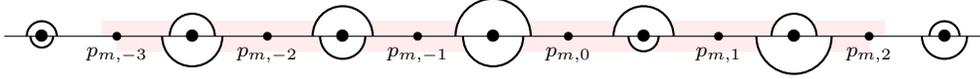
\begin{figure}[ht]
\begin{tikzpicture}
\filldraw[fill=pink!30,color=pink!30] (-3,-0.2) rectangle (7,0);
\filldraw[fill=pink!30,color=pink!30] (-3.2,0) rectangle (7.2,0.2);
\foreach \x in {1,2,3,4}
\filldraw[fill=white,thick] (-6+2*\x+\x*0.1+0.1,0) arc (0:180:\x*0.1+0.1);
\foreach \x in {1,2,3}
\filldraw[fill=white,thick] (10-2*\x+\x*0.1+0.1,0) arc (0:180:\x*0.1+0.1);
\filldraw[fill=white,thick] (-4+0.15,0) arc (0:-180:0.15);
\filldraw[fill=white,thick] (-2+0.4,0) arc (0:-180:0.4);
\filldraw[fill=white,thick] (0+0.3,0) arc (0:-180:0.3);
\filldraw[fill=white,thick] (2+0.4,0) arc (0:-180:0.4);
\filldraw[fill=white,thick] (4+0.2,0) arc (0:-180:0.2);
\filldraw[fill=white,thick] (6+0.5,0) arc (0:-180:0.5);
\filldraw[fill=white,thick] (8+0.3,0) arc (0:-180:0.3);
\foreach \x in {-4,-2,0,2,4,6,8}
\draw node at (\x,0) {$\bullet$};
\draw (-4.5,0) -- (8.5,0);
\node[below,font=\tiny] at (1,0) {$p_{m,-1}$};
\node[font=\tiny] at (1,0) {$\bullet$};
\node[below,font=\tiny] at (3,0) {$p_{m,0}$};
\node[font=\tiny] at (3,0) {$\bullet$};
\node[below,font=\tiny] at (5,0) {$p_{m,1}$};
\node[font=\tiny] at (5,0) {$\bullet$};
\node[below,font=\tiny] at (7,0) {$p_{m,2}$};
\node[font=\tiny] at (7,0) {$\bullet$};
\node[below,font=\tiny] at (-1,0) {$p_{m,-2}$};
\node[font=\tiny] at (-1,0) {$\bullet$};
\node[below,font=\tiny] at (-3,0) {$p_{m,-3}$};
\node[font=\tiny] at (-3,0) {$\bullet$};
\end{tikzpicture}
\caption{The projection of $\Sigma_{\epsilon,M}^+\cap\Sigma_{\epsilon,M}^{-}$ on the $m$-sheet in the $z_{m,+}$ coordinates.}\label{fig:intersection}
\end{figure} 
%
\end{proof}
Let $M,\epsilon>0$ and $\theta\in(0,2\pi)$, we define
the homology 
\be
H_\bullet(X_{\epsilon,M},D_{\infty,\theta}):= H_{\bullet}\left(\varprojlim_{L\ge M} C_\bullet(X_{\epsilon,M},D_{L,\theta})\right)\,.
\ee 
For every $L>L'\geq M>0$, the homotopy equivalence between $D_{L,\theta}$ and $D_{L',\theta}$ induces an isomorphism of chain complex $C_\bullet(X_{\epsilon,M},D_{L,\theta})\to C_\bullet(X_{\epsilon,M},D_{L',\theta})$. Hence, there exists a long exact sequence 
\be\label{eq:exact_sequence}
\ldots\to H_{n}(D_{L,\theta})\to H_{n}(X_{\epsilon,M})\to H_n(X_{\epsilon,M},D_{\infty,\theta})\to\ldots 
\ee
for $L$ large enough. From Proposition~\ref{prop:relative-homology}, we deduce that $H_1(X_{\epsilon,M},D_{\infty,\theta})$ is finite dimensional too.
By a similar argument, we can define the homology of $\Sigma$ relative to $D_{\infty,\theta}$: notice that for every $M'\geq M>0$ and $\epsilon\geq\epsilon'>0$ the homotopy equivalence between $X_{\epsilon,M}$ and $X_{\epsilon',M'}$ induces an isomorphism of chain complexes $C_\bullet(X_{\epsilon,M},D_{\infty,\theta})\to C_\bullet(X_{\epsilon',M'},D_{\infty,\theta})$, thus 
\be\label{eq:full_H}
H_\bullet(\Sigma,D_{\infty,\theta}):=H_\bullet\left(\varprojlim_{M,\epsilon} C_\bullet(X_{\epsilon,M},D_{\infty,\theta})\right)
\ee 
defines the infinite dimensional homology group of $\Sigma$ relative to $D_{\infty,\theta}$.  
\subsection{A state integral basis}\label{sec:state_int_basis}
Let $\tilde{\calJ}_{\ell,m,\tau}$ be the lift of $\calJ_{\ell,\tau}$ to the $m$-th sheet of $\Sigma$ in the $z_+$ coordinates. The following proposition shows that such classes generate the homology $H_1(\Sigma,D_{\infty,\theta})$ of Equation~\eqref{eq:full_H}.
\begin{proposition}\label{prop:base_homology}
Let $\epsilon, M>0$. A subset of the state integral contours $\tilde{\calJ}_{\ell,m,\tau}$ for $|\ell|<M$ and $m\in\{0,\dots,A-1\}$ define a basis for $H_1(X_{\epsilon,M},D_{\infty,\theta})$, where $\theta=\arg(-1/\tau)\in(0,2\pi)$. 
\end{proposition}
\begin{proof}
By construction we find that $[\tilde{\calJ}_{\ell,m,\tau}]\in H_1(X_{\epsilon,M}, D_{\infty,\theta})$, where $|\ell|\leq M$. Recall, as we show in the proof of Proposition~\ref{prop:relative-homology}, that for every $L\geq M$ there is an isomorphism 
$H_1(X_{\epsilon,M}\,,D_{L,\theta})\cong H_0(\Sigma_{\epsilon,M}^{+}\cap\Sigma_{\epsilon,M}^{-})$, which induces the isomorphism
\be\label{eq:iso_homology}
H_1(X_{\epsilon,M}\,,D_{\infty,\theta})\cong H_0(\Sigma_{\epsilon,M}^{+}\cap\Sigma_{\epsilon,M}^{-})
\ee
as deduced from Equation~\eqref{eq:exact_sequence}. 
Then, through the isomorphism in Equation~\eqref{eq:iso_homology}, the contours $[\tilde{\calJ}_{\ell,m,\tau}]$ can be identified with the class of points $[p_{\ell,m}]\in H_0(\Sigma_{\epsilon,M}^{+}\cap\Sigma_{\epsilon,M}^{-})$ for some $m=0,\ldots,A-1$ and $|\ell|<M$.   
\end{proof}
\subsubsection{The class of the thimble $\calC_{p_0,\vartheta}$}
We can now describe the class of the thimble $\mathcal{C}_{p_0,\vartheta}$.
\begin{lemma}\label{lem:class-thimble}
If $ e^{i\vartheta}\notin\mathcal{S}(V_0)$ and $\vartheta\in(\tfrac{\pi}{2},\tfrac{3\pi}{2})$, then there exists an $M_0>0$ such that for $M>M_0$ we have $\big[\mathcal{C}_{p_0,\vartheta}\big]\in H_1(X_{\epsilon,M}\,,D_{\infty,\theta})$, for all $\theta\in(\vartheta-\tfrac{\pi}{2},\vartheta+\tfrac{\pi}{2})$.
\end{lemma}
\begin{proof}
By Corollary~\ref{cor:thimble}, the thimble $\mathcal{C}_{p_0,\vartheta}$ has a smooth parametrization given by $\gamma\colon\BR\to\Sigma$. In addition, from Corollary~\ref{cor:thimble}, there exists $\epsilon, M>0$ such that $\mathcal{C}_{p_0,\vartheta}\subset\Sigma_\epsilon$ and for $|z|>M$ the tangent is approximated by Equation~\eqref{eq:tangent.lim}. Therefore using Equations~\eqref{eq:D.angles}, we see that
\be
  \arg(i\sqrt{ie^{i\vartheta}})\in\arg(D_{\infty,\theta}^{+})\,,
  \qquad\text{and}\qquad
 \arg\Big(-i\sqrt{ie^{i\vartheta}\frac{|(A-B)|}{A-B}}\Big)\in\arg(D_{\infty,\theta}^{-})\,,
\ee
for each $\theta\in(\vartheta-\tfrac{\pi}{2},\vartheta+\tfrac{\pi}{2})$. The proof by pictures is given in Figures~\ref{fig:good.region.I}~and~\ref{fig:good.region.II}.
\end{proof}
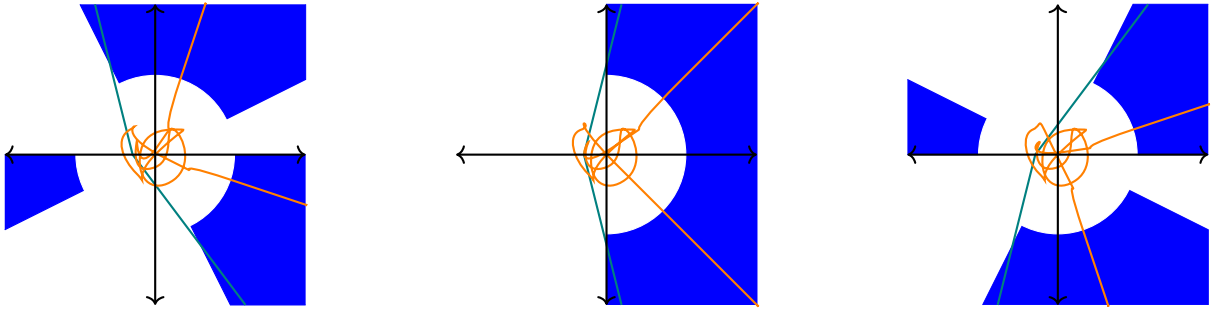
\begin{figure}[ht]
\begin{tikzpicture}
\filldraw[blue,opacity=0.2] (-2,-1) -- (0,0) -- (-2,0);
\filldraw[blue,opacity=0.2] (2,0) -- (0,0) -- (1,-2) -- (2,-2);
\filldraw[blue,opacity=0.2] (-1,2) -- (0,0) -- (2,1) -- (2,2);
\filldraw[white](0,0) circle (30pt);
\draw[teal,thick] (-1+0.2,2) -- (-0.3,0) -- (1+0.2,-2);
\draw[orange,thick] plot [smooth, tension=2] coordinates {(2,-2/3) (1,-1/3) (1/8,-1/24) (-1/3,1/3) (-1/4,-1/4) (0,0) (1/3,1/3) (-1/5,0) (1/4,-1/3) (1/4,1/3) (0,-1/6) (-1/4,1/7) (1/24,1/8) (1/3,1) (2/3,2)};
\draw[<->,thick] (-2,0) -- (2,0);
\draw[<->,thick] (0,-2) -- (0,2);
\filldraw[blue,opacity=0.2,xshift=6cm] (0,-2) -- (0,2) -- (2,2) -- (2,-2);
\filldraw[white,xshift=6cm](0,0) circle (30pt);
\draw[teal,thick,xshift=6cm] (0.2,2) -- (-0.3,0) -- (0.2,-2);
\draw[orange,thick,xshift=6cm] plot [smooth, tension=2] coordinates {(2,-2) (1,-1) (1/8,-1/8) (-1/3,1/3) (-1/4,-1/4) (0,0) (1/3,1/3) (-1/5,0) (1/4,-1/3) (1/4,1/3) (0,-1/6) (-1/4,1/7) (1/8,1/8) (1,1) (2,2)};
\draw[<->,thick,xshift=6cm] (-2,0) -- (2,0);
\draw[<->,thick,xshift=6cm] (0,-2) -- (0,2);
\filldraw[blue,opacity=0.2,xshift=12cm] (-1,-2) -- (0,0) -- (2,-1) -- (2,-2);
\filldraw[blue,opacity=0.2,xshift=12cm] (2,0) -- (0,0) -- (1,2) -- (2,2);
\filldraw[blue,opacity=0.2,xshift=12cm] (-2,1) -- (0,0) -- (-2,0);
\filldraw[white,xshift=12cm](0,0) circle (30pt);
\draw[teal,thick,xshift=12cm] (1+0.2,2) -- (-0.3,0) -- (-1+0.2,-2);
\draw[orange,thick,xshift=12cm] plot [smooth, tension=2] coordinates {(2/3,-2) (1/3,-1) (1/24,-1/8) (-1/3,1/3) (-1/4,-1/4) (0,0) (1/3,1/3) (-1/5,0) (1/4,-1/3) (1/4,1/3) (0,-1/6) (-1/4,1/7) (1/8,1/24) (1,1/3) (2,2/3)};
\draw[<->,thick,xshift=12cm] (-2,0) -- (2,0);
\draw[<->,thick,xshift=12cm] (0,-2) -- (0,2);
\end{tikzpicture}
\caption{For $-1/\tau= e^{i\vartheta}=e^{i\theta}$, $\Re(\tau)>0$ and $A-B<0$, these pictures illustrate from left to right the regions near infinity with exponential decay $D_{\infty,\theta}$ in \textcolor{blue}{blue} for $\Im(\tau)<0$, $\Im(\tau)=0$, and $\Im(\tau)>0$. A state integral contour $\tilde{\calJ}_{\ell,m,\tau}^{+}$ is depicted in \textcolor{teal}{teal} and the thimble $\mathcal{C}_{p_0,\vartheta}$ is depicted in \textcolor{orange}{orange}.}\label{fig:good.region.I}
\end{figure}
\begin{figure}[ht]
\begin{tikzpicture}
\filldraw[blue,opacity=0.2] (-1,2) -- (0,0) -- (2,1) -- (2,2);
\filldraw[blue,opacity=0.2] (0,0) -- (1,-2) -- (-2,-2)-- (-2,-1);
\filldraw[white](0,0) circle (30pt);
\draw[teal,thick] (-1+0.2,2) -- (-0.3,0) -- (1-0.4,-2);
\draw[orange,thick] plot [smooth, tension=2] coordinates {(-2/3,-2) (-1/3,-1) (-1/24,-1/8) (-1/3,1/3) (-1/4,-1/4) (0,0) (1/3,1/3) (-1/5,0) (1/4,-1/3) (1/4,1/3) (0,-1/6) (-1/4,1/7) (1/24,1/8) (1/3,1) (2/3,2)};
\draw[<->,thick] (-2,0) -- (2,0);
\draw[<->,thick] (0,-2) -- (0,2);
\filldraw[blue,opacity=0.2,xshift=6cm] (0,0) -- (0,2) -- (2,2) -- (2,0);
\filldraw[blue,opacity=0.2,xshift=6cm] (0,0) -- (0,-2) -- (-2,-2)-- (-2,0);
\filldraw[white,xshift=6cm](0,0) circle (30pt);
\draw[teal,thick,xshift=6cm] (0.2,2) -- (-0.3,0) -- (-0.4,-2);
\draw[orange,thick,xshift=6cm] plot [smooth, tension=2] coordinates {(-2,-2) (-1,-1) (-1/8,-1/8) (-1/3,1/3) (-1/4,-1/4) (0,0) (1/3,1/3) (-1/5,0) (1/4,-1/3) (1/4,1/3) (0,-1/6) (-1/4,1/7) (1/8,1/8) (1,1) (2,2)};
\draw[<->,thick,xshift=6cm] (-2,0) -- (2,0);
\draw[<->,thick,xshift=6cm] (0,-2) -- (0,2);
\filldraw[blue,opacity=0.2,xshift=12cm] (2,-1) -- (0,0) -- (1,2) -- (2,2);
\filldraw[blue,opacity=0.2,xshift=12cm] (-2,1) -- (0,0) -- (-1,-2) -- (-2,-2);
\filldraw[white,xshift=12cm](0,0) circle (30pt);
\draw[teal,thick,xshift=12cm] (1+0.2,2) -- (-0.3,0) -- (-1-0.4,-2);
\draw[orange,thick,xshift=12cm] plot [smooth, tension=2] coordinates {(-2,-2/3) (-1,-1/3) (-1/8,-1/24) (-1/3,1/3) (-1/4,-1/4) (0,0) (1/3,1/3) (-1/5,0) (1/4,-1/3) (1/4,1/3) (0,-1/6) (-1/4,1/7) (1/8,1/24) (1,1/3) (2,2/3)};
\draw[<->,thick,xshift=12cm] (-2,0) -- (2,0);
\draw[<->,thick,xshift=12cm] (0,-2) -- (0,2);
\end{tikzpicture}
\caption{This reproduces Figure~\ref{fig:good.region.I} with the change $A-B>0$.}\label{fig:good.region.II}
\end{figure}
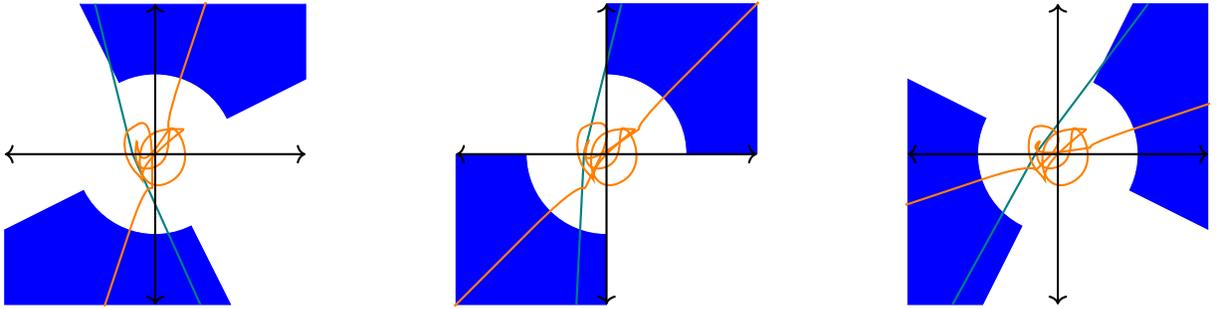
\noindent We see that Lemma~\ref{lem:class-thimble} and Proposition~\ref{prop:base_homology} imply the following corollary.
\begin{corollary}\label{cor:thimb.stateint}
If $e^{i\vartheta}\notin\mathcal{S}(V_0)$ and $\vartheta\in(\tfrac{\pi}{2},\tfrac{3\pi}{2})$, the thimble $[\mathcal{C}_{p_0,\vartheta}]$ can be written in terms of a finite collection of state integral contours $[\tilde{\calJ}_{i,m,\tau}]$ for all $\tau\in\BC\backslash\BR_{\leq0}$.
\end{corollary}
\subsubsection{A mix of state integrals and steepest descent contours.}
It turns out that we are interested in more classes than just the steepest descent contours and the state integral contours. Indeed, we want a combination of both. For a point $p\in\Sigma$ such that $V(p)\in V_0-i e^{i\vartheta}\BR_{>0}+\BZ$, we can define an element of the homology $H_1(\Sigma,D_{\infty,\theta})$ from Equation~\eqref{eq:full_H}. We assume that $e^{i\vartheta}\notin\mathcal{S}(V_0)$. For $V(p)=V_0-i e^{i\vartheta}\lambda+\BZ$, we define $\gamma_{p,\theta}$ as follows (see also Figure~\ref{fig:half.thimb.41}):
\begin{definition}\label{def:gamma_p}
For $V(p)=V_0-i e^{i\vartheta}\lambda+\BZ$, the contour $\gamma_{p,\theta}$ is the curve that contains the connected component of $p$ in the set $\{r\in\Sigma\;|\;V(r)\in V_0-i e^{i\vartheta}\BR_{\geq\lambda}+\BZ\}$ union the straight line in the direction $e^{i\theta/2-iA\epsilon}$ for $\Im(z_+(p))>0$ and $-e^{i\theta/2-i(A-B)\epsilon}\BR_{\leq0}$ for $\Im(z_+(p))<0$. 
\end{definition}
One can see that $\gamma_{p,\theta}$ is a straight line connected to $\Gamma_{p,\theta}$ from Definition~\ref{def:GGamma_p}. We call the straight line segment a \emph{tail}.
Similarly, associated to any connected segment $e\subseteq\{r\in\Sigma\;|\;V(r)\in V_0-i e^{i\vartheta}\BR_{\geq\lambda}+\BZ\}$ with boundary $\partial e=p_+\cup p_-$ we can define the contour $\gamma_{e,\theta}$ as follows (see also Figure~\ref{fig:edge.thimb.41}): 
\begin{definition}\label{def:gamma_e}
Let $\gamma_{e,\theta}$ be the union of $e$ and the straight lines in the direction $e^{i\theta/2-iA\epsilon}$ for $\Im(z_+(p_\pm))>0$ and $-e^{i\theta/2-i(A-B)\epsilon}\BR_{\leq0}$ for $\Im(z_\pm(p))<0$.
\end{definition}
Again we call the straight line segments tails. Essentially by definition, these contours give elements of the homology $H_1(\Sigma,D_{\infty,\theta})$ from Equation~\eqref{eq:full_H}, which is completely analogous to the results of Corollary~\ref{cor:thimb.stateint} and Lemma~\ref{lem:class-thimble}. One can see that at the level of homology $[\gamma_{e,\theta}]=[\gamma_{p_+,\theta}]-[\gamma_{p_-,\theta}]$. Moreover, Lemma~\ref{lem:class-thimble} and Proposition~\ref{prop:base_homology} imply the following corollary:
\begin{corollary}\label{cor:half.thimb.stateint}
If $e^{i\vartheta}\in\BC\backslash\mathcal{S}(V_{0})$, $\theta\in(\vartheta-\tfrac{\pi}{2},\vartheta+\tfrac{\pi}{2})$, $p\in\Sigma$ is such that $V(p)\in V_0-i e^{i\vartheta}\BR_{>0}+\BZ$, and $e\subseteq\Sigma$ is a segment with $V(e)\subseteq V_0-i e^{i\vartheta}\BR_{>0}+\BZ$, then the contours $[\gamma_{p,\theta}]$ and $[\gamma_{e,\theta}]$ can each be written in terms of a finite collection of state integral contours $[\tilde{\calJ}_{i,m,\tau}]$.
\end{corollary}
In fact, for most choices of $e$ and $p$ we find that $[\gamma_{p,\theta}]=0$ and $[\gamma_{e,\theta}]=0$. They can only give rise to non-trivial classes when either the flow from $p$ crosses the reals or similarly when $e$ intersects the reals. For example, in Figure~\ref{fig:half.thimb.41} we see that $[\gamma_{p_1,\theta}]=0$ and in Figure~\ref{fig:edge.thimb.41} we see that $[\gamma_{e_1,\theta}]=0$. More precisely, denoting the set $W_{\epsilon,M}:=\{p\in X_{\epsilon,M}\;\vert\; |z_\pm(p)|\leq M\}$, we have the following result:
\begin{corollary}\label{cor:half.thimb.stateint.zero}
If $e^{i\vartheta}\in\BC\backslash\mathcal{S}(V_{0})$, and $\theta\in(\vartheta-\tfrac{\pi}{2},\vartheta+\tfrac{\pi}{2})$, then there exists constants $M,\epsilon>0$ such that for any $p\in\Sigma_{\epsilon}\backslash W_{\epsilon,M}$ or $e\subseteq\Sigma_{\epsilon}\backslash W_{\epsilon,M}$ with $V(p),V(e)\subset V_0-i e^{i\vartheta}\BR_{>0}+\BZ$, we find that $[\gamma_{p,\theta}]=[\gamma_{e,\theta}]=0\in H_1(\Sigma,D_{\infty,\theta})$.
\end{corollary}
\begin{proof}
From Equation~\eqref{eq:tangent-outside-ball} in the proof of Lemma~\ref{lem:thimble}, there exists constants $\epsilon, M>0$ such that for $x\in\Sigma_{\epsilon}\backslash W_{\epsilon,M}$ with $\Re(z_\pm(x))\neq0$ and $V(x)\in V_0-i e^{i\vartheta}\BR_{>0}+\BZ$ the imaginary part of the tangent to $\Gamma_{p,\theta}$ (or $\Gamma_{p_e,\theta}$) always has the same sign as $\Im(z_\pm(p))$ (or $\Im(z_\pm(p_e))$). This implies that the curves $\gamma_{p,\theta}$ and $\gamma_{e,\theta}$ satisfying the assumptions would never cross the real line and therefore represent trivial classes.
\end{proof}
\begin{figure}[ht]
\begin{tikzpicture}[scale=2]
\draw[<->,thick] (-3.5,0) -- (2.2,0);
\draw[<->,thick] (0,-2) -- (0,0.5);
\draw[orange,thick,opacity=0.5] (-0.64333, -0.26982)--(-0.65333, -0.24799)--(-0.66333, -0.22758)--(-0.67333, -0.20840)--(-0.68333, -0.19032)--(-0.69333, -0.17322)--(-0.70333, -0.15699)--(-0.71333, -0.14154)--(-0.72333, -0.12681)--(-0.73333, -0.11273)--(-0.74333, -0.099259)--(-0.75333, -0.086346)--(-0.76333, -0.073958)--(-0.77333, -0.062065)--(-0.78333, -0.050641)--(-0.79333, -0.039667)--(-0.80333, -0.029127)--(-0.81333, -0.019008)--(-0.82333, -0.0093009)--(-0.83333, 0)--(-0.84333, 0.0088985)--(-0.85333, 0.017396)--(-0.86333, 0.025491)--(-0.87333, 0.033181)--(-0.88333, 0.040462)--(-0.89333, 0.047330)--(-0.90333, 0.053778)--(-0.91333, 0.059802)--(-0.92333, 0.065397)--(-0.93333, 0.070561)--(-0.94333, 0.075292)--(-0.95333, 0.079593)--(-0.96333, 0.083470)--(-0.97333, 0.086930)--(-0.98333, 0.089986)--(-0.99333, 0.092652)--(-1.0033, 0.094944)--(-1.0133, 0.096881)--(-1.0233, 0.098481)--(-1.0333, 0.099765)--(-1.0433, 0.10075)--(-1.0533, 0.10145)--(-1.0633, 0.10190)--(-1.0733, 0.10209)--(-1.0833, 0.10206)--(-1.0933, 0.10181)--(-1.1033, 0.10136)--(-1.1133, 0.10073)--(-1.1233, 0.099910)--(-1.1333, 0.098927)--(-1.1433, 0.097789)--(-1.1533, 0.096503)--(-1.1633, 0.095080)--(-1.1733, 0.093526)--(-1.1833, 0.091850)--(-1.1933, 0.090059)--(-1.2033, 0.088160)--(-1.2133, 0.086158)--(-1.2233, 0.084060)--(-1.2333, 0.081871)--(-1.2433, 0.079596)--(-1.2533, 0.077241)--(-1.2633, 0.074810)--(-1.2733, 0.072308)--(-1.2833, 0.069739)--(-1.2933, 0.067106)--(-1.3033, 0.064415)--(-1.3133, 0.061669)--(-1.3233, 0.058870)--(-1.3333, 0.056024)--(-1.3433, 0.053133)--(-1.3533, 0.050200)--(-1.3633, 0.047228)--(-1.3733, 0.044221)--(-1.3833, 0.041182)--(-1.3933, 0.038113)--(-1.4033, 0.035017)--(-1.4133, 0.031897)--(-1.4233, 0.028756)--(-1.4333, 0.025596)--(-1.4433, 0.022420)--(-1.4533, 0.019230)--(-1.4633, 0.016030)--(-1.4733, 0.012821)--(-1.4833, 0.0096063)--(-1.4933, 0.0063882)--(-1.5033, 0.0031691)--(-1.5233, -0.0032617)--(-1.5333, -0.0064685)--(-1.5433, -0.0096661)--(-1.5533, -0.012852)--(-1.5633, -0.016024)--(-1.5733, -0.019179)--(-1.5833, -0.022314)--(-1.5933, -0.025428)--(-1.6033, -0.028516)--(-1.6133, -0.031577)--(-1.6233, -0.034607)--(-1.6333, -0.037604)--(-1.6433, -0.040564)--(-1.6533, -0.043484)--(-1.6633, -0.046362)--(-1.6733, -0.049194)--(-1.6833, -0.051976)--(-1.6933, -0.054705)--(-1.7033, -0.057377)--(-1.7133, -0.059988)--(-1.7233, -0.062535)--(-1.7333, -0.065012)--(-1.7433, -0.067416)--(-1.7533, -0.069742)--(-1.7633, -0.071984)--(-1.7733, -0.074138)--(-1.7833, -0.076197)--(-1.7933, -0.078156)--(-1.8033, -0.080008)--(-1.8133, -0.081747)--(-1.8233, -0.083365)--(-1.8333, -0.084854)--(-1.8433, -0.086205)--(-1.8533, -0.087409)--(-1.8633, -0.088456)--(-1.8733, -0.089335)--(-1.8833, -0.090032)--(-1.8933, -0.090535);
\filldraw[red] (-5/6,0) circle (1pt);
\draw[<-,violet,thick] (-0.43333, -1.9326)--(-0.44333, -1.7687)--(-0.45333, -1.6119)--(-0.46333, -1.4632)--(-0.47333, -1.3237)--(-0.48333, -1.1942)--(-0.49333, -1.0754)--(-0.50333, -0.96772)--(-0.51333, -0.87121)--(-0.52333, -0.78545)--(-0.53333, -0.70972)--(-0.54333, -0.64306)--(-0.55333, -0.58439)--(-0.56333, -0.53264)--(-0.57333, -0.48678)--(-0.58333, -0.44592)--(-0.59333, -0.40928)--(-0.60333, -0.37619)--(-0.61333, -0.34612)--(-0.62333, -0.31860)--(-0.63333, -0.29327)--(-0.64333, -0.26982);
\filldraw[violet] (-0.64333, -0.26982) circle (1pt) node[right] {$p_1$};
\filldraw[violet] (-0.65, -1.7) node {$\gamma_{p_1,\theta}$};
\draw[violet,->,thick] (-0.64333, -0.26982)--(-0.64333-1.3460, -0.26982-1.4793);
\foreach \x in {-3,-2,-1,0,1,2}
\draw[ultra thick,red,xshift=\x cm] (-0.04,-0.04)--(0.04,0.04);
\foreach \x in {-3,-2,-1,0,1,2}
\draw[ultra thick,red,xshift=\x cm] (0.04,-0.04)--(-0.04,0.04);
\filldraw[magenta] (-1.8933, -0.090535) circle (1pt) node[below] {\;\;\;$p_2$};;
\filldraw[magenta] (-3.15, -1.2) node {$\gamma_{p_2,\theta}$};
\draw[magenta,->,thick] (-1.9033, -0.090829)--(-1.9133, -0.090896)--(-1.9233, -0.090717)--(-1.9333, -0.090270)--(-1.9433, -0.089532)--(-1.9533, -0.088475)--(-1.9633, -0.087065)--(-1.9733, -0.085264)--(-1.9833, -0.083029)--(-1.9933, -0.080306)--(-2.0033, -0.077028)--(-2.0133, -0.073115)--(-2.0233, -0.068459)--(-2.0333, -0.062912)--(-2.0433, -0.056253)--(-2.0533, -0.048106)--(-2.0633, -0.037701)--(-2.0733, -0.022747)--(-2.0793, -0.0057982)--(-2.0793, 0.016623)--(-2.0733, 0.033969)--(-2.0633, 0.049578)--(-2.0533, 0.060644)--(-2.0433, 0.069475)--(-2.0333, 0.076853)--(-2.0233, 0.083160)--(-2.0133, 0.088626)--(-2.0033, 0.093398)--(-1.9933, 0.097586)--(-1.9833, 0.10127)--(-1.9733, 0.10451)--(-1.9633, 0.10736)--(-1.9533, 0.10986)--(-1.9433, 0.11205)--(-1.9333, 0.11396)--(-1.9233, 0.11560)--(-1.9133, 0.11701)--(-1.9033, 0.11821)--(-1.8933, 0.11921)--(-1.8833, 0.12002)--(-1.8733, 0.12066)--(-1.8633, 0.12115)--(-1.8533, 0.12149)--(-1.8433, 0.12169)--(-1.8333, 0.12177)--(-1.8233, 0.12172)--(-1.8133, 0.12157)--(-1.8033, 0.12132)--(-1.7933, 0.12097)--(-1.7833, 0.12052)--(-1.7733, 0.12000)--(-1.7633, 0.11939)--(-1.7533, 0.11872)--(-1.7433, 0.11797)--(-1.7333, 0.11716)--(-1.7233, 0.11629)--(-1.7133, 0.11536)--(-1.7033, 0.11438)--(-1.6933, 0.11336)--(-1.6833, 0.11229)--(-1.6733, 0.11118)--(-1.6633, 0.11003)--(-1.6533, 0.10885)--(-1.6433, 0.10763)--(-1.6333, 0.10639)--(-1.6233, 0.10512)--(-1.6133, 0.10383)--(-1.6033, 0.10253)--(-1.5933, 0.10120)--(-1.5833, 0.099861)--(-1.5733, 0.098509)--(-1.5633, 0.097149)--(-1.5533, 0.095783)--(-1.5433, 0.094413)--(-1.5333, 0.093042)--(-1.5233, 0.091672)--(-1.5133, 0.090306)--(-1.5033, 0.088946)--(-1.4933, 0.087595)--(-1.4833, 0.086255)--(-1.4733, 0.084929)--(-1.4633, 0.083620)--(-1.4533, 0.082329)--(-1.4433, 0.081060)--(-1.4333, 0.079815)--(-1.4233, 0.078597)--(-1.4133, 0.077409)--(-1.4033, 0.076253)--(-1.3933, 0.075132)--(-1.3833, 0.074049)--(-1.3733, 0.073007)--(-1.3633, 0.072009)--(-1.3533, 0.071058)--(-1.3433, 0.070157)--(-1.3333, 0.069311)--(-1.3233, 0.068521)--(-1.3133, 0.067792)--(-1.3033, 0.067127)--(-1.2933, 0.066530)--(-1.2833, 0.066005)--(-1.2733, 0.065556)--(-1.2633, 0.065187)--(-1.2533, 0.064903)--(-1.2433, 0.064707)--(-1.2333, 0.064605)--(-1.2233, 0.064601)--(-1.2133, 0.064699)--(-1.2033, 0.064906)--(-1.1933, 0.065226)--(-1.1833, 0.065663)--(-1.1733, 0.066223)--(-1.1633, 0.066910)--(-1.1533, 0.067729)--(-1.1433, 0.068683)--(-1.1333, 0.069775)--(-1.1233, 0.071008)--(-1.1133, 0.072381)--(-1.1033, 0.073893)--(-1.0933, 0.075540)--(-1.0833, 0.077314)--(-1.0733, 0.079204)--(-1.0633, 0.081195)--(-1.0533, 0.083267)--(-1.0433, 0.085395)--(-1.0333, 0.087555)--(-1.0233, 0.089719)--(-1.0133, 0.091859)--(-1.0033, 0.093951)--(-0.99333, 0.095976)--(-0.98333, 0.097917)--(-0.97333, 0.099764)--(-0.96333, 0.10151)--(-0.95333, 0.10315)--(-0.94333, 0.10469)--(-0.93333, 0.10612)--(-0.92333, 0.10745)--(-0.91333, 0.10868)--(-0.90333, 0.10982)--(-0.89333, 0.11087)--(-0.88333, 0.11183)--(-0.87333, 0.11271)--(-0.86333, 0.11351)--(-0.85333, 0.11424)--(-0.84333, 0.11490)--(-0.83333, 0.11549)--(-0.82333, 0.11602)--(-0.81333, 0.11650)--(-0.80333, 0.11691)--(-0.79333, 0.11728)--(-0.78333, 0.11759)--(-0.77333, 0.11786)--(-0.76333, 0.11808)--(-0.75333, 0.11826)--(-0.74333, 0.11840)--(-0.73333, 0.11850)--(-0.72333, 0.11857)--(-0.71333, 0.11861)--(-0.70333, 0.11861)--(-0.69333, 0.11859)--(-0.68333, 0.11854)--(-0.67333, 0.11846)--(-0.66333, 0.11836)--(-0.65333, 0.11824)--(-0.64333, 0.11811)--(-0.63333, 0.11795)--(-0.62333, 0.11777)--(-0.61333, 0.11758)--(-0.60333, 0.11738)--(-0.59333, 0.11717)--(-0.58333, 0.11694)--(-0.57333, 0.11671)--(-0.56333, 0.11646)--(-0.55333, 0.11622)--(-0.54333, 0.11596)--(-0.53333, 0.11571)--(-0.52333, 0.11545)--(-0.51333, 0.11519)--(-0.50333, 0.11493)--(-0.49333, 0.11468)--(-0.48333, 0.11443)--(-0.47333, 0.11418)--(-0.46333, 0.11394)--(-0.45333, 0.11370)--(-0.44333, 0.11348)--(-0.43333, 0.11326)--(-0.42333, 0.11306)--(-0.41333, 0.11287)--(-0.40333, 0.11269)--(-0.39333, 0.11253)--(-0.38333, 0.11238)--(-0.37333, 0.11226)--(-0.36333, 0.11215)--(-0.35333, 0.11207)--(-0.34333, 0.11200)--(-0.33333, 0.11197)--(-0.32333, 0.11195)--(-0.31333, 0.11197)--(-0.30333, 0.11201)--(-0.29333, 0.11209)--(-0.28333, 0.11219)--(-0.27333, 0.11233)--(-0.26333, 0.11251)--(-0.25333, 0.11272)--(-0.24333, 0.11297)--(-0.23333, 0.11326)--(-0.22333, 0.11359)--(-0.21333, 0.11397)--(-0.20333, 0.11439)--(-0.19333, 0.11485)--(-0.18333, 0.11536)--(-0.17333, 0.11592)--(-0.16333, 0.11653)--(-0.15333, 0.11719)--(-0.14333, 0.11791)--(-0.13333, 0.11867)--(-0.12333, 0.11948)--(-0.11333, 0.12034)--(-0.10333, 0.12125)--(-0.093333, 0.12221)--(-0.083333, 0.12321)--(-0.073333, 0.12425)--(-0.063333, 0.12532)--(-0.053333, 0.12643)--(-0.043333, 0.12756)--(-0.033333, 0.12870)--(-0.023333, 0.12986)--(-0.013333, 0.13101)--(-0.0033333, 0.13217)--(0.0066667, 0.13331)--(0.016667, 0.13443)--(0.026667, 0.13553)--(0.036667, 0.13660)--(0.046667, 0.13764)--(0.056667, 0.13865)--(0.066667, 0.13962)--(0.076667, 0.14055)--(0.086667, 0.14144)--(0.096667, 0.14229)--(0.10667, 0.14310)--(0.11667, 0.14387)--(0.12667, 0.14461)--(0.13667, 0.14530)--(0.14667, 0.14596)--(0.15667, 0.14659)--(0.16667, 0.14718)--(0.17667, 0.14773)--(0.18667, 0.14826)--(0.19667, 0.14875)--(0.20667, 0.14921)--(0.21667, 0.14965)--(0.22667, 0.15006)--(0.23667, 0.15044)--(0.24667, 0.15080)--(0.25667, 0.15114)--(0.26667, 0.15145)--(0.27667, 0.15174)--(0.28667, 0.15202)--(0.29667, 0.15227)--(0.30667, 0.15251)--(0.31667, 0.15273)--(0.32667, 0.15293)--(0.33667, 0.15312)--(0.34667, 0.15330)--(0.35667, 0.15347)--(0.36667, 0.15362)--(0.37667, 0.15376)--(0.38667, 0.15390)--(0.39667, 0.15402)--(0.40667, 0.15414)--(0.41667, 0.15425)--(0.42667, 0.15436)--(0.43667, 0.15445)--(0.44667, 0.15455)--(0.45667, 0.15464)--(0.46667, 0.15473)--(0.47667, 0.15482)--(0.48667, 0.15491)--(0.49667, 0.15499)--(0.50667, 0.15508)--(0.51667, 0.15517)--(0.52667, 0.15527)--(0.53667, 0.15536)--(0.54667, 0.15546)--(0.55667, 0.15557)--(0.56667, 0.15568)--(0.57667, 0.15579)--(0.58667, 0.15592)--(0.59667, 0.15605)--(0.60667, 0.15619)--(0.61667, 0.15634)--(0.62667, 0.15651)--(0.63667, 0.15668)--(0.64667, 0.15687)--(0.65667, 0.15707)--(0.66667, 0.15728)--(0.67667, 0.15751)--(0.68667, 0.15775)--(0.69667, 0.15801)--(0.70667, 0.15829)--(0.71667, 0.15858)--(0.72667, 0.15890)--(0.73667, 0.15923)--(0.74667, 0.15958)--(0.75667, 0.15996)--(0.76667, 0.16035)--(0.77667, 0.16077)--(0.78667, 0.16121)--(0.79667, 0.16167)--(0.80667, 0.16216)--(0.81667, 0.16267)--(0.82667, 0.16321)--(0.83667, 0.16377)--(0.84667, 0.16435)--(0.85667, 0.16496)--(0.86667, 0.16559)--(0.87667, 0.16625)--(0.88667, 0.16693)--(0.89667, 0.16763)--(0.90667, 0.16835)--(0.91667, 0.16909)--(0.92667, 0.16985)--(0.93667, 0.17062)--(0.94667, 0.17140)--(0.95667, 0.17220)--(0.96667, 0.17300)--(0.97667, 0.17381)--(0.98667, 0.17462)--(0.99667, 0.17543)--(1.0067, 0.17624)--(1.0167, 0.17704)--(1.0267, 0.17784)--(1.0367, 0.17862)--(1.0467, 0.17939)--(1.0567, 0.18015)--(1.0667, 0.18089)--(1.0767, 0.18162)--(1.0867, 0.18233)--(1.0967, 0.18302)--(1.1067, 0.18369)--(1.1167, 0.18434)--(1.1267, 0.18498)--(1.1367, 0.18559)--(1.1467, 0.18618)--(1.1567, 0.18676)--(1.1667, 0.18731)--(1.1767, 0.18785)--(1.1867, 0.18837)--(1.1967, 0.18887)--(1.2067, 0.18936)--(1.2167, 0.18983)--(1.2267, 0.19028)--(1.2367, 0.19072)--(1.2467, 0.19114)--(1.2567, 0.19154)--(1.2667, 0.19194)--(1.2767, 0.19232)--(1.2867, 0.19269)--(1.2967, 0.19304)--(1.3067, 0.19339)--(1.3167, 0.19372)--(1.3267, 0.19405)--(1.3367, 0.19436)--(1.3467, 0.19467)--(1.3567, 0.19497)--(1.3667, 0.19526)--(1.3767, 0.19554)--(1.3867, 0.19582)--(1.3967, 0.19609)--(1.4067, 0.19636)--(1.4167, 0.19662)--(1.4267, 0.19688)--(1.4367, 0.19713)--(1.4467, 0.19739)--(1.4567, 0.19764)--(1.4667, 0.19788)--(1.4767, 0.19813)--(1.4867, 0.19838)--(1.4967, 0.19862)--(1.5067, 0.19887)--(1.5167, 0.19912)--(1.5267, 0.19936)--(1.5367, 0.19961)--(1.5467, 0.19987)--(1.5567, 0.20012)--(1.5667, 0.20038)--(1.5767, 0.20065)--(1.5867, 0.20092)--(1.5967, 0.20119)--(1.6067, 0.20147)--(1.6167, 0.20176)--(1.6267, 0.20205)--(1.6367, 0.20235)--(1.6467, 0.20266)--(1.6567, 0.20298)--(1.6667, 0.20330)--(1.6767, 0.20363)--(1.6867, 0.20398)--(1.6967, 0.20433)--(1.7067, 0.20469)--(1.7167, 0.20507)--(1.7267, 0.20546)--(1.7367, 0.20585)--(1.7467, 0.20626)--(1.7567, 0.20669)--(1.7667, 0.20712)--(1.7767, 0.20757)--(1.7867, 0.20803)--(1.7967, 0.20850)--(1.8067, 0.20899)--(1.8167, 0.20949)--(1.8267, 0.21001)--(1.8367, 0.21053)--(1.8467, 0.21107)--(1.8567, 0.21163)--(1.8667, 0.21219)--(1.8767, 0.21277)--(1.8867, 0.21336)--(1.8967, 0.21396)--(1.9067, 0.21457)--(1.9167, 0.21520)--(1.9267, 0.21583)--(1.9367, 0.21646)--(1.9467, 0.21711)--(1.9567, 0.21775)--(1.9667, 0.21841)--(1.9767, 0.21906)--(1.9867, 0.21972)--(1.9967, 0.22038)--(2.0067, 0.22103);
\draw[magenta,->,thick] (-1.8933, -0.090535)--(-1.8933-1.3460, -0.090535-1.4793);
\end{tikzpicture}
\caption{This figure depicts in \textcolor{orange}{orange} the set $\calC_{(-5/6,0),\vartheta}$ for $\vartheta=1.6650$ and $(A,B)=(1,2)$. The \textcolor{violet}{violet} contour represents the curve $\gamma_{p_1,\theta}$, which gives a trivial cycle, and the \textcolor{magenta}{magenta} contour represents the curve $\gamma_{p_2,\theta}$, which gives the class of a state integral contour. The \textcolor{red}{red} crosses are the branch points of the function $V$. The straight line segments are the tails.}
\label{fig:half.thimb.41}
\end{figure}
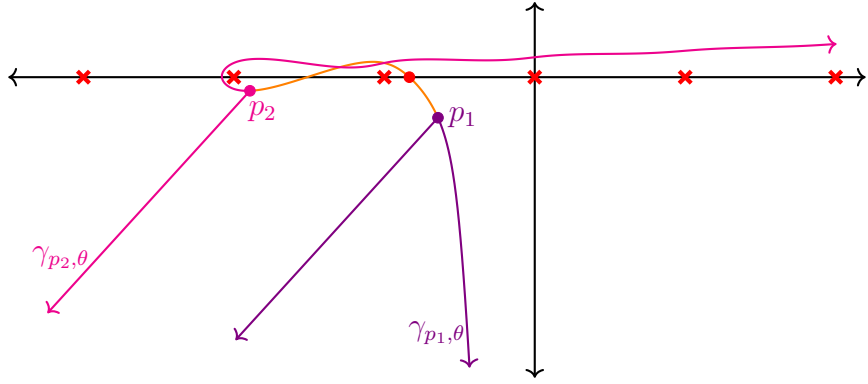
\begin{figure}[ht]
\begin{tikzpicture}[scale=2]
\draw[<->,thick] (-3.5,0) -- (2.2,0);
\draw[<->,thick] (0,-2) -- (0,0.5);
\draw[orange,thick,opacity=0.5]
(-0.64333, -0.26982)--(-0.65333, -0.24799)--(-0.66333, -0.22758)--(-0.67333, -0.20840)--(-0.68333, -0.19032)--(-0.69333, -0.17322)--(-0.70333, -0.15699)--(-0.71333, -0.14154)--(-0.72333, -0.12681)--(-0.73333, -0.11273)--(-0.74333, -0.099259)--(-0.75333, -0.086346)--(-0.76333, -0.073958)--(-0.77333, -0.062065)--(-0.78333, -0.050641)--(-0.79333, -0.039667)--(-0.80333, -0.029127)--(-0.81333, -0.019008)--(-0.82333, -0.0093009)--(-0.83333,0)--(-0.84333, 0.0088985)--(-0.85333, 0.017396)--(-0.86333, 0.025491)--(-0.87333, 0.033181)--(-0.88333, 0.040462)--(-0.89333, 0.047330)--(-0.90333, 0.053778)--(-0.91333, 0.059802)--(-0.92333, 0.065397);
\draw[orange,thick,opacity=0.5,<-] (-0.43333, -1.9326)--(-0.44333, -1.7687)--(-0.45333, -1.6119)--(-0.46333, -1.4632)--(-0.47333, -1.3237)--(-0.48333, -1.1942)--(-0.49333, -1.0754)--(-0.50333, -0.96772);
\draw[orange,thick,opacity=0.5,->] (-1.8933, -0.090535)--(-1.9033, -0.090829)--(-1.9133, -0.090896)--(-1.9233, -0.090717)--(-1.9333, -0.090270)--(-1.9433, -0.089532)--(-1.9533, -0.088475)--(-1.9633, -0.087065)--(-1.9733, -0.085264)--(-1.9833, -0.083029)--(-1.9933, -0.080306)--(-2.0033, -0.077028)--(-2.0133, -0.073115)--(-2.0233, -0.068459)--(-2.0333, -0.062912)--(-2.0433, -0.056253)--(-2.0533, -0.048106)--(-2.0633, -0.037701)--(-2.0733, -0.022747)--(-2.0793, -0.0057982)--(-2.0793, 0.016623)--(-2.0733, 0.033969)--(-2.0633, 0.049578)--(-2.0533, 0.060644)--(-2.0433, 0.069475)--(-2.0333, 0.076853)--(-2.0233, 0.083160)--(-2.0133, 0.088626)--(-2.0033, 0.093398)--(-1.9933, 0.097586)--(-1.9833, 0.10127)--(-1.9733, 0.10451)--(-1.9633, 0.10736)--(-1.9533, 0.10986)--(-1.9433, 0.11205)--(-1.9333, 0.11396)--(-1.9233, 0.11560)--(-1.9133, 0.11701)--(-1.9033, 0.11821)--(-1.8933, 0.11921)--(-1.8833, 0.12002)--(-1.8733, 0.12066)--(-1.8633, 0.12115)--(-1.8533, 0.12149)--(-1.8433, 0.12169)--(-1.8333, 0.12177)--(-1.8233, 0.12172)--(-1.8133, 0.12157)--(-1.8033, 0.12132)--(-1.7933, 0.12097)--(-1.7833, 0.12052)--(-1.7733, 0.12000)--(-1.7633, 0.11939)--(-1.7533, 0.11872)--(-1.7433, 0.11797)--(-1.7333, 0.11716)--(-1.7233, 0.11629)--(-1.7133, 0.11536)--(-1.7033, 0.11438)--(-1.6933, 0.11336)--(-1.6833, 0.11229)--(-1.6733, 0.11118)--(-1.6633, 0.11003)--(-1.6533, 0.10885)--(-1.6433, 0.10763)--(-1.6333, 0.10639)--(-1.6233, 0.10512)--(-1.6133, 0.10383)--(-1.6033, 0.10253)--(-1.5933, 0.10120)--(-1.5833, 0.099861)--(-1.5733, 0.098509)--(-1.5633, 0.097149)--(-1.5533, 0.095783)--(-1.5433, 0.094413)--(-1.5333, 0.093042)--(-1.5233, 0.091672)--(-1.5133, 0.090306)--(-1.5033, 0.088946)--(-1.4933, 0.087595)--(-1.4833, 0.086255)--(-1.4733, 0.084929)--(-1.4633, 0.083620)--(-1.4533, 0.082329)--(-1.4433, 0.081060)--(-1.4333, 0.079815)--(-1.4233, 0.078597)--(-1.4133, 0.077409)--(-1.4033, 0.076253)--(-1.3933, 0.075132)--(-1.3833, 0.074049)--(-1.3733, 0.073007)--(-1.3633, 0.072009)--(-1.3533, 0.071058)--(-1.3433, 0.070157)--(-1.3333, 0.069311)--(-1.3233, 0.068521)--(-1.3133, 0.067792)--(-1.3033, 0.067127)--(-1.2933, 0.066530)--(-1.2833, 0.066005)--(-1.2733, 0.065556)--(-1.2633, 0.065187)--(-1.2533, 0.064903)--(-1.2433, 0.064707)--(-1.2333, 0.064605)--(-1.2233, 0.064601)--(-1.2133, 0.064699)--(-1.2033, 0.064906)--(-1.1933, 0.065226)--(-1.1833, 0.065663)--(-1.1733, 0.066223)--(-1.1633, 0.066910)--(-1.1533, 0.067729)--(-1.1433, 0.068683)--(-1.1333, 0.069775)--(-1.1233, 0.071008)--(-1.1133, 0.072381)--(-1.1033, 0.073893)--(-1.0933, 0.075540)--(-1.0833, 0.077314)--(-1.0733, 0.079204)--(-1.0633, 0.081195)--(-1.0533, 0.083267)--(-1.0433, 0.085395)--(-1.0333, 0.087555)--(-1.0233, 0.089719)--(-1.0133, 0.091859)--(-1.0033, 0.093951)--(-0.99333, 0.095976)--(-0.98333, 0.097917)--(-0.97333, 0.099764)--(-0.96333, 0.10151)--(-0.95333, 0.10315)--(-0.94333, 0.10469)--(-0.93333, 0.10612)--(-0.92333, 0.10745)--(-0.91333, 0.10868)--(-0.90333, 0.10982)--(-0.89333, 0.11087)--(-0.88333, 0.11183)--(-0.87333, 0.11271)--(-0.86333, 0.11351)--(-0.85333, 0.11424)--(-0.84333, 0.11490)--(-0.83333, 0.11549)--(-0.82333, 0.11602)--(-0.81333, 0.11650)--(-0.80333, 0.11691)--(-0.79333, 0.11728)--(-0.78333, 0.11759)--(-0.77333, 0.11786)--(-0.76333, 0.11808)--(-0.75333, 0.11826)--(-0.74333, 0.11840)--(-0.73333, 0.11850)--(-0.72333, 0.11857)--(-0.71333, 0.11861)--(-0.70333, 0.11861)--(-0.69333, 0.11859)--(-0.68333, 0.11854)--(-0.67333, 0.11846)--(-0.66333, 0.11836)--(-0.65333, 0.11824)--(-0.64333, 0.11811)--(-0.63333, 0.11795)--(-0.62333, 0.11777)--(-0.61333, 0.11758)--(-0.60333, 0.11738)--(-0.59333, 0.11717)--(-0.58333, 0.11694)--(-0.57333, 0.11671)--(-0.56333, 0.11646)--(-0.55333, 0.11622)--(-0.54333, 0.11596)--(-0.53333, 0.11571)--(-0.52333, 0.11545)--(-0.51333, 0.11519)--(-0.50333, 0.11493)--(-0.49333, 0.11468)--(-0.48333, 0.11443)--(-0.47333, 0.11418)--(-0.46333, 0.11394)--(-0.45333, 0.11370)--(-0.44333, 0.11348)--(-0.43333, 0.11326)--(-0.42333, 0.11306)--(-0.41333, 0.11287)--(-0.40333, 0.11269)--(-0.39333, 0.11253)--(-0.38333, 0.11238)--(-0.37333, 0.11226)--(-0.36333, 0.11215)--(-0.35333, 0.11207)--(-0.34333, 0.11200)--(-0.33333, 0.11197)--(-0.32333, 0.11195)--(-0.31333, 0.11197)--(-0.30333, 0.11201)--(-0.29333, 0.11209)--(-0.28333, 0.11219)--(-0.27333, 0.11233)--(-0.26333, 0.11251)--(-0.25333, 0.11272)--(-0.24333, 0.11297)--(-0.23333, 0.11326)--(-0.22333, 0.11359)--(-0.21333, 0.11397)--(-0.20333, 0.11439)--(-0.19333, 0.11485)--(-0.18333, 0.11536)--(-0.17333, 0.11592)--(-0.16333, 0.11653)--(-0.15333, 0.11719)--(-0.14333, 0.11791)--(-0.13333, 0.11867)--(-0.12333, 0.11948)--(-0.11333, 0.12034)--(-0.10333, 0.12125)--(-0.093333, 0.12221)--(-0.083333, 0.12321)--(-0.073333, 0.12425)--(-0.063333, 0.12532)--(-0.053333, 0.12643)--(-0.043333, 0.12756)--(-0.033333, 0.12870)--(-0.023333, 0.12986)--(-0.013333, 0.13101)--(-0.0033333, 0.13217)--(0.0066667, 0.13331)--(0.016667, 0.13443)--(0.026667, 0.13553)--(0.036667, 0.13660)--(0.046667, 0.13764)--(0.056667, 0.13865)--(0.066667, 0.13962)--(0.076667, 0.14055)--(0.086667, 0.14144)--(0.096667, 0.14229)--(0.10667, 0.14310)--(0.11667, 0.14387)--(0.12667, 0.14461)--(0.13667, 0.14530)--(0.14667, 0.14596)--(0.15667, 0.14659)--(0.16667, 0.14718)--(0.17667, 0.14773)--(0.18667, 0.14826)--(0.19667, 0.14875)--(0.20667, 0.14921)--(0.21667, 0.14965)--(0.22667, 0.15006)--(0.23667, 0.15044)--(0.24667, 0.15080)--(0.25667, 0.15114)--(0.26667, 0.15145)--(0.27667, 0.15174)--(0.28667, 0.15202)--(0.29667, 0.15227)--(0.30667, 0.15251)--(0.31667, 0.15273)--(0.32667, 0.15293)--(0.33667, 0.15312)--(0.34667, 0.15330)--(0.35667, 0.15347)--(0.36667, 0.15362)--(0.37667, 0.15376)--(0.38667, 0.15390)--(0.39667, 0.15402)--(0.40667, 0.15414)--(0.41667, 0.15425)--(0.42667, 0.15436)--(0.43667, 0.15445)--(0.44667, 0.15455)--(0.45667, 0.15464)--(0.46667, 0.15473)--(0.47667, 0.15482)--(0.48667, 0.15491)--(0.49667, 0.15499)--(0.50667, 0.15508)--(0.51667, 0.15517)--(0.52667, 0.15527)--(0.53667, 0.15536)--(0.54667, 0.15546)--(0.55667, 0.15557)--(0.56667, 0.15568)--(0.57667, 0.15579)--(0.58667, 0.15592)--(0.59667, 0.15605)--(0.60667, 0.15619)--(0.61667, 0.15634)--(0.62667, 0.15651)--(0.63667, 0.15668)--(0.64667, 0.15687)--(0.65667, 0.15707)--(0.66667, 0.15728)--(0.67667, 0.15751)--(0.68667, 0.15775)--(0.69667, 0.15801)--(0.70667, 0.15829)--(0.71667, 0.15858)--(0.72667, 0.15890)--(0.73667, 0.15923)--(0.74667, 0.15958)--(0.75667, 0.15996)--(0.76667, 0.16035)--(0.77667, 0.16077)--(0.78667, 0.16121)--(0.79667, 0.16167)--(0.80667, 0.16216)--(0.81667, 0.16267)--(0.82667, 0.16321)--(0.83667, 0.16377)--(0.84667, 0.16435)--(0.85667, 0.16496)--(0.86667, 0.16559)--(0.87667, 0.16625)--(0.88667, 0.16693)--(0.89667, 0.16763)--(0.90667, 0.16835)--(0.91667, 0.16909)--(0.92667, 0.16985)--(0.93667, 0.17062)--(0.94667, 0.17140)--(0.95667, 0.17220)--(0.96667, 0.17300)--(0.97667, 0.17381)--(0.98667, 0.17462)--(0.99667, 0.17543)--(1.0067, 0.17624)--(1.0167, 0.17704)--(1.0267, 0.17784)--(1.0367, 0.17862)--(1.0467, 0.17939)--(1.0567, 0.18015)--(1.0667, 0.18089)--(1.0767, 0.18162)--(1.0867, 0.18233)--(1.0967, 0.18302)--(1.1067, 0.18369)--(1.1167, 0.18434)--(1.1267, 0.18498)--(1.1367, 0.18559)--(1.1467, 0.18618)--(1.1567, 0.18676)--(1.1667, 0.18731)--(1.1767, 0.18785)--(1.1867, 0.18837)--(1.1967, 0.18887)--(1.2067, 0.18936)--(1.2167, 0.18983)--(1.2267, 0.19028)--(1.2367, 0.19072)--(1.2467, 0.19114)--(1.2567, 0.19154)--(1.2667, 0.19194)--(1.2767, 0.19232)--(1.2867, 0.19269)--(1.2967, 0.19304)--(1.3067, 0.19339)--(1.3167, 0.19372)--(1.3267, 0.19405)--(1.3367, 0.19436)--(1.3467, 0.19467)--(1.3567, 0.19497)--(1.3667, 0.19526)--(1.3767, 0.19554)--(1.3867, 0.19582)--(1.3967, 0.19609)--(1.4067, 0.19636)--(1.4167, 0.19662)--(1.4267, 0.19688)--(1.4367, 0.19713)--(1.4467, 0.19739)--(1.4567, 0.19764)--(1.4667, 0.19788)--(1.4767, 0.19813)--(1.4867, 0.19838)--(1.4967, 0.19862)--(1.5067, 0.19887)--(1.5167, 0.19912)--(1.5267, 0.19936)--(1.5367, 0.19961)--(1.5467, 0.19987)--(1.5567, 0.20012)--(1.5667, 0.20038)--(1.5767, 0.20065)--(1.5867, 0.20092)--(1.5967, 0.20119)--(1.6067, 0.20147)--(1.6167, 0.20176)--(1.6267, 0.20205)--(1.6367, 0.20235)--(1.6467, 0.20266)--(1.6567, 0.20298)--(1.6667, 0.20330)--(1.6767, 0.20363)--(1.6867, 0.20398)--(1.6967, 0.20433)--(1.7067, 0.20469)--(1.7167, 0.20507)--(1.7267, 0.20546)--(1.7367, 0.20585)--(1.7467, 0.20626)--(1.7567, 0.20669)--(1.7667, 0.20712)--(1.7767, 0.20757)--(1.7867, 0.20803)--(1.7967, 0.20850)--(1.8067, 0.20899)--(1.8167, 0.20949)--(1.8267, 0.21001)--(1.8367, 0.21053)--(1.8467, 0.21107)--(1.8567, 0.21163)--(1.8667, 0.21219)--(1.8767, 0.21277)--(1.8867, 0.21336)--(1.8967, 0.21396)--(1.9067, 0.21457)--(1.9167, 0.21520)--(1.9267, 0.21583)--(1.9367, 0.21646)--(1.9467, 0.21711)--(1.9567, 0.21775)--(1.9667, 0.21841)--(1.9767, 0.21906)--(1.9867, 0.21972)--(1.9967, 0.22038)--(2.0067, 0.22103);
\filldraw[red] (-5/6,0) circle (1pt);
\draw[violet,ultra thick] (-0.50333, -0.96772)--(-0.51333, -0.87121)--(-0.52333, -0.78545)--(-0.53333, -0.70972)--(-0.54333, -0.64306)--(-0.55333, -0.58439)--(-0.56333, -0.53264)--(-0.57333, -0.48678)--(-0.58333, -0.44592)--(-0.59333, -0.40928)--(-0.60333, -0.37619)--(-0.61333, -0.34612)--(-0.62333, -0.31860)--(-0.63333, -0.29327)--(-0.64333, -0.26982);
\filldraw[violet] (-0.64333, -0.26982) circle (1pt);
\filldraw[violet] (-0.50333, -0.96772) circle (1pt);
\draw[violet] (-0.42, -0.58439) node {$e_{1}$};
\draw[violet] (-1.45, -1.7) node {$\gamma_{e_1,\theta}$};
\draw[violet,->,thick] (-0.64333, -0.26982)--(-0.64333-1.3460, -0.26982-1.4793);
\draw[violet,->,thick] (-0.50333, -0.96772)--(-0.50333-1.3460*0.55, -0.96772-1.4793*0.55);
\foreach \x in {-3,-2,-1,0,1,2}
\draw[ultra thick,red,xshift=\x cm] (-0.04,-0.04)--(0.04,0.04);
\foreach \x in {-3,-2,-1,0,1,2}
\draw[ultra thick,red,xshift=\x cm] (0.04,-0.04)--(-0.04,0.04);
\filldraw[magenta] (-0.93333, 0.070561) circle (1pt);
\filldraw[magenta] (-1.8933, -0.090535) circle (1pt);
\draw[magenta,ultra thick]
(-0.93333, 0.070561)--(-0.94333, 0.075292)--(-0.95333, 0.079593)--(-0.96333, 0.083470)--(-0.97333, 0.086930)--(-0.98333, 0.089986)--(-0.99333, 0.092652)--(-1.0033, 0.094944)--(-1.0133, 0.096881)--(-1.0233, 0.098481)--(-1.0333, 0.099765)--(-1.0433, 0.10075)--(-1.0533, 0.10145)--(-1.0633, 0.10190)--(-1.0733, 0.10209)--(-1.0833, 0.10206)--(-1.0933, 0.10181)--(-1.1033, 0.10136)--(-1.1133, 0.10073)--(-1.1233, 0.099910)--(-1.1333, 0.098927)--(-1.1433, 0.097789)--(-1.1533, 0.096503)--(-1.1633, 0.095080)--(-1.1733, 0.093526)--(-1.1833, 0.091850)--(-1.1933, 0.090059)--(-1.2033, 0.088160)--(-1.2133, 0.086158)--(-1.2233, 0.084060)--(-1.2333, 0.081871)--(-1.2433, 0.079596)--(-1.2533, 0.077241)--(-1.2633, 0.074810)--(-1.2733, 0.072308)--(-1.2833, 0.069739)--(-1.2933, 0.067106)--(-1.3033, 0.064415)--(-1.3133, 0.061669)--(-1.3233, 0.058870)--(-1.3333, 0.056024)--(-1.3433, 0.053133)--(-1.3533, 0.050200)--(-1.3633, 0.047228)--(-1.3733, 0.044221)--(-1.3833, 0.041182)--(-1.3933, 0.038113)--(-1.4033, 0.035017)--(-1.4133, 0.031897)--(-1.4233, 0.028756)--(-1.4333, 0.025596)--(-1.4433, 0.022420)--(-1.4533, 0.019230)--(-1.4633, 0.016030)--(-1.4733, 0.012821)--(-1.4833, 0.0096063)--(-1.4933, 0.0063882)--(-1.5033, 0.0031691)--(-1.5233, -0.0032617)--(-1.5333, -0.0064685)--(-1.5433, -0.0096661)--(-1.5533, -0.012852)--(-1.5633, -0.016024)--(-1.5733, -0.019179)--(-1.5833, -0.022314)--(-1.5933, -0.025428)--(-1.6033, -0.028516)--(-1.6133, -0.031577)--(-1.6233, -0.034607)--(-1.6333, -0.037604)--(-1.6433, -0.040564)--(-1.6533, -0.043484)--(-1.6633, -0.046362)--(-1.6733, -0.049194)--(-1.6833, -0.051976)--(-1.6933, -0.054705)--(-1.7033, -0.057377)--(-1.7133, -0.059988)--(-1.7233, -0.062535)--(-1.7333, -0.065012)--(-1.7433, -0.067416)--(-1.7533, -0.069742)--(-1.7633, -0.071984)--(-1.7733, -0.074138)--(-1.7833, -0.076197)--(-1.7933, -0.078156)--(-1.8033, -0.080008)--(-1.8133, -0.081747)--(-1.8233, -0.083365)--(-1.8333, -0.084854)--(-1.8433, -0.086205)--(-1.8533, -0.087409)--(-1.8633, -0.088456)--(-1.8733, -0.089335)--(-1.8833, -0.090032)--(-1.8933, -0.090535);
\draw[magenta,->,thick] (-1.8933, -0.090535)--(-1.8933-1.3460, -0.090535-1.4793);
\draw[magenta,->,thick] (-0.93333, 0.070561)--(-0.93333+1.3460*0.3, 0.070561+1.4793*0.3);
\draw[magenta] (-1.4, -0.1) node {$e_{2}$};
\filldraw[magenta] (-3.15, -1.2) node {$\gamma_{e_2,\theta}$};
\end{tikzpicture}
\caption{This figure depicts in \textcolor{orange}{orange} the set $\calC_{(-5/6,0),\vartheta}$ for $\vartheta=1.6650$ and $(A,B)=(1,2)$. The \textcolor{violet}{violet} contour represents the curve $\gamma_{e_1,\theta}$, which gives a trivial cycle, and the \textcolor{magenta}{magenta} contour represents the curve $\gamma_{e_2,\theta}$, which gives the class of a state integral contour. The \textcolor{red}{red} crosses are the branch points of the function~$V$. The straight line segments are the tails.}
\label{fig:edge.thimb.41}
\end{figure}
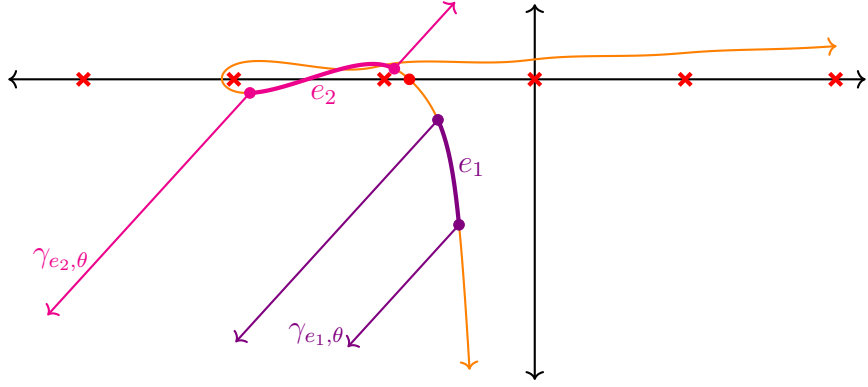
\section{Asymptotics of state integrals}\label{sec:asymp_state_int}
We will now describe an algorithm to construct a combination of state integrals for each $e^{i\vartheta}\in\BC\backslash\calS(V_0)$. This will be a locally constant assignment and have uniform asymptotics given by the saddle point approximation at a critical point $p_0$ for all $\arg(-1/\tau)=\theta\in[\vartheta-\tfrac{\pi}{2},\vartheta+\tfrac{\pi}{2}]$. When $\vartheta$ crosses the Stokes rays, the algorithm will produce different combinations of state integrals. This gives an explicit method that can be used to compute Stokes constants. We will use the algorithm to give a constructive proof of the main Theorem~\ref{thm:main_intro}.

\subsection{A steepest descent algorithm}\label{sec:alg}
Fix a critical point $p_0\in\Sigma$ with a lift of the critical value $V_0\in\BC$ and $\vartheta\in\arg(\BC\backslash\calS(V_{0}))\cap(\tfrac{\pi}{2},\tfrac{3\pi}{2})$. This determines the steepest descent contour~$\calC_{p_0,\vartheta}$ and constants $\epsilon,M>0$ so that all points $p\notin W_{\epsilon,M}$ satisfy $[\gamma_{p,\theta}]=0$ (see Corollary~\ref{cor:half.thimb.stateint.zero}). Let $W=W_{\epsilon,M}$ and consider the set
\be
\calZ\;:=\;\{p\in W\;|\;\Re(e^{-i\vartheta}(V_0-V(p)))\in\cos(2\pi\vartheta)\BZ\}\,.
\ee
We can identify $\calZ$ with $\big[e^{-i\vartheta}(V-V_0)\big]^{-1}(i\BR/e^{-i\vartheta}\BZ)\cap W$, thus $\calZ$ has a natural orientation induced by $\BR$. We define 
\be
\begin{aligned}
\calZ_+&\;:=\;\big[e^{-i\vartheta}(V_0-V)\big]^{-1}(i\BR_{\geq 0}/e^{-i\vartheta}\BZ)\cap W\subset\calZ\,,\\
\calP&\;:=\;\{p\in W\cap\calZ_+ \;\vert\; \Re(z(p)e^{-i\vartheta})\in\cos(2\pi\vartheta)\BZ\}\,.
\end{aligned}
\ee
Notice that $\calP$ is a finite set. To each point $p\in\calP\setminus\{p_0\}$ we can associate the cycle $\gamma_{p,\theta}$, which represents a class in $H_1(\Sigma,D_{\infty,\theta})$ as in Corollary~\ref{cor:half.thimb.stateint}. The points $p\in\calP$ partition the set $\calZ_+$ into a set of edges and we define $\calE$ to be the set of edges whose end points lie in $\calP$, i.e. 
\be
\calE\;:=\;\big\{e\in\pi_0(\calZ_+\setminus\calP)\colon \partial e\subseteq\calP\big\}\,.
\ee
We denote the boundary of $e$ by $(p_e, r_e)$, where the order respects the orientation of $\calZ_+$. To each $e\in\calE$ we associate a cycle $\gamma_{e,\theta}$, which represents a class in $H_1(\Sigma,D_{\infty,\theta})$ as shown in Corollary~\ref{cor:half.thimb.stateint}. 
In order to describe the algorithm we need the following lemma:
\begin{lemma}\label{lem:}
Let $p=(z,m)\in\Sigma$ and $e^{i\vartheta}\in\BC\setminus\calS(V_0)$. The sequence of tangent directions to $\Gamma_{(z,m+k),\vartheta}$ at the point $(z,m+k)$ is strictly monotone for $k\in\mathrm{sign}(\Im(z))\BZ_{\geq 0}$.
\end{lemma}
\begin{proof}
From Equation~\eqref{eq:tangent-GGamma_p} in the proof of Lemma~\ref{lem:thimble} and Equation~\eqref{eq:tangent}, the tangent to $\Gamma_{(z,m+k),\vartheta}$ at the point $(z,m+k)$ is $e^{i\vartheta}(iV'(z,m)+ik)^{-1}$. 
\end{proof}
To the data $(p_0,V_0,\vartheta)$ we can define a finite combination of state integrals. To begin the algorithm, in place of step $(\rm I)$, we consider the critical point $p_0$ and the lift $V_0\in\BC$ such that $V_0=V(p_0)\in\BC/\BZ$, and the thimble $\calC_{p_0,\vartheta}$ as in Corollary~\ref{cor:thimble}. Then we go on from step $(\rm II)$. Here is the algorithm:
\begin{itemize}
  \item[$(\rm I)$] Let $p\in\calP$ with a lift of $V(p)$ to $\BC$ and its associated cycle $\gamma_{p,\vartheta}$ (see Definition~\ref{def:gamma_p}).
  \item[$(\rm II)$] The cycle $\gamma_{p,\theta}$ can be decomposed as union of $\gamma_{e,\theta}$ for every $e\subset\gamma_{p,\theta}$ and $\gamma_{r,\theta}$ where $r\in\calP\cap\gamma_{p,\theta}$ is the last point with respect to the order inherited from $\gamma_{p,\theta}$. 
  \item[$(\rm III)$] To each $\gamma_{e,\theta}$ we associate the integral
  \be\label{eq:ge.si}
  \int_{\gamma_{e,\theta}} \e\big(\tfrac{A}{2}z(z+1+\tfrac{1}{\tau})\tau+mz\tau+n\tau\big)\Phi((z-\ell)\tau;\tau)^B\, dz\,,
  \ee 
  where $m$ and $n$ are dictated by the original lift of $V(p)$ and $\ell=\lceil \Re(z(x))+\tan(\vartheta)\Im(z(x))\rceil$ for some $x\in e$, the ceiling of the projection onto the reals parallel to~$e^{-i\vartheta}$. If $\Im(z(p_e))\Im(z(r_e))<0$ then $\gamma_{e,\theta}$ is homotopic to a state integral contour, hence the integral in Equation~\eqref{eq:ge.si} is equal to $q^n\calI_{m,\ell}$. Otherwise, it vanishes. 
  \item[$(\rm IV)$] To each point $(z,m)\in\partial\big(\calE\cap\gamma_{p,\theta}\big)$, we take the collection of points $(z,m\pm k)\in\calP$ where $k=1,\ldots, N$ and $N\leq |\calP|$, and the sign $\pm$ equals the sign of $\Im(z)$.  To these collections we assign the lifts $V((z,m\pm k))=V(z,m)\pm k(z-\ell)$ and take their associated cycles $-\gamma_{(z,m\pm k),\theta}$ weighted by constants $a_{k}\in\BZ$ defined as follows: from Lemma~\ref{lem:} there exists a constant $k_0$ such that 
\begin{itemize}
  \item[$\bullet$]  for $0<k<k_0$,  the tangent\footnote{If the tangent is equal to $\pm i e^{i\vartheta}$, then we consider the direction of higher derivatives.} at $(z,m\pm k)$ to the set $-\gamma_{(z,m\pm k),\theta}$ is in $-i e^{i\vartheta}\mathbb{H}$ if $\Im(z)>0$, and is in $i(A-B) e^{i\vartheta}\mathbb{H}$ if $\Im(z)<0$,
  \item[$\bullet$] for $k\geq k_0$, the tangent at $(z,m\pm k)$ to the set $-\gamma_{(z,m\pm k),\theta}$ is in $i e^{i\vartheta}\mathbb{H}$ if $\Im(z)>0$, and is in $-i(A-B) e^{i\vartheta}\mathbb{H}$ if $\Im(z)<0$,
\end{itemize}
and $a_k$ are the unique constants that satisfy the equation
\be\label{eq:alg-tails}
\begin{aligned}
& \big[\Phi((w-\ell)\tau;\tau)^B-\Phi((w-\ell-1)\tau;\tau)^B\big] \e\Big(\frac{A}{2}w(w+1+\tfrac{1}{\tau})\tau+mw\tau\Big)\\
&\=\sum_{0<k<k_0}a_k \Phi((w-\ell_0)\tau;\tau)^B\e\Big(\frac{A}{2}w(w+1+\tfrac{1}{\tau})\tau+(m\pm k)w\tau\Big)\\
&\quad + \sum_{k\geq k_0}a_k  \Phi((w-\ell_1)\tau;\tau)^B\e\Big(\frac{A}{2}w(w+1+\tfrac{1}{\tau})\tau+(m\pm k)w\tau\Big)\,,
\end{aligned}
\ee
for every $w$ with $\Im(w)=\Im(z)$, and where $\ell_0=\ell+1$ and $\ell_1=\ell$ if $\Im(z)>0$, and $\ell_0=\ell+(1-\mathrm{sign}(A-B))/2$ and $\ell_1=\ell+(1+\mathrm{sign}(A-B))/2$ if $\Im(z)<0$.
\end{itemize}
The algorithm terminates if at step $(\rm IV)$ there are no $(z,m\pm k)\in\calP$.
\begin{corollary}\label{cor:end-algorithm}
This algorithm terminates after finitely many steps producing a function $\calI_{p_0,V_0,\vartheta}(\tau)$ given as a $\BZ$-linear combination of $q^n\calI_{m,\ell}$.
\end{corollary}
\begin{proof}
We can partially order the points $p$ in the set $\calP\cap\BH$ according to the $\Re(z_{+}(p))$ and $m\in \{0,\ldots ,A-1\}$, in particular we use the lexicographic order on $\BR\times\{0,\ldots ,A-1\}$. The new base points appearing at step $(\mathrm{IV})$ of the algorithm always strictly increase with respect to this order. Hence after a finite number of steps all points will move outside $\calP$ (a finite set). A similar argument holds in the lower half-plane. Then from Corollary~\ref{cor:half.thimb.stateint.zero}, all of these points $p$ have trivial associated cycles $[\gamma_{p,\vartheta}]=0$. 
\end{proof}
\begin{lemma}
The function $\calI_{p_0,V_0,\vartheta}(\tau)$ is locally constant as $\vartheta$ varies in $\arg(\BC\backslash\calS(V_{0}))\cap(\tfrac{\pi}{2},\tfrac{3\pi}{2})$.
\end{lemma}
\begin{proof}
For $\vartheta\in\arg(\BC\backslash\calS(V_{0}))\cap(\tfrac{\pi}{2},\tfrac{3\pi}{2})$ contained in a connected component, the sets $\calZ_{+}$ are all homotopic. We can restrict to a compact subinterval and on this compact set of $\vartheta$, we can choose $\epsilon,M>0$ uniformly. Then for small enough variations of $\vartheta$, we can identify $p\in\calP$ while $\vartheta$ varies. Since $p$ can not cross the reals without going through a branch point, which $\calZ_+$ does not intersect, the cycles $\gamma_{p,\vartheta}$ are all homotopic. This implies that the algorithm will generically produce the same $\calI_{p_0,V_0,\vartheta}(\tau)$ for small variations of $\vartheta$. Besides the generic behaviour, there are critical directions across which the set $\calP$ can gain or loose points. Points in $\calP$ either appear and disappear in pairs in the interior of $W$---therefore they will both have the same sign of $\Im(z_\pm)$---or a single point falls outside the boundary of $W$.
If two points appear or disappear inside $W$, given that both points are in the same half plane, the cycle associated to them is trivial and this does not affect the function $\calI_{p_0,V_0,\vartheta}(\tau)$. In the other case, if a point $p\in\calP$ falls outside the boundary of $W$ then $[\gamma_{p,\vartheta}]=0$ by Corollary~\ref{cor:half.thimb.stateint.zero} and therefore, as $[\gamma_{p,\vartheta}]$ is constant for all $\vartheta$ in our compact set, it does not affect $\calI_{p_0,V_0,\vartheta}(\tau)$. 
\end{proof}
If $\rm I$ denotes a connected component of $\BC\backslash\calS(V_{0}))\cap(\tfrac{\pi}{2},\tfrac{3\pi}{2})$, then we denote $\calI_{p_0,V_0,I}(\tau)$ the function $\calI_{p_0,V_0,\vartheta}(\tau)$ for any $\vartheta\in\arg{I}$. 
\subsection{Proof of Theorem~\ref{thm:main_intro}}\label{sec:proof_main}
We now have all the ingredients we need to prove our main theorem. The main idea is to compute the asymptotics of the function $\calI_{p_0,V_0,\vartheta}(\tau)$ for $\arg(-1/\tau)=\theta\in(-\tfrac{\pi}{2}+\vartheta,\tfrac{\pi}{2}+\vartheta)$. We will state this in the following theorem, which is essentially equivalent to Theorem~\ref{thm:main_intro}.
\begin{theorem}
If $[a,b]\subseteq\arg(\BC\backslash\calS(V_{0}))\cap(\tfrac{\pi}{2},\tfrac{3\pi}{2})$, $\vartheta\in[a,b]$, $\arg(-1/\tau)=\theta\in(a-\tfrac{\pi}{2},b-\tfrac{\pi}{2})$ and $|\tau|\to\infty$, then there exists constants $C=C(a,b),\varepsilon=\varepsilon(a,b)>0$
\be
  |\e(-V_0\tau)
  \calI_{p_0,V_0,\vartheta}(\tau)
  -
  \Phi_{\Xi}^{(K)}(-2\pi i/\tau)|
  <
  C\,\varepsilon^{-K}\,K!\,|\tau|^{-K}\,,
\ee
where $\Phi_{\Xi}^{(K)}(\hbar)$ is the truncation of the series $\Phi_{\Xi}(\hbar)$ to $\mathrm{O}(\hbar^{K})$.
\end{theorem}
\begin{proof}
Firstly, as we vary $\arg(-1/\tau)=\theta\in(\vartheta-\tfrac{\pi}{2},\vartheta+\tfrac{\pi}{2})$, we can deform the contours of the state integrals of $\calI_{p_0,V_0,\vartheta}(\tau)$ so that they are contained in $\calZ_{+}$ besides tails based at the set $\{p\in X\cap\calZ_+ \colon \Im(z(p)e^{-i\theta})\in\sin(2\pi\theta)\BZ\}$. Considering this deformation we can now compute the asymptotics.

\medskip

The first iteration of the algorithm gives a collection of state integrals (from step III) whose asymptotics is determined by the saddle point approximation at $p_0$ with the addition of the asymptotics of the tails (from in step II). The tail at $(z,m)$ is 
\be\label{eq:tail}
\int_{z}^{\pm e^{i\theta}\infty} \big[\Phi((w-\ell)\tau;\tau)^B-\Phi((w-\ell-1)\tau;\tau)^B\big] \e\Big(\frac{A}{2}w(w+1+\tfrac{1}{\tau})\tau+mw\tau\Big)\,dw\,.
\ee
Step IV produces a new collection of cycles $\gamma_{(z,m\pm k)}$ (described in Equation~\eqref{eq:alg-tails}) whose asymptotics is equal to the asymptotics of the remaining tails. Therefore, at any iteration of the algorithm, the asymptotics are determined by the asymptotics of a finite collection of tails and the saddle point approximation at $p_0$.

\medskip

By construction, after a finite number of steps, the starting points of these tails will be outside $W$ (see Corollary~\ref{cor:end-algorithm}). Notice that the volumes of the tails' base points are always decreasing, and therefore eventually they leave the finite set $\calP$. If a tail starts outside $W$ its asymptotics are exponentially smaller than the asymptotics coming from the saddle point approximation at $p_0$. This can be seen from the proof of Corollary~\ref{cor:half.thimb.stateint.zero}.

\medskip

Therefore, the tails in $\calI_{p_0,V_0,\vartheta}(\tau)$ only make exponentially small contributions and the asymptotics are determined by the integral along the set $\calZ_+$. For $(\vartheta-\tfrac{\pi}{2},\vartheta+\tfrac{\pi}{2})$, the set $\calZ_+$ is contained in the region where $\Im(V(p)\tau)>0$ for the local lift of $V(p)\in\BC$. We therefore see that $\calI_{p_0,V_0,\vartheta}(\tau)$ has asymptotics determined by the stationary phase approximation at $p_0$ (see for example~\cite[Chap.~9, Sec.~2.2, Theorem~2.1]{Olver}). This was exactly how $\Phi_{\Xi}$ was defined in Equation~\eqref{eq:def.phi} and therefore this completes the proof.
\end{proof}
\begin{proof}[Proof of Theorem~\ref{thm:main_intro}]
To transfer this result to a proof of the main theorem, we simply follow the same proof as in Corollary~\ref{cor:fad.resum}. Given these asymptotics we can apply Nevanlinna's theorem~\cite{nevanlinna} or simply just take Mordell's inverse Laplace transform to define the analytic continuation of the Borel transform in the cone $\theta\in(a,b)$. This proves the Borel summability. We can run exactly the same construction and proof for $\Re(\tau)<0$ using the functions given in Remark~\ref{rem:neg.rt} to give the full result including $\Re(\tau)<0$.
\end{proof}
\begin{proof}[Proof of Theorem~\ref{thm:4152}]
The result follows immediately by specialising Theorem~\ref{thm:main_intro} to the cases of $(A,B)$ equal to $(1,2)$ for $4_1$ and $(2,3)$ for $5_2$. 
\end{proof}

\section{Examples}\label{sec:ex}

We now go though two explicit examples. One is associated to the figure eight knot $4_1$, where $(A,B)=(1,2)$ and the other is for $(A,B)=(4,1)$. The first was studied in~\cite{GGM:I} and full conjectural answers were given for the resurgent structures. We will prove the exact formulas given in~\cite{GGM:I} are correct by applying our method to compute the Borel--Laplace resummation. A similar analysis for the example $(A,B)=(4,1)$ was carried out in~\cite{Wh:thesis}. Here again conjectural formulas were given and we will also prove them.

\subsection{\texorpdfstring{The $4_1$ knot}{The 4\_1 knot}}\label{sec:4_1}
This example corresponds to $A=1$ and $B=2$. See for example~\cite[Sec.~7]{GSW}. We will compute the Borel--Laplace resummation of $\Phi_\Xi$ with the algorithm described in Section~\ref{sec:alg}. We will see that this exactly matches the numerical predictions of~\cite{GGM:I}. The critical points for the volume $V$ are at $p_1=(-5/6,0)$ and $p_2=(-1/6,0)$ and their critical values are respectively at $V_1=-0.0514175\ldots i$ and $V_2=0.0514175\ldots i$. Thus the first Stokes line is at $e^{i\vartheta}\BR_{>0}$ with $\vartheta=3.1416$ and it separates the regions $\mathrm{I}$ and $\mathrm{II}$. While the second one is at $\vartheta=1.6733$ and separates the regions $\mathrm{II}$ and $\mathrm{III}$ (see Figure~\ref{fig:41.stokes}).
\begin{figure}[ht]
\begin{tiny}
\begin{tikzpicture}[scale=2]
\draw[<->,thick] (-2,0) -- (2,0);
\draw[<->,thick] (0,-1) -- (0,1);
\foreach \x in {1,2,3,4,5,6,7,8,9,10,11,12,13,14,15,16,17,18,19,20,21,22,23,24,25}
\draw[red,->] (0,0)--(1/\x,1);
\foreach \x in {-3,-2,-1,0,1,2,3}
\draw[red,->,thick] (0,0)--(0+0.01*\x,1);
\foreach \x in {1,2,3,4,5,6,7,8,9,10,11,12,13,14,15,16,17,18,19,20,21,22,23,24,25}
\draw[red,->] (0,0)--(-1/\x,1);
\foreach \x in {1,2,3,4,5,6,7,8,9,10,11,12,13,14,15,16,17,18,19,20,21,22,23,24,25}
\draw[red,->] (0,0)--(1/\x,-1);
\foreach \x in {-3,-2,-1,0,1,2,3}
\draw[red,->,thick] (0,0)--(0+0.01*\x,-1);
\foreach \x in {1,2,3,4,5,6,7,8,9,10,11,12,13,14,15,16,17,18,19,20,21,22,23,24,25}
\draw[red,->] (0,0)--(-1/\x,-1);
\draw[red,->] (0,0)--(2,0);
\draw[red,->] (0,0)--(-2,0);
\draw(1,-0.4) node {$\mathrm{I}$};
\draw(1,0.4) node {$\mathrm{II}$};
\draw(0.6,0.83) node {$\mathrm{III}$};
\draw(-1,0.4) node {$\mathrm{X}$};
\draw[red] (2,0) node[right] {$\vartheta=\pi$};
\draw[red] (1,1) node[right] {$\vartheta=1.6733$};
\draw[red] (-1,1) node[left] {$\vartheta=1.4683$};
\draw[red] (-2,0) node[left] {$\vartheta=0$};
\end{tikzpicture}
\end{tiny}
\caption{This figure illustrates the Stokes rays in $\tau$-plane with $\vartheta=\arg(-1/\tau)$ forming a peacock pattern. Note that this is stretched in the direction of the $x$-axis and the Stokes rays are all even closer to the vertical line.}\label{fig:41.stokes}
\end{figure}
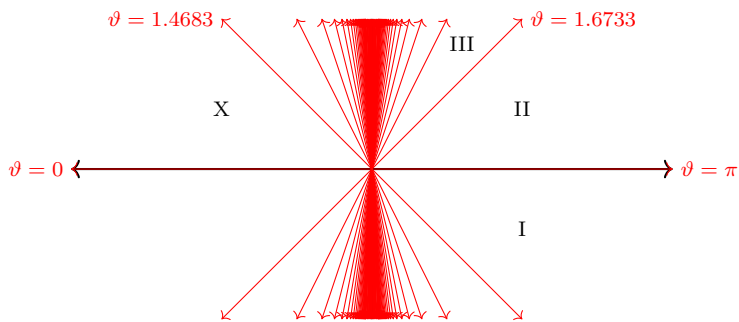


To perform the computation, we plot the set $\calC_{p_0,\vartheta}$ and keep track of the branching of the function $V$. In the plots, this will be noted at the end point with $(m,n)$ keeping track of the sheet of $\Sigma$ by $m\in\BZ$ and analytic continuation of the function $V$ with values $n\in\BZ$ so that for $\Im(z)>0$ we have $V(z)=B\tfrac{\Li_{2}(\e(z))}{(2\pi i)^2}+\tfrac{B}{24}+\tfrac{A}{2}z(z+1)+mz+n$ while for $\Im(z)<0$ we have $V(z)=-B\tfrac{\Li_{2}(\e(-z))}{(2\pi i)^2}+\tfrac{B}{12}-\tfrac{B}{2}(z-\tfrac{1}{2})^2+\tfrac{A}{2}z(z+1)+mz+n$. This allows us to read off the state integrals given by the algorithm.
\subsubsection{The critical point $p_2=(-1/6,0)$}
The critical point $p_2$ has no Stokes rays for $\vartheta\in(\tfrac{\pi}{2},\tfrac{3\pi}{2})$. Therefore, computing the Borel--Laplace resummation for a fixed $\vartheta\in(\tfrac{\pi}{2},\tfrac{3\pi}{2})$ will be enough to determine it for the whole region. We plot the set $\calC_{p_2,\vartheta}$ in Figure~\ref{fig:thimb.41.ak}.
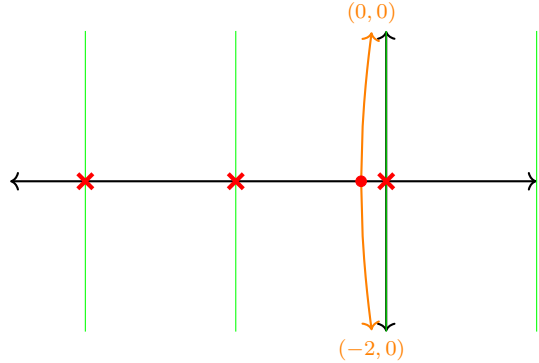
\begin{figure}[ht]
\begin{tiny}
\begin{tikzpicture}[scale=2]
\draw[<->,thick] (-2.5,0) -- (1,0);
\draw[<->,thick] (0,-1) -- (0,1);
\foreach \x in {-2,-1,0,1}
\draw[green] (\x,-1)--(\x,1);
\draw[ultra thick,red] (-0.05,-0.05)--(0.05,0.05);
\draw[ultra thick,red] (0.05,-0.05)--(-0.05,0.05);
\draw[ultra thick,red,xshift=-1cm] (-0.05,-0.05)--(0.05,0.05);
\draw[ultra thick,red,xshift=-1cm] (0.05,-0.05)--(-0.05,0.05);
\draw[ultra thick,red,xshift=-2cm] (-0.05,-0.05)--(0.05,0.05);
\draw[ultra thick,red,xshift=-2cm] (0.05,-0.05)--(-0.05,0.05);
\draw[orange,thick,->]
(-0.16667, 0)--(-0.15667, 0.36539)--(-0.14667, 0.51902)--(-0.13667, 0.63842)--(-0.12667, 0.74034)--(-0.11667, 0.83122)--(-0.10667, 0.91436)--(-0.096667, 0.99171) node[above] {$(0,0)$};
\draw[orange,thick,->]
(-0.16667,0)--(-0.15667, -0.36539)--(-0.14667, -0.51902)--(-0.13667, -0.63842)--(-0.12667, -0.74034)--(-0.11667, -0.83122)--(-0.10667, -0.91436)--(-0.096667, -0.99171) node[below] {$(-2,0)$};
\filldraw[red] (-1/6,0) circle (1pt);
\end{tikzpicture}
\end{tiny}
\caption{This figure depicts in \textcolor{orange}{orange} the set $\calC_{(-1/6,0),\vartheta}$ for $\vartheta=\arg(-1/\tau)=3.1416$ and $(A,B)=(1,2)$.}\label{fig:thimb.41.ak}
\end{figure}
We can see that this steepest descent contour is equivalent to the state integral contour $\calJ_{0,0}$. Therefore, the algorithm gives
\be\label{eq:Iarb.41}
  \calI_{p_2,V_2,\mathrm{Y}}
  \=
  \calI_{0,0}\,,
\ee
where for example $\mathrm{Y}=\mathrm{I},\mathrm{II},\mathrm{III}$.
This is the original Andersen--Kashaev invariant for $4_1$~\cite[Eq.~(38)]{AK:I} and~\cite[Sec.~1.3]{GK:qser}. Therefore, our results gives a refined version of their volume conjecture. In particular, not only does the Andersen--Kashaev invariant have exponential decay determined by the volume of the figure eight knot but it is in fact the resummation of its own all order asymptotics.

\subsubsection{State integrals decomposition in $\mathrm{I}$}
Let $\vartheta\in\mathrm{I}$. We plot the set $\calC_{p_1,\vartheta}$ in Figure~\ref{fig:thimb.41.I}.
\begin{figure}[ht]
\begin{tiny}
\begin{tikzpicture}[scale=2]
\draw[<->,thick] (-2.5,0) -- (1,0);
\draw[<->,thick] (0,-1) -- (0,1);
\foreach \x in {1}
\draw[green] (\x-0.80902/0.58779/4,0.58779*4/0.58779/4)--(1,-0.58779*4/0.80902+\x*0.58779*4/0.80902);
\foreach \x in {-2,-1,0}
\draw[green] (\x+1*0.80902/0.58779/4,-1)--(\x-0.80902/0.58779/4,0.58779*4/0.58779/4);
\draw[ultra thick,red] (-0.05,-0.05)--(0.05,0.05);
\draw[ultra thick,red] (0.05,-0.05)--(-0.05,0.05);
\draw[ultra thick,red,xshift=-1cm] (-0.05,-0.05)--(0.05,0.05);
\draw[ultra thick,red,xshift=-1cm] (0.05,-0.05)--(-0.05,0.05);
\draw[ultra thick,red,xshift=-2cm] (-0.05,-0.05)--(0.05,0.05);
\draw[ultra thick,red,xshift=-2cm] (0.05,-0.05)--(-0.05,0.05);
\draw[orange,thick,<->]
(-0.52333, 0.97696) node[above] {$(0,0)$}--(-0.53333, 0.92828)--(-0.54333, 0.88157)--(-0.55333, 0.83668)--(-0.56333, 0.79349)--(-0.57333, 0.75187)--(-0.58333, 0.71173)--(-0.59333, 0.67294)--(-0.60333, 0.63542)--(-0.61333, 0.59908)--(-0.62333, 0.56385)--(-0.63333, 0.52965)--(-0.64333, 0.49641)--(-0.65333, 0.46409)--(-0.66333, 0.43262)--(-0.67333, 0.40197)--(-0.68333, 0.37208)--(-0.69333, 0.34293)--(-0.70333, 0.31448)--(-0.71333, 0.28669)--(-0.72333, 0.25955)--(-0.73333, 0.23304)--(-0.74333, 0.20713)--(-0.75333, 0.18182)--(-0.76333, 0.15708)--(-0.77333, 0.13292)--(-0.78333, 0.10933)--(-0.79333, 0.086308)--(-0.80333, 0.063855)--(-0.81333, 0.041978)--(-0.82333, 0.020688)--(-0.83333, 0)--(-0.84333, -0.020067)--(-0.85333, -0.039487)--(-0.86333, -0.058230)--(-0.87333, -0.076258)--(-0.88333, -0.093524)--(-0.89333, -0.10997)--(-0.90333, -0.12554)--(-0.91333, -0.14014)--(-0.92333, -0.15369)--(-0.93333, -0.16607)--(-0.94333, -0.17715)--(-0.95333, -0.18679)--(-0.96333, -0.19482)--(-0.97333, -0.20108)--(-0.98333, -0.20537)--(-0.99333, -0.20756)--(-1.0033, -0.20750)--(-1.0133, -0.20513)--(-1.0233, -0.20043)--(-1.0333, -0.19344)--(-1.0433, -0.18423)--(-1.0533, -0.17291)--(-1.0633, -0.15958)--(-1.0733, -0.14436)--(-1.0833, -0.12735)--(-1.0933, -0.10866)--(-1.1033, -0.088382)--(-1.1133, -0.066606)--(-1.1233, -0.043408)--(-1.1333, -0.018860)--(-1.1433, 0.0069742)--(-1.1533, 0.034036)--(-1.1633, 0.062273)--(-1.1733, 0.091638)--(-1.1833, 0.12209)--(-1.1933, 0.15358)--(-1.2033, 0.18609)--(-1.2133, 0.21958)--(-1.2233, 0.25402)--(-1.2333, 0.28939)--(-1.2433, 0.32567)--(-1.2533, 0.36284)--(-1.2633, 0.40089)--(-1.2733, 0.43980)--(-1.2833, 0.47957)--(-1.2933, 0.52019)--(-1.3033, 0.56166)--(-1.3133, 0.60399)--(-1.3233, 0.64717)--(-1.3333, 0.69121)--(-1.3433, 0.73613)--(-1.3533, 0.78195)--(-1.3633, 0.82868)--(-1.3733, 0.87636)--(-1.3833, 0.92501)--(-1.3933, 0.97467) node[above] {$(2,2)$};
\filldraw[red] (-5/6,0) circle (1pt);
\filldraw[magenta] (-1+1*0.80902*0.074,-0.58779*4*0.074) circle (1pt);
\draw[magenta,thick,->] (-1+1*0.80902*0.074,-0.58779*4*0.074)--(-0.93013, -0.22358)--(-0.92013, -0.26718)--(-0.91013, -0.30635)--(-0.90013, -0.34206)--(-0.89013, -0.37495)--(-0.88013, -0.40554)--(-0.87013, -0.43418)--(-0.86013, -0.46116)--(-0.85013, -0.48670)--(-0.84013, -0.51099)--(-0.83013, -0.53417)--(-0.82013, -0.55636)--(-0.81013, -0.57768)--(-0.80013, -0.59822)--(-0.79013, -0.61804)--(-0.78013, -0.63722)--(-0.77013, -0.65582)--(-0.76013, -0.67388)--(-0.75013, -0.69146)--(-0.74013, -0.70860)--(-0.73013, -0.72533)--(-0.72013, -0.74170)--(-0.71013, -0.75772)--(-0.70013, -0.77344)--(-0.69013, -0.78887)--(-0.68013, -0.80404)--(-0.67013, -0.81897)--(-0.66013, -0.83369)--(-0.65013, -0.84821)--(-0.64013, -0.86255)--(-0.63013, -0.87673)--(-0.62013, -0.89076)--(-0.61013, -0.90466)--(-0.60013, -0.91845)--(-0.59013, -0.93213)--(-0.58013, -0.94572)--(-0.57013, -0.95923)--(-0.56013, -0.97267)--(-0.55013, -0.98605)--(-0.54013, -0.99939) node[below] {$(-3,-1)$};
\end{tikzpicture}
\end{tiny}
\caption{This figure depicts in \textcolor{orange}{orange} the set $\calC_{(-5/6,0),\vartheta}$ for $\vartheta=\arg(-1/\tau)=4.0841$, and $(A,B)=(1,2)$. The \textcolor{green}{green} lines are parallel to the lines $i/\tau$. The \textcolor{magenta}{magenta} curve represents the tail at the intersection of the set $\calC_{(-5/6,0),\vartheta}$ with the lines parallel to $i/\tau$.}
\label{fig:thimb.41.I}
\end{figure}
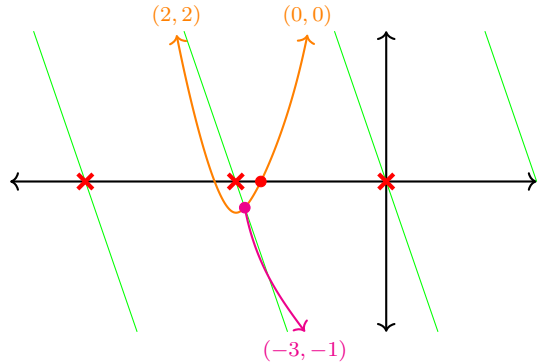
We can decompose the thimble in terms on state integrals contours. We have contributions from $\calI_{0,0}$ from the intersection of the orange contour with the interval $(-1,0)$, which has $(n,m,\ell)=(0,0,0)$ and $-q^2\calI_{2,-1}$ from the intersection of the orange contour with the interval $(-2,-1)$, which has $(n,m,\ell)=(2,2,-1)$. We can see that the tail leads to no contribution as it does not cross the reals and heads to $\infty$. The algorithm therefore gives
\be\label{eq:I1.41}
\calI_{p_1,V_1,\mathrm{I}}
\=
\calI_{0,0}-q^2\calI_{2,-1}
\=
\calI_{0,0}+\calI_{1,0}\,,
\ee
where the second equality follows from Equation~\eqref{eq:state_integral_1}.
\subsubsection{The saddle connection between $\mathrm{I}$ and $\mathrm{II}$}
The saddle connection is depicted in Figure~\ref{fig:thimb.41.saddleI.II}.
\begin{figure}[ht]
\begin{tiny}
\begin{tikzpicture}[scale=2]
\draw[<->,thick] (-2.5,0) -- (1,0);
\draw[<->,thick] (0,-1) -- (0,1);
\foreach \x in {-2,-1,0,1}
\draw[green] (\x,-1)--(\x,1);
\draw[ultra thick,red] (-0.05,-0.05)--(0.05,0.05);
\draw[ultra thick,red] (0.05,-0.05)--(-0.05,0.05);
\draw[ultra thick,red,xshift=-1cm] (-0.05,-0.05)--(0.05,0.05);
\draw[ultra thick,red,xshift=-1cm] (0.05,-0.05)--(-0.05,0.05);
\draw[ultra thick,red,xshift=-2cm] (-0.05,-0.05)--(0.05,0.05);
\draw[ultra thick,red,xshift=-2cm] (0.05,-0.05)--(-0.05,0.05);
\draw[orange,thick,->]
(-0.16667, 0)--(-0.15667, 0.36539)--(-0.14667, 0.51902)--(-0.13667, 0.63842)--(-0.12667, 0.74034)--(-0.11667, 0.83122)--(-0.10667, 0.91436)--(-0.096667, 0.99171) node[above] {$(0,0)$};
\draw[orange,thick,->]
(-0.16667,0)--(-0.15667, -0.36539)--(-0.14667, -0.51902)--(-0.13667, -0.63842)--(-0.12667, -0.74034)--(-0.11667, -0.83122)--(-0.10667, -0.91436)--(-0.096667, -0.99171) node[below] {$(-2,0)$};
\draw[orange,thick]
(-0.16667,0)--(-1, 0);
\draw[ultra thick,orange,,xshift=-1cm] (-0.05,0)--(0.05,0);
\draw[ultra thick,orange,xshift=-1cm] (0,-0.05)--(0,0.05);
\filldraw[red] (-5/6,0) circle (1pt);
\filldraw[red] (-1/6,0) circle (1pt);
\end{tikzpicture}
\end{tiny}
\caption{This figure depicts in \textcolor{orange}{orange} the set $\calC_{(-5/6,0),\vartheta}$ for $\vartheta=\arg(-1/\tau)=3.1416$ and $(A,B)=(1,2)$. The \textcolor{green}{green} lines are parallel to the lines $i/\tau$.}\label{fig:thimb.41.saddleI.II}
\end{figure}
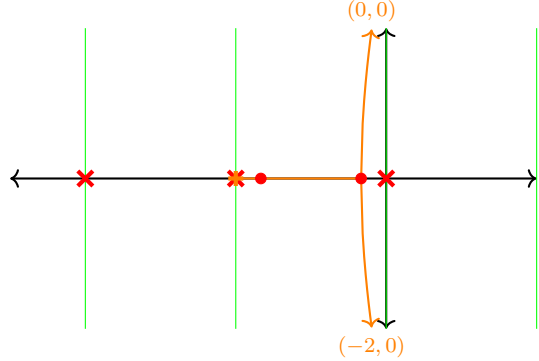

\subsubsection{State integrals decomposition in $\mathrm{II}$}
Let $\vartheta\in\mathrm{II}$ then $\calC_{p_1,\vartheta}$ is depicted in Figure~\ref{fig:thimb.41.II}.
\begin{figure}[ht]
\begin{tiny}
\begin{tikzpicture}[scale=2]
\draw[<->,thick] (-2.5,0) -- (1,0);
\draw[<->,thick] (0,-1) -- (0,1);
\foreach \x in {}
\draw[green] (-2.5,-0.58779*4*2.5/0.80902-0.58779*4*\x/0.80902)--(\x+0.80902/0.58779/4,0.58779*4/0.58779/4);
\foreach \x in {}
\draw[green] (-2.5,-0.58779*4*2.5/0.80902-0.58779*4*\x/0.80902)--(1,0.58779*4/0.80902-\x*0.58779*4);
\foreach \x in {1}
\draw[green] (\x-0.80902/0.58779/4,-1)--(1,0.58779*4/0.80902-\x*0.58779*4/0.80902);
\foreach \x in {-2,-1,0}
\draw[green] (\x-0.80902/0.58779/4,-1)--(\x+0.80902/0.58779/4,0.58779*4/0.58779/4);
\draw[ultra thick,red] (-0.05,-0.05)--(0.05,0.05);
\draw[ultra thick,red] (0.05,-0.05)--(-0.05,0.05);
\draw[ultra thick,red,xshift=-1cm] (-0.05,-0.05)--(0.05,0.05);
\draw[ultra thick,red,xshift=-1cm] (0.05,-0.05)--(-0.05,0.05);
\draw[ultra thick,red,xshift=-2cm] (-0.05,-0.05)--(0.05,0.05);
\draw[ultra thick,red,xshift=-2cm] (0.05,-0.05)--(-0.05,0.05);
\draw[orange,thick,<->]
(-0.52333, -0.97696) node[below] {$(-2,0)$}--(-0.53333, -0.92828)--(-0.54333, -0.88157)--(-0.55333, -0.83668)--(-0.56333, -0.79349)--(-0.57333, -0.75187)--(-0.58333, -0.71173)--(-0.59333, -0.67294)--(-0.60333, -0.63542)--(-0.61333, -0.59908)--(-0.62333, -0.56385)--(-0.63333, -0.52965)--(-0.64333, -0.49641)--(-0.65333, -0.46409)--(-0.66333, -0.43262)--(-0.67333, -0.40197)--(-0.68333, -0.37208)--(-0.69333, -0.34293)--(-0.70333, -0.31448)--(-0.71333, -0.28669)--(-0.72333, -0.25955)--(-0.73333, -0.23304)--(-0.74333, -0.20713)--(-0.75333, -0.18182)--(-0.76333, -0.15708)--(-0.77333, -0.13292)--(-0.78333, -0.10933)--(-0.79333, -0.086308)--(-0.80333, -0.063855)--(-0.81333, -0.041978)--(-0.82333, -0.020688)--(-0.83333, 0)--(-0.84333, 0.020067)--(-0.85333, 0.039487)--(-0.86333, 0.058230)--(-0.87333, 0.076258)--(-0.88333, 0.093524)--(-0.89333, 0.10997)--(-0.90333, 0.12554)--(-0.91333, 0.14014)--(-0.92333, 0.15369)--(-0.93333, 0.16607)--(-0.94333, 0.17715)--(-0.95333, 0.18679)--(-0.96333, 0.19482)--(-0.97333, 0.20108)--(-0.98333, 0.20537)--(-0.99333, 0.20756)--(-1.0033, 0.20750)--(-1.0133, 0.20513)--(-1.0233, 0.20043)--(-1.0333, 0.19344)--(-1.0433, 0.18423)--(-1.0533, 0.17291)--(-1.0633, 0.15958)--(-1.0733, 0.14436)--(-1.0833, 0.12735)--(-1.0933, 0.10866)--(-1.1033, 0.088382)--(-1.1133, 0.066606)--(-1.1233, 0.043408)--(-1.1333, 0.018860)--(-1.1433, -0.0069742)--(-1.1533, -0.034036)--(-1.1633, -0.062273)--(-1.1733, -0.091638)--(-1.1833, -0.12209)--(-1.1933, -0.15358)--(-1.2033, -0.18609)--(-1.2133, -0.21958)--(-1.2233, -0.25402)--(-1.2333, -0.28939)--(-1.2433, -0.32567)--(-1.2533, -0.36284)--(-1.2633, -0.40089)--(-1.2733, -0.43980)--(-1.2833, -0.47957)--(-1.2933, -0.52019)--(-1.3033, -0.56166)--(-1.3133, -0.60399)--(-1.3233, -0.64717)--(-1.3333, -0.69121)--(-1.3433, -0.73613)--(-1.3533, -0.78195)--(-1.3633, -0.82868)--(-1.3733, -0.87636)--(-1.3833, -0.92501)--(-1.3933, -0.97467)node[below] {$(-4,-2)$};
\filldraw[red] (-5/6,0) circle (1pt);
\filldraw[magenta](-1+1*0.80902*0.074,0.58779*4*0.074) circle (1pt);
\draw[magenta,->,thick] (-0.94013, 0.17413)--(-0.93013, 0.22356)--(-0.92013, 0.26717)--(-0.91013, 0.30635)--(-0.90013, 0.34205)--(-0.89013, 0.37495)--(-0.88013, 0.40553)--(-0.87013, 0.43417)--(-0.86013, 0.46115)--(-0.85013, 0.48669)--(-0.84013, 0.51098)--(-0.83013, 0.53416)--(-0.82013, 0.55636)--(-0.81013, 0.57768)--(-0.80013, 0.59821)--(-0.79013, 0.61803)--(-0.78013, 0.63721)--(-0.77013, 0.65581)--(-0.76013, 0.67388)--(-0.75013, 0.69146)--(-0.74013, 0.70860)--(-0.73013, 0.72533)--(-0.72013, 0.74169)--(-0.71013, 0.75772)--(-0.70013, 0.77343)--(-0.69013, 0.78887)--(-0.68013, 0.80404)--(-0.67013, 0.81897)--(-0.66013, 0.83369)--(-0.65013, 0.84821)--(-0.64013, 0.86255)--(-0.63013, 0.87673)--(-0.62013, 0.89076)--(-0.61013, 0.90466)--(-0.60013, 0.91844)--(-0.59013, 0.93212)--(-0.58013, 0.94571)--(-0.57013, 0.95922)--(-0.56013, 0.97266)--(-0.55013, 0.98605)--(-0.54013, 0.99939) node[above] {$(1,1)$};
\end{tikzpicture}
\end{tiny}
\caption{This figure depicts in \textcolor{orange}{orange} the set $\calC_{(-5/6,0),\vartheta}$ for $\vartheta=\arg(-1/\tau)=2.1991$, and $(A,B)=(1,2)$. The \textcolor{green}{green} lines are parallel to the lines $i/\tau$. The \textcolor{magenta}{magenta} curve represents the tail at the intersection of the set $\calC_{(-5/6,0),\vartheta}$ with the lines parallel to $1/\tau$.}
\label{fig:thimb.41.II}
\end{figure}
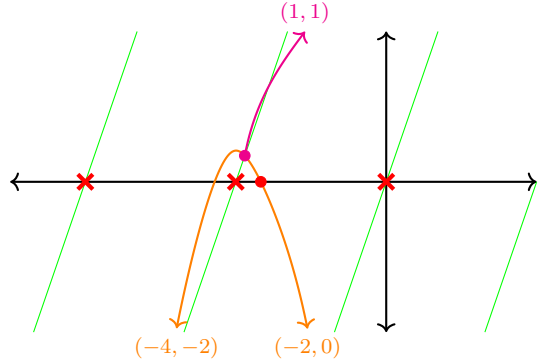
We can decompose the thimble in terms on state integrals contours. We have contributions from $-\calI_{0,0}$ from the intersection of the orange contour with the interval $(-1,0)$, which has $(n,m,\ell)=(0,0,0)$ and $\calI_{0,-1}$ from the intersection of the orange contour with the interval $(-2,-1)$, which has $(n,m,\ell)=(0,0,-1)$. We can see that the tail leads to no contribution as it does not cross the reals and heads to $\infty$. The algorithm therefore gives
\be\label{eq:I2.41}
\calI_{p_1,V_1,\mathrm{II}}
\=
-\calI_{0,0}+\calI_{0,-1}
\=
-\calI_{0,0}-\calI_{-1,0}\,,
\ee
where the equality follows from Equation~\eqref{eq:state_integral_1}.
\subsubsection{The saddle connection between II and III}
The saddle connection is depicted in Figure~\ref{fig:thimb.41.saddle.II.III}.
\begin{figure}[ht]
\begin{tiny}

\end{tiny}
\caption{This figure depicts in \textcolor{orange}{orange} the set $\calC_{(-5/6,0),\vartheta}$ for $\vartheta=\arg(-1/\tau)=1.6733$, and $(A,B)=(1,2)$. The \textcolor{green}{green} lines are parallel to the lines $i/\tau$.
}\label{fig:thimb.41.saddle.II.III}
\end{figure}
\subsubsection{State integrals decomposition in $\mathrm{III}$}
Let $\vartheta\in\mathrm{III}$ then $\calC_{p_1,\vartheta}$ is depicted in Figure~\ref{fig:thimb.41.III}.
\begin{figure}[ht]
\begin{tiny}
%
\end{tiny}
\caption{This figure depicts in \textcolor{orange}{orange} the set $\calC_{(-5/6,0),\vartheta}$ for $\vartheta=\arg(-1/\tau)=1.6650$ and $(A,B)=(1,2)$. The \textcolor{green}{green} lines are parallel to the lines $1/\tau$.}
\label{fig:thimb.41.III}
\end{figure}
\begin{figure}[ht]
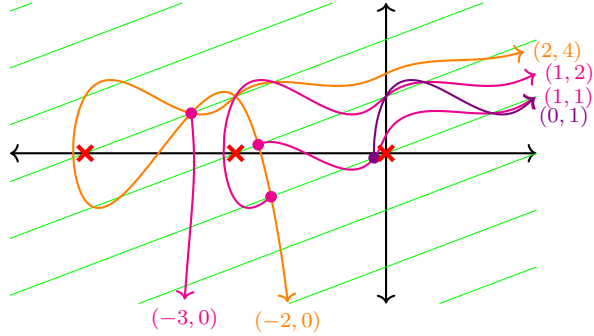

\begin{tiny}
%
\end{tiny}
\caption{This figure depicts in \textcolor{orange}{orange} the set $\calC_{(-5/6,0),\vartheta}$ for $\vartheta=\arg(-1/\tau)=1.6650$, and $(A,B)=(1,2)$. The \textcolor{green}{green} lines are parallel to the lines $i/\tau$. The \textcolor{magenta}{magenta} contours represent half of the additional contours that appear in the first step of the algorithm where we ignore any zero contributions. Then the \textcolor{violet}{violet} contour represents half of the contour that appears in the second step of the algorithm. The additional contours are all crossing the real axis, hence they lead to a non-zero state integral.
}\label{fig:thimb.41.III.t1}
\end{figure}
This picture gives rise to many intersections between the thimble and the green lines $\BZ+\tau^{-1}\BR$. Therefore, it could require iterations of the algorithm. This is indeed true and will involve two additional integrals at the first iteration and an additional integral at the second iteration.

\medskip

We can decompose the thimble in terms on state integrals contours. We have contributions of $-\calI_{0,0}$ from the intersection of the orange contour with the interval $(-1,0)$, which has $(n,m,\ell)=(0,0,0)$, $\calI_{0,-1}$ from the intersection of the orange contour with the interval $(-2,-1)$, which has $(n,m,\ell)=(0,0,-1)$, and $-q^4\calI_{2,-2}$ from the intersection of the orange contour with the interval $(-3,-2)$, which has $(n,m,\ell)=(4,2,-2)$.

\medskip

At the first step of the algorithm we obtain two additional cycles coming from the magenta curves as seen in Figure~\ref{fig:thimb.41.III.t1}. We have a contribution of $4q^2\calI_{1,-1}$ from the intersection of the two magenta contours in $(-1,0)$, which have $(n,m,\ell)=(2,1,-1)$. There is a final magenta curve depicted, which has additional intersections with the line $\tau^{-1}\BR$.

\medskip

At the second step of the algorithm there is an additional contribution coming from the violet contour in Figure~\ref{fig:thimb.41.III.t1}. This appears with the constant attached to the magenta contour it emanates from multiplied by its own constant. We have a contribution of $-4q\calI_{0,0}$ from the intersection of the violet contour with the interval $(-1,0)$, which has $(n,m,\ell)=(1,0,0)$. The additional state integrals added are depicted in Figure~\ref{fig:thimb.41.III.t2}.
\begin{figure}[ht]
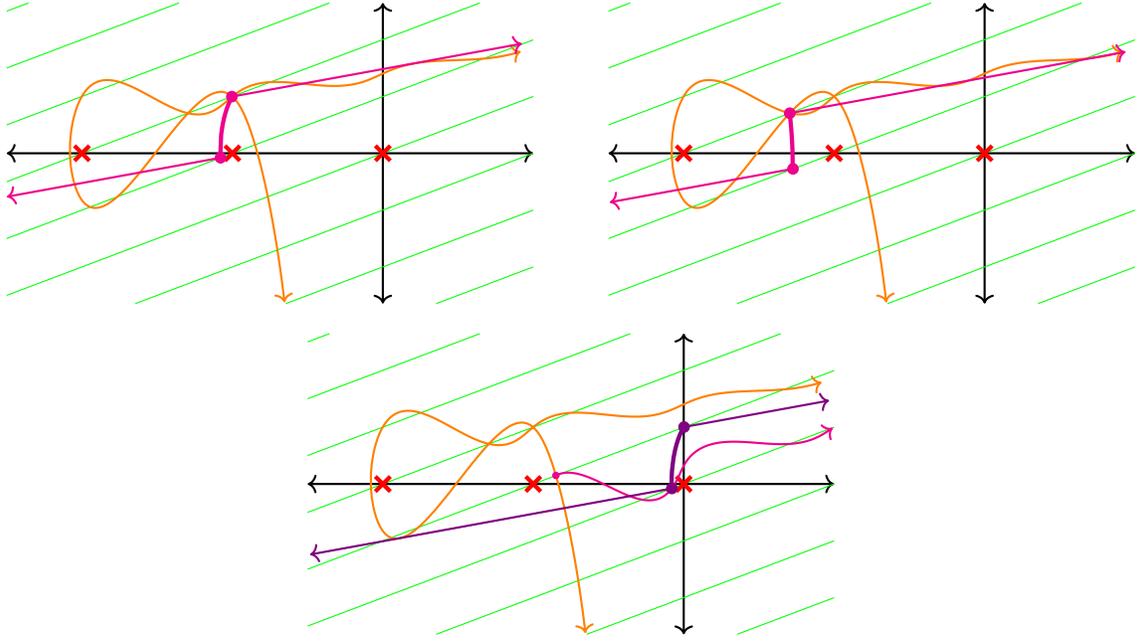


\caption{This figure depicts the contours $\gamma_{e,\vartheta}$ that we add at the first and second iteration of the algorithm to $\calI_{p_1,V_1,\mathrm{III}}$ as showed in Figure~\ref{fig:thimb.41.III.t1}.
}\label{fig:thimb.41.III.t2}
\end{figure}
\medskip

All together the algorithm produces the function
\be\label{eq:I3.41}
  \calI_{p_1,V_1,\mathrm{III}}
  \=
  -\calI_{0,0}+\calI_{0,-1}
  +4q^2\calI_{1,-1}
  -4q\calI_{0,0}
  -q^4\calI_{2,-2}
  \=
  -\calI_{0,0}-\calI_{-1,0}-9q\calI_{0,0}\,,
\ee
where the second equality follows from Equation~\eqref{eq:state_integral_1}.
\subsubsection{First few Stokes constants}
We can now describe the Stokes matrices. We see from Equation~\eqref{eq:Iarb.41}, Equation~\eqref{eq:I1.41}, Equation~\eqref{eq:I2.41}, and Equation~\eqref{eq:I3.41} that
\be
  \begin{pmatrix}
    \calI_{p_1,V_1,\mathrm{II}} & \calI_{p_2,V_2,\mathrm{II}}
  \end{pmatrix}
  \=
  \begin{pmatrix}
    \calI_{p_1,V_1,\mathrm{I}} & \calI_{p_2,V_2,\mathrm{I}}
  \end{pmatrix}
  \begin{pmatrix}
    1 & 0\\
    -3 & 1
  \end{pmatrix}\,,
\ee
and
\be
  \begin{pmatrix}
    \calI_{p_1,V_1,\mathrm{III}} & \calI_{p_2,V_2,\mathrm{III}}
  \end{pmatrix}
  \=
  \begin{pmatrix}
    \calI_{p_1,V_1,\mathrm{II}} & \calI_{p_2,V_2,\mathrm{II}}
  \end{pmatrix}
  \begin{pmatrix}
    1 & 0\\
    -9q & 1
  \end{pmatrix}\,.
\ee
These agree with the Stokes constants computed in~\cite{GGM:I}. Therefore, we see that
\be
  \mathfrak{S}_{\mathrm{I},\mathrm{II}}(q)
  \=
  \begin{pmatrix}
    1 & 0\\
    -3 & 1
  \end{pmatrix}\,,\qquad
  \mathfrak{S}_{\mathrm{II},\mathrm{III}}(q)
  \=
  \begin{pmatrix}
    1 & 0\\
    -9q & 1
  \end{pmatrix}\,.
\ee

\subsubsection{Computing all Stokes constants}

We want to compute the Stokes matrix $\mathfrak{S}_{\pi-\epsilon,0+\epsilon}$ for small $\epsilon>0$. In order to do this, we will use formulae for state integrals found by~\cite{GK:qser}. This was the method used in the numerical approach of~\cite{GGM:I}. We can express state integrals---and therefore their Borel--Laplace resummations---as bilinear combinations of $q=\e(\tau)$ and $\tq=\e(-1/\tau)$-series. Consider the collections of $q$-series
\be
\begin{aligned}
  g_m(q)
  &\=
  \sum_{k=0}^{\infty}(-1)^k\frac{q^{k(k+1)/2+mk}}{(q;q)_{k}^2}\,,\\
  G_m(q)
  &\=
  \sum_{k=0}^{\infty}(-1)^k\frac{q^{k(k+1)/2+mk}}{(q;q)_{k}^2}\Big(m-2E_1(q)+\sum_{\ell=1}\frac{1+q^{\ell}}{1-q^{\ell}}\Big)\,,
\end{aligned}
\ee
where
\be
  E_1(q)
  \=
  -\frac{1}{4}+\sum_{k=1}^{\infty}\frac{q^k}{1-q^k}
  \=-\tfrac{1}{4}+q + 2q^2 + 2q^3 + 3q^4 + 2q^5 + 4q^6 + 2q^7 + 4q^8 + 3q^9+\cdots\,.
\ee
For example, we have
\be
\begin{aligned}
  g_{0}(q)
  &\=
  1 - q - 2q^2 - 2q^3 - 2q^4 + q^6 + 5q^7 + 7q^8 + 11q^9+ 13q^{10} + 16q^{11}+\cdots\,,\\
  G_{0}(q)
  &\=
  \tfrac{1}{2}(1 - 7q - 14q^2 - 8q^3 - 2q^4 + 30q^5 + 43q^6 + 95q^7 + 109q^8 + 137q^9+\cdots)\,.
\end{aligned}
\ee
Then as shwon in~\cite[Eq.~(14)]{GK:qser}, the state integral is given by the following formula
\be\label{eq:si.qqt}
\begin{aligned}
  &i\tq^{-1/12}q^{1/12}
  \int_{\calJ_{\ell,\tau}}\Phi(z\tau,\tau)^2\,\e\Big(\frac{1}{2} z(z\tau+\tau+1)+mz\tau+nz\Big) dz\\
  &\=
  \tau^{-1}g_m(q)G_{-n}(\tq)-g_{-n}(\tq)G_m(q)\,.
\end{aligned}
\ee
We can use this to write the Borel--Laplace resummation in matrix form
\be
\begin{aligned}
  &\begin{pmatrix}
    \calI_{p_1,V_1,\mathrm{I}} & \calI_{p_2,V_2,\mathrm{I}}
  \end{pmatrix}\\
  &\=
  \begin{pmatrix}
    g_{0}(\tq) & G_{0}(\tq)
  \end{pmatrix}
  \begin{pmatrix}
    -1 & 0\\
    0 & \tau^{-1}
  \end{pmatrix}
  \begin{pmatrix}
    G_0(q)+G_{1}(q) & G_0(q)\\
    g_0(q)+g_{1}(q) & g_0(q)
  \end{pmatrix}\,.
\end{aligned}
\ee
We can similarly, write
\be
\begin{aligned}
  &\begin{pmatrix}
    \calI_{p_1,V_1,\mathrm{X}} & \calI_{p_2,V_2,\mathrm{X}}
  \end{pmatrix}\\
  &\=
  \begin{pmatrix}
    g_{0}(\tq) & G_{0}(\tq)
  \end{pmatrix}
  \begin{pmatrix}
    1 & 0\\
    0 & \tau^{-1}
  \end{pmatrix}
  \begin{pmatrix}
    -G_0(q) & G_0(q)+G_1(q)\\
    -g_0(q) & g_0(q)+g_1(q)
  \end{pmatrix}\,.
\end{aligned}
\ee
Then we find\footnote{To really prove this, one needs to make a full matrix equation by taking another choice of $n$ in Equation~\eqref{eq:si.qqt}, which has exactly the same Stokes automorphisms as discussed in the Section~\ref{sec:BL-intro}.} that the Stokes matrix going from $\mathrm{I}$ to $\mathrm{X}$ is given by
\be
\begin{aligned}
  &\mathsf{S}_+(q)\\
  &\=\begin{pmatrix}
    G_0(q)+G_{-1}(q) & G_0(q)\\
    g_0(q)+g_{-1}(q) & g_0(q)
  \end{pmatrix}^{-1}
  \begin{pmatrix}
    -1 & 0\\
    0 & 1
  \end{pmatrix}
  \begin{pmatrix}
    -G_0(q) & G_0(q)+G_1(q)\\
    -g_0(q) & g_0(q)+g_1(q)
  \end{pmatrix}\\
  &\=
  \begin{pmatrix}
  -1 + 8q + 9q^2 - 18q^3 - 46q^4 & 3 - 15q - 24q^2 + 15q^3 + 69q^4\\
  -9q - 3q^2 + 39q^3 + 69q^4 & -1 + 19q + 17q^2 - 53q^3 - 126q^4
  \end{pmatrix}+\cdots\\
  &\=
  \begin{pmatrix}
  1 & 0\\
  -3 - 9q - 75q^2 - 642q^3 - 5580q^4 & 1
  \end{pmatrix}\\
  &\quad\times\begin{pmatrix}
  1 - 8q - 9q^2 + 18q^3 + 46q^4 & 0\\
  0 & 1 + 8q + 73q^2 + 638q^3 + 5571q^4
  \end{pmatrix}\\
  &\quad\times\begin{pmatrix}
  1 & 9q + 75q^2 + 642q^3 + 5580q^4\\
  0 & 1
  \end{pmatrix}+\cdots\,.
\end{aligned}
\ee
One can see exact agreement with the previous computations of the first two Stokes constants $-3$ and $-9$. To compute the next constant $-75$ using the methods we have described would presumably involve many iterations of the algorithm and would be best done via computer implementation, which we have not attempted.
\subsection{\texorpdfstring{The case of $(A,B)=(4,1)$}{The case of (A,B)=(4,1)}}\label{sec:A4B1}
We will compute the Borel--Laplace resummation of $\Phi_\Xi$ with the algorithm described in Section~\ref{sec:alg}. We will see that this exactly matches the numerical predictions of~\cite[Example~39]{Wh:thesis}. The critical points for the function $V$ and their critical values are listed below 
\begin{center}
\begin{longtblr}{colspec={c|c|[3pt] c|c}}
$p_1$ & $(-0.5 -0.0317451 i,0)$ & $V_1$ & $-0.435753$\\ 
\hline
$p_2$ & $(0.0512932 i,-2)$ & $V_2$ & $0.0127913$\\
\hline
$p_3$ & $(-0.212516-0.009774 i, -1)$ & $V_3$ & $-0.0801858 +0.0248584 i$\\
\hline
$p_4$ & $(0.212516-0.009774 i, -3)$ & $V_4$ & $-0.0801858 -0.0248584 i$
\end{longtblr}
\end{center}
Taking the relative difference between volumes, the first non trivial\footnote{By non trivial we mean with non zero Stokes constants. Indeed the line at $\vartheta=1.27018$ is not an effective Stokes line.} Stokes line is at $e^{i\vartheta}\BR_{>0}$ with $\vartheta=1.30955$ and it separates the regions $\mathrm{I}_\star$ and $\mathrm{II}$. The second one is at $\vartheta=1.501$ and separates the regions $\mathrm{II}$ and $\mathrm{III}$. The third one is at $\vartheta=1.52112$ and separates the regions $\mathrm{III}$ and $\mathrm{IV}$. As expected, they form the peacock pattern as illustrated in Figure~\ref{fig:A4B1.stokes}.
\begin{figure}[ht]
\begin{tiny}
\begin{tikzpicture}[scale=2]
\draw[<->,thick] (-2,0) -- (2,0);
\draw[<->,thick] (0,-1) -- (0,1);
\foreach \x in {1,2,3,4,5,6,7,8,9,10,11,12,13,14,15,16,17,18,19,20}
\draw[red,->] (0,0)--(1/\x/5,1);
\foreach \x in {1,2,3,4,5,6,7,8,9,10,11,12,13,14,15,16,17,18,19,20}
\draw[red,->] (0,0)--(1/\x/0.9,1);
\foreach \x in {1,2,3,4,5,6,7,8,9,10,11,12,13,14,15,16,17,18,19,20}
\draw[red,->] (0,0)--(1/\x/3.7,1);
\foreach \x in {1,2,3,4,5,6,7,8,9,10,11,12,13,14,15,16,17,18,19,20}
\draw[red,->] (0,0)--(1/\x/6.4,1);
\foreach \x in {-3,-2,-1,0,1,2,3}
\draw[red,->,thick] (0,0)--(0-0.01*\x,1);
\foreach \x in {1,2,3,4,5,6,7,8,9,10,11,12,13,14,15,16,17,18,19,20}
\draw[red,->] (0,0)--(-1/\x/6.4,1);
\foreach \x in {1,2,3,4,5,6,7,8,9,10,11,12,13,14,15,16,17,18,19,20}
\draw[red,->] (0,0)--(-1/\x/5,1);
\foreach \x in {1,2,3,4,5,6,7,8,9,10,11,12,13,14,15,16,17,18,19,20}
\draw[red,->] (0,0)--(-1/\x/3.7,1);
\foreach \x in {1,2,3,4,5,6,7,8,9,10,11,12,13,14,15,16,17,18,19,20}
\draw[red,->] (0,0)--(-1/\x/0.9,1);
\foreach \x in {1,2,3,4,5,6,7,8,9,10,11,12,13,14,15,16,17,18,19,20}
\draw[red,->] (0,0)--(1/\x/6.4,-1);
\foreach \x in {1,2,3,4,5,6,7,8,9,10,11,12,13,14,15,16,17,18,19,20}
\draw[red,->] (0,0)--(1/\x/5,-1);
\foreach \x in {1,2,3,4,5,6,7,8,9,10,11,12,13,14,15,16,17,18,19,20}
\draw[red,->] (0,0)--(1/\x/3.7,-1);
\foreach \x in {1,2,3,4,5,6,7,8,9,10,11,12,13,14,15,16,17,18,19,20}
\draw[red,->] (0,0)--(1/\x/0.9,-1);
\foreach \x in {1,2,3,4,5,6,7,8,9,10,11,12,13,14,15,16,17,18,19,20}
\draw[red,->] (0,0)--(-1/\x/6.4,-1);
\foreach \x in {1,2,3,4,5,6,7,8,9,10,11,12,13,14,15,16,17,18,19,20}
\draw[red,->] (0,0)--(-1/\x/5,-1);
\foreach \x in {1,2,3,4,5,6,7,8,9,10,11,12,13,14,15,16,17,18,19,20}
\draw[red,->] (0,0)--(-1/\x/3.7,-1);
\foreach \x in {1,2,3,4,5,6,7,8,9,10,11,12,13,14,15,16,17,18,19,20}
\draw[red,->] (0,0)--(-1/\x/0.9,-1);
\foreach \x in {-3,-2,-1,0,1,2,3}
\draw[red,->,thick] (0,0)--(0-0.01*\x,-1);
\draw[red,dashed,->] (0,0)--(1.5,1.5*0.6); 
\draw[red,<->] (-2,0)--(2,0);
\draw[font=\tiny](1.4,0.4) node {$\mathrm{I}$};
\draw[font=\tiny](1,0.73) node {$\mathrm{I_\star}$};
\draw[font=\tiny](0.7,0.9) node {$\mathrm{II}$};
\draw[font=\fontsize{3pt}{4pt}\selectfont](0.415,0.9) node {$\mathrm{III}$};
\draw[red] (2,0) node[right] {$\vartheta=\pi$};
\draw[red] (-2,0) node[left] {$\vartheta=0$};
\end{tikzpicture}
\end{tiny}
\caption{This figure depicts the Stokes rays in $\tau$-plane with $\vartheta=\arg(-1/\tau)$ forming a peacock pattern. Note that this is stretched in the direction of the $x$-axis and the Stokes rays are all even closer to the vertical line. In addition, it is symmetric with respect to the imaginary axis: on the right, the first solid Stokes rays contain the singularties at the points $V_2-V_4$, $V_3-V_1$, $V_3-V_4+1$, $V_1-V_4+1$, which correspond to $\vartheta=1.83204$, $\vartheta=1.64069$, $\vartheta=1.62047$ and $\vartheta=1.60935$, respectively. The dashed Stokes line corresponds to $-V_4$, with $\vartheta=1.87141$, and crossing that ray there will be no Stokes phenomenon as shown in Section~\ref{sec:stokes_A4B1}.}\label{fig:A4B1.stokes}
\end{figure}
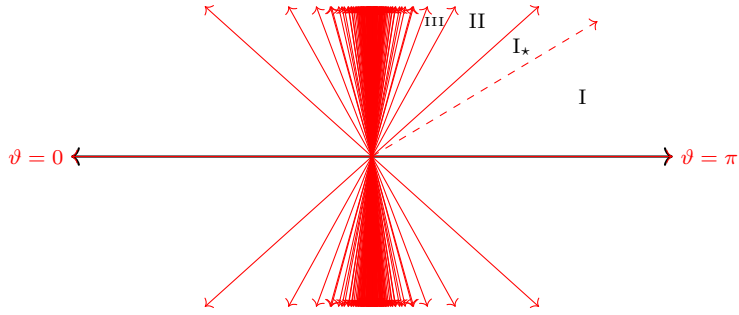

\medskip

To perform the computation, we plot the set $\calC_{p_0,\vartheta}$ and keep track of the branching of the function $V$. We adopt the same conventions as in Section~\ref{sec:4_1}. 

\subsubsection{State integrals decomposition in $\mathrm{I}$}
Let $\vartheta\in\mathrm{I}$. We plot the curves $\calC_{p_j,\vartheta}$ for $j=1,2,3,4$ in Figure~\ref{fig:thimb.A4B1.I}.
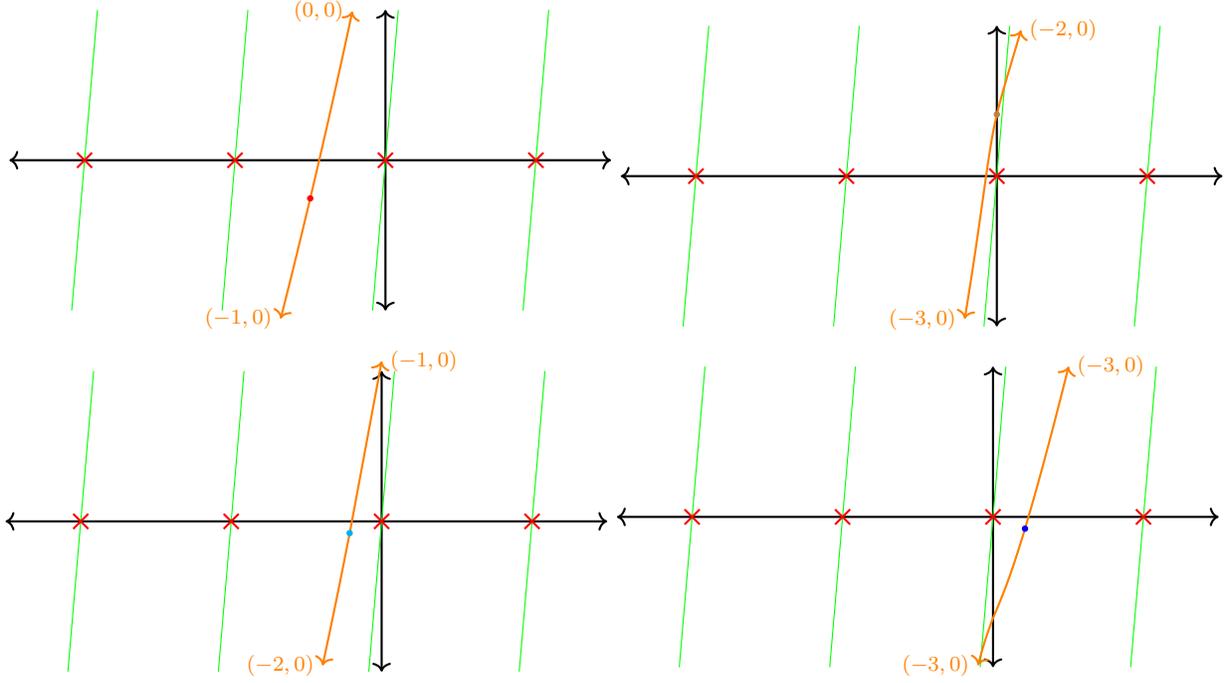
\begin{figure}[ht]
\center
\begin{tiny}
\begin{tikzpicture}[scale=2]
\draw[<->,thick] (-2.5,0) -- (1.5,0);
\draw[<->,thick] (0,-1) -- (0,1);
\foreach \x in {-2,-1,0,1}
\draw[green] (-0.564642/0.825336/8+\x,-1)--(0.564642/0.825336/8+\x,1);
\draw[thick,red] (-0.05,-0.05)--(0.05,0.05);
\draw[ thick,red] (0.05,-0.05)--(-0.05,0.05);
\draw[thick,red,xshift=-1cm] (-0.05,-0.05)--(0.05,0.05);
\draw[thick,red,xshift=-1cm] (0.05,-0.05)--(-0.05,0.05);
\draw[thick,red,xshift=-2cm] (-0.05,-0.05)--(0.05,0.05);
\draw[thick,red,xshift=-2cm] (0.05,-0.05)--(-0.05,0.05);
\draw[thick,red,xshift=1cm] (-0.05,-0.05)--(0.05,0.05);
\draw[thick,red,xshift=1cm] (0.05,-0.05)--(-0.05,0.05);
\draw[orange,thick,->] (-0.5,- 0.0317451*8)--(-0.49000, -0.21168)--(-0.48000, -0.16925)--(-0.47000, -0.12667)--(-0.46000, -0.083928)--(-0.45000, -0.041031)--(-0.44000, 0.0020231)--(-0.43000, 0.045235)--(-0.42000, 0.088605)--(-0.41000, 0.13213)--(-0.40000, 0.17582)--(-0.39000, 0.21967)--(-0.38000, 0.26369)--(-0.37000, 0.30786)--(-0.36000, 0.35219)--(-0.35000, 0.39669)--(-0.34000, 0.44135)--(-0.33000, 0.48618)--(-0.32000, 0.53116)--(-0.31000, 0.57631)--(-0.30000, 0.62162)--(-0.29000, 0.66709)--(-0.28000, 0.71272)--(-0.27000, 0.75850)--(-0.26000, 0.80443)--(-0.25000, 0.85051)--(-0.24000, 0.89673)--(-0.23000, 0.94309)--(-0.22000, 0.98957) node[left] {$(0,0)$};
\draw[orange,thick,->] (-0.5,- 0.0317451*8)--(-0.5190233114, -0.3339608000)--(-0.5381794073, -0.4139608000)--(-0.5574701920, -0.4939608000)--(-0.5768977462, -0.5739608000)--(-0.5964642113, -0.6539608000)--(-0.6161717141, -0.7339608000)--(-0.6360222521, -0.8139608000)--(-0.6560175276, -0.8939608000)--(-0.6761587146, -0.9739608000)--(-0.6964461414, -1.053960800) node[left] {$(-1,0)$};
\filldraw[red] (-0.5,- 0.0317451*8) circle (0.5pt);
\end{tikzpicture}
\vspace{0.5mm}
\begin{tikzpicture}[scale=2]
\draw[<->,thick] (-2.5,0) -- (1.5,0);
\draw[<->,thick] (0,-1) -- (0,1);
\foreach \x in {-2,-1,0,1}
\draw[green] (-0.564642/0.825336/8+\x,-1)--(0.564642/0.825336/8+\x,1);
\draw[thick,red] (-0.05,-0.05)--(0.05,0.05);
\draw[ thick,red] (0.05,-0.05)--(-0.05,0.05);
\draw[thick,red,xshift=-1cm] (-0.05,-0.05)--(0.05,0.05);
\draw[thick,red,xshift=-1cm] (0.05,-0.05)--(-0.05,0.05);
\draw[thick,red,xshift=-2cm] (-0.05,-0.05)--(0.05,0.05);
\draw[thick,red,xshift=-2cm] (0.05,-0.05)--(-0.05,0.05);
\draw[thick,red,xshift=1cm] (-0.05,-0.05)--(0.05,0.05);
\draw[thick,red,xshift=1cm] (0.05,-0.05)--(-0.05,0.05);
\draw[orange,thick,->] (0,0.0512932*8)--(0.020183, 0.49035)--(0.042316, 0.57035)--(0.065578, 0.65035)--(0.089337, 0.73035)--(0.11321, 0.81035)--(0.13699, 0.89035)--(0.16058, 0.97035) node[right] {$(-2,0)$};
\draw[orange,thick,->] (0,0.0512932*8)--(-0.017626, 0.33035)--(-0.032785, 0.25035)--(-0.046140, 0.17035)--(-0.058391, 0.090346)--(-0.070036, 0.010346)--(-0.081387, -0.069654)--(-0.092637, -0.14965)--(-0.10391, -0.22965)--(-0.11528, -0.30965)--(-0.12681, -0.38965)--(-0.13853, -0.46965)--(-0.15045, -0.54965)--(-0.16260, -0.62965)--(-0.17498, -0.70965)--(-0.18759, -0.78965)--(-0.20043, -0.86965)--(-0.21351, -0.94965) node[left] {$(-3,0)$};
\filldraw[brown] (0,0.0512932*8) circle (0.5pt);
\end{tikzpicture}
\vspace{0.5mm}
\begin{tikzpicture}[scale=2]
\draw[<->,thick] (-2.5,0) -- (1.5,0);
\draw[<->,thick] (0,-1) -- (0,1);
\foreach \x in {-2,-1,0,1}
\draw[green] (-0.564642/0.825336/8+\x,-1)--(0.564642/0.825336/8+\x,1);
\draw[thick,red] (-0.05,-0.05)--(0.05,0.05);
\draw[ thick,red] (0.05,-0.05)--(-0.05,0.05);
\draw[thick,red,xshift=-1cm] (-0.05,-0.05)--(0.05,0.05);
\draw[thick,red,xshift=-1cm] (0.05,-0.05)--(-0.05,0.05);
\draw[thick,red,xshift=-2cm] (-0.05,-0.05)--(0.05,0.05);
\draw[thick,red,xshift=-2cm] (0.05,-0.05)--(-0.05,0.05);
\draw[thick,red,xshift=1cm] (-0.05,-0.05)--(0.05,0.05);
\draw[thick,red,xshift=1cm] (0.05,-0.05)--(-0.05,0.05);
\draw[orange,thick,->] (-0.212516,-0.00977403*8)--(-0.22252, -0.12993)--(-0.23252, -0.18126)--(-0.24252, -0.23222)--(-0.25252, -0.28280)--(-0.26252, -0.33303)--(-0.27252, -0.38291)--(-0.28252, -0.43247)--(-0.29252, -0.48171)--(-0.30252, -0.53064)--(-0.31252, -0.57928)--(-0.32252, -0.62764)--(-0.33252, -0.67573)--(-0.34252, -0.72355)--(-0.35252, -0.77112)--(-0.36252, -0.81845)--(-0.37252, -0.86554)--(-0.38252, -0.91240)--(-0.39252, -0.95904) node[left] {$(-2,0)$};
\draw[orange,thick,->] (-0.212516,-0.00977403*8)--(-0.20252, -0.026048)--(0.00022597, 1.0652) node[right] {$(-1,0)$};
\filldraw[cyan] (-0.212516,-0.009774*8) circle (0.5pt);
\end{tikzpicture}
\vspace{0.5mm}
\begin{tikzpicture}[scale=2]
\draw[<->,thick] (-2.5,0) -- (1.5,0);
\draw[<->,thick] (0,-1) -- (0,1);
\foreach \x in {-2,-1,0,1}
\draw[green] (-0.564642/0.825336/8+\x,-1)--(0.564642/0.825336/8+\x,1);
\draw[thick,red] (-0.05,-0.05)--(0.05,0.05);
\draw[ thick,red] (0.05,-0.05)--(-0.05,0.05);
\draw[thick,red,xshift=-1cm] (-0.05,-0.05)--(0.05,0.05);
\draw[thick,red,xshift=-1cm] (0.05,-0.05)--(-0.05,0.05);
\draw[thick,red,xshift=-2cm] (-0.05,-0.05)--(0.05,0.05);
\draw[thick,red,xshift=-2cm] (0.05,-0.05)--(-0.05,0.05);
\draw[thick,red,xshift=1cm] (-0.05,-0.05)--(0.05,0.05);
\draw[thick,red,xshift=1cm] (0.05,-0.05)--(-0.05,0.05);
\draw[orange,thick,->] (0.212516,-0.009774*8)--(0.22252, -0.044656)--(0.23252, -0.010760)--(0.24252, 0.023478)--(0.25252, 0.058045)--(0.26252, 0.092924)--(0.27252, 0.12810)--(0.28252, 0.16357)--(0.29252, 0.19932)--(0.30252, 0.23533)--(0.31252, 0.27159)--(0.32252, 0.30811)--(0.33252, 0.34487)--(0.34252, 0.38185)--(0.35252, 0.41907)--(0.36252, 0.45650)--(0.37252, 0.49414)--(0.38252, 0.53199)--(0.39252, 0.57004)--(0.40252, 0.60829)--(0.41252, 0.64672)--(0.42252, 0.68535)--(0.43252, 0.72415)--(0.44252, 0.76313)--(0.45252, 0.80229)--(0.46252, 0.84162)--(0.47252, 0.88111)--(0.48252, 0.92077)--(0.49252, 0.96059)--(0.50252, 1.0006) node[right] {$(-3,0)$};
\draw[thick,orange,->] (0.212516,-0.009774*8)--(0.20252, -0.11135)--(0.19252, -0.14411)--(0.18252, -0.17646)--(0.17252, -0.20836)--(0.16252, -0.23980)--(0.15252, -0.27074)--(0.14252, -0.30116)--(0.13252, -0.33104)--(0.12252, -0.36033)--(0.11252, -0.38902)--(0.10252, -0.41707)--(0.092516, -0.44447)--(0.082516, -0.47119)--(0.072516, -0.49725)--(0.062516, -0.52266)--(0.052516, -0.54749)--(0.042516, -0.57184)--(0.032516, -0.59585)--(0.022516, -0.61976)--(0.012516, -0.64381)--(0.0025160, -0.66830)--(0, -0.67457)--(-0.010000, -0.70003)--(-0.020000, -0.72654)--(-0.030000, -0.75431)--(-0.040000, -0.78347)--(-0.050000, -0.81407)--(-0.060000, -0.84612)--(-0.070000, -0.87959)--(-0.080000, -0.91439)--(-0.090000, -0.95043)--(-0.10000, -0.98762) node[left] {$(-3,0)$};
\filldraw[blue] (0.212516,-0.009774*8) circle (0.5pt);
\end{tikzpicture}
\end{tiny}
\caption{From the right, this figure depicts in \textcolor{orange}{orange} the analytic continuation of the set $\calC_{p_1,\vartheta}, \calC_{p_2,\vartheta}, \calC_{p_3,\vartheta}$ and $\calC_{p_4,\vartheta}$ for $\vartheta=\arg(-1/\tau)=2.54159$ and $(A,B)=(4,1)$. The \textcolor{green}{green} lines are parallel to the lines $i/\tau$.}
\label{fig:thimb.A4B1.I}
\end{figure}
We can decompose the thimble in terms on state integrals contours, hence the algorithm gives
\be\label{eq:region_I_A4B1}
\begin{aligned}
\calI_{p_1,V_1,\mathrm{I}}
&\=
\calI_{0,0}\,,\qquad 
\calI_{p_2,V_2,\mathrm{I}}
&\=
\calI_{-2,0}\,, \qquad
\calI_{p_3,V_3,\mathrm{I}}
&\=
\calI_{-1,0}\,, \qquad
\calI_{p_4,V_4,\mathrm{I}}
&\=
\calI_{-3,1}\,.
\end{aligned}
\ee
\subsubsection{The \emph{false} saddle connection}
Let $\vartheta=1.87141$, which corresponds to the Stokes line at the angle $\arg\big(-\frac{2\pi i}{V_4}\big)$. We plot the set $\calC_{p_4,\vartheta}$ in Figure~\ref{fig:FAKE_saddle.A4B1}.
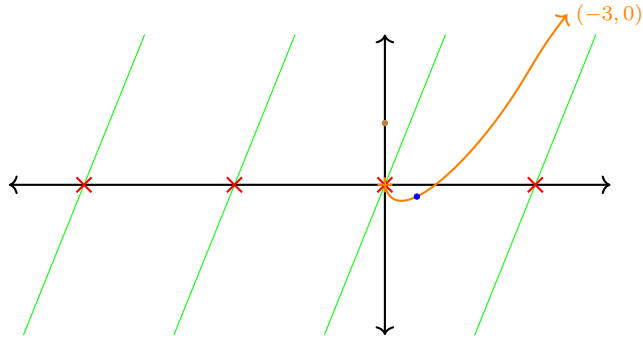
\begin{figure}[ht]
\begin{tiny}
\begin{tikzpicture}[scale=2]
\draw[<->,thick] (-2.5,0) -- (1.5,0);
\draw[<->,thick] (0,-1) -- (0,1);
\foreach \x in {-2,-1,0,1}
\draw[green] (-0.95516/8/0.29609+\x,-1)--(0.95516/8/0.29609+\x,1);
\draw[thick,red] (-0.05,-0.05)--(0.05,0.05);
\draw[thick,red] (0.05,-0.05)--(-0.05,0.05);
\draw[thick,red,xshift=-1cm] (-0.05,-0.05)--(0.05,0.05);
\draw[thick,red,xshift=-1cm] (0.05,-0.05)--(-0.05,0.05);
\draw[thick,red,xshift=-2cm] (-0.05,-0.05)--(0.05,0.05);
\draw[thick,red,xshift=-2cm] (0.05,-0.05)--(-0.05,0.05);
\draw[thick,red,xshift=1cm] (-0.05,-0.05)--(0.05,0.05);
\draw[thick,red,xshift=1cm] (0.05,-0.05)--(-0.05,0.05);
\draw[orange,thick,->] (0.2225160000, -0.07324363262)--(0.2325160000, -0.06800069288)--(0.2425160000, -0.06247886801)--(0.2525160000, -0.05669201082)--(0.2625160000, -0.05065262820)--(0.2725160000, -0.04437204753)--(0.2825160000, -0.03786055704)--(0.2925160000, -0.03112752498)--(0.3025160000, -0.02418150133)--(0.3125160000, -0.01703030502)--(0.3225160000, -0.009681099135)--(0.3325160000, -0.002140455987)--(0.3425160000, 0.005585586302)--(0.3525160000, 0.01349147447)--(0.3625160000, 0.02157209693)--(0.3725160000, 0.02982274590)--(0.3825160000, 0.03823908403)--(0.3925160000, 0.04681711498)--(0.4025160000, 0.05555315746)--(0.4125160000, 0.06444382221)--(0.4225160000, 0.07348599155)--(0.4325160000, 0.08267680137)--(0.4425160000, 0.09201362503)--(0.4525160000, 0.1014940591)--(0.4625160000, 0.1111159108)--(0.4725160000, 0.1208771868)--(0.4825160000, 0.1307760832)--(0.4925160000, 0.1408109772)--(0.5025160000, 0.1509804192)--(0.5125160000, 0.1612831261)--(0.5225160000, 0.1717179761)--(0.5325160000, 0.1822840032)--(0.5425160000, 0.1929803935)--(0.5525160000, 0.2038064818)--(0.5625160000, 0.2147617489)--(0.5725160000, 0.2258458196)--(0.5825160000, 0.2370584613)--(0.5925160000, 0.2483995831)--(0.6025160000, 0.2598692360)--(0.6125160000, 0.2714676126)--(0.6225160000, 0.2831950485)--(0.6325160000, 0.2950520232)--(0.6425160000, 0.3070391622)--(0.6525160000, 0.3191572388)--(0.6625160000, 0.3314071769)--(0.6725160000, 0.3437900541)--(0.6825160000, 0.3563071040)--(0.6925160000, 0.3689597199)--(0.7025160000, 0.3817494577)--(0.7125160000, 0.3946780385)--(0.7225160000, 0.4077473507)--(0.7325160000, 0.4209594512)--(0.7425160000, 0.4343165647)--(0.7525160000, 0.4478210811)--(0.7625160000, 0.4614755494)--(0.7725160000, 0.4752826669)--(0.7825160000, 0.4892452627)--(0.7925160000, 0.5033662713)--(0.8025160000, 0.5176486963)--(0.8125160000, 0.5320955562)--(0.8225160000, 0.5467098111)--(0.8325160000, 0.5614942599)--(0.8425160000, 0.5764514011)--(0.8525160000, 0.5915832460)--(0.8625160000, 0.6068910698)--(0.8725160000, 0.6223750869)--(0.8825160000, 0.6380340332)--(0.8925160000, 0.6538646422)--(0.9025160000, 0.6698610107)--(0.9125160000, 0.6860138648)--(0.9225160000, 0.7023097711)--(0.9325160000, 0.7187303841)--(0.9425160000, 0.7352518792)--(0.9525160000, 0.7518447743)--(0.9625160000, 0.7684743567)--(0.9725160000, 0.7851018675)--(0.9825160000, 0.8016864291)--(0.9925160000, 0.8181874790)--(1.002516000, 0.8345672810)--(1.012516000, 0.8507930410)--(1.022516000, 0.8668382826)--(1.032516000, 0.8826833709)--(1.042516000, 0.8983152945)--(1.052516000, 0.9137269384)--(1.062516000, 0.9289161009)--(1.072516000, 0.9438844549)--(1.082516000, 0.9586365826)--(1.092516000, 0.9731791489)--(1.102516000, 0.9875202289)--(1.112516000, 1.001668785)--(1.122516000, 1.015634273)--(1.132516000, 1.029426347)--(1.142516000, 1.043054654)--(1.152516000, 1.056528688)--(1.162516000, 1.069857687)--(1.172516000, 1.083050574)--(1.182516000, 1.096115915)--(1.192516000, 1.109061899)--(1.202516000, 1.121896331)--(1.212516000, 1.134626633) node[right] {$(-3,0)$};
\draw[orange,thick] (0.2025160000, -0.08282948439)--(0.1925160000, -0.08713596354)--(0.1825160000, -0.09109004966)--(0.1725160000, -0.09466720439)--(0.1625160000, -0.09783948589)--(0.1525160000, -0.1005748645)--(0.1425160000, -0.1028363482)--(0.1325160000, -0.1045808493)--(0.1225160000, -0.1057576905)--(0.1125160000, -0.1063065974)--(0.1025160000, -0.1061549425)--(0.09251600000, -0.1052138628)--(0.08251600000, -0.1033726209)--(0.07251600000, -0.1004901121)--(0.06251600000, -0.09638148330)--(0.05251600000, -0.09079581999)--(0.04251600000, -0.08337607088)--(0.03251600000, -0.07357935415)--(0.02251600000, -0.06049274119)--(0.01251600000, -0.04228477713)--(0.002516000000, -0.01321851667);
\filldraw[brown] (0,0.0512932*8) circle (0.5pt);
\filldraw[blue] (0.212516,- 0.0097740*8) circle (0.5pt);
\draw[thick,orange,] (-0.05,0)--(0.05,0);
\draw[thick,orange] (0,-0.05)--(0,0.05);
\end{tikzpicture}
\end{tiny}
\caption{This figure depicts in \textcolor{orange}{orange} the set $\calC_{p_4,\vartheta}$ for $\vartheta=\arg(-1/\tau)=1.87141$, and $(A,B)=(4,1)$. The \textcolor{green}{green} lines are parallel to the lines $i/\tau$.}
\label{fig:FAKE_saddle.A4B1}
\end{figure}
Moving away from this apparent saddle, we plot the set $\calC_{p_4,\vartheta}$ in Figure~\ref{fig:after_FAKE_saddle.A4B1} for $\vartheta=1.85159$.

\begin{figure}[ht]
\begin{tiny}
\begin{tikzpicture}[scale=2]
\draw[<->,thick] (-2.5,0) -- (1.5,0);
\draw[<->,thick] (0,-1) -- (0,1);
\foreach \x in {-2,-1,0,1}
\draw[green] (-0.960835/8/0.277121+\x,-1)--(0.960835/8/0.277121+\x,1);
\draw[thick, red] (-0.05,-0.05)--(0.05,0.05);
\draw[thick, red] (0.05,-0.05)--(-0.05,0.05);
\draw[thick,red,xshift=-1cm] (-0.05,-0.05)--(0.05,0.05);
\draw[thick,red,xshift=-1cm] (0.05,-0.05)--(-0.05,0.05);
\draw[thick,red,xshift=-2cm] (-0.05,-0.05)--(0.05,0.05);
\draw[thick,red,xshift=-2cm] (0.05,-0.05)--(-0.05,0.05);
\draw[thick,red,xshift=1cm] (-0.05,-0.05)--(0.05,0.05);
\draw[thick,red,xshift=1cm] (0.05,-0.05)--(-0.05,0.05);
\draw[orange,thick,->] (0.2225160000, -0.07403834333)--(0.2325160000, -0.06959075409)--(0.2425160000, -0.06486487986)--(0.2525160000, -0.05987454944)--(0.2625160000, -0.05463224781)--(0.2725160000, -0.04914928232)--(0.2825160000, -0.04343592277)--(0.2925160000, -0.03750152039)--(0.3025160000, -0.03135460936)--(0.3125160000, -0.02500299386)--(0.3225160000, -0.01845382319)--(0.3325160000, -0.01171365670)--(0.3425160000, -0.004788520271)--(0.3525160000, 0.002316044467)--(0.3625160000, 0.009594937017)--(0.3725160000, 0.01704346022)--(0.3825160000, 0.02465728697)--(0.3925160000, 0.03243243089)--(0.4025160000, 0.04036522043)--(0.4125160000, 0.04845227584)--(0.4225160000, 0.05669048898)--(0.4325160000, 0.06507700517)--(0.4425160000, 0.07360920732)--(0.4525160000, 0.08228470169)--(0.4625160000, 0.09110130543)--(0.4725160000, 0.1000570354)--(0.4825160000, 0.1091500986)--(0.4925160000, 0.1183788832)--(0.5025160000, 0.1277419515)--(0.5125160000, 0.1372380334)--(0.5225160000, 0.1468660204)--(0.5325160000, 0.1566249614)--(0.5425160000, 0.1665140585)--(0.5525160000, 0.1765326642)--(0.5625160000, 0.1866802787)--(0.5725160000, 0.1969565481)--(0.5825160000, 0.2073612636)--(0.5925160000, 0.2178943612)--(0.6025160000, 0.2285559214)--(0.6125160000, 0.2393461705)--(0.6225160000, 0.2502654817)--(0.6325160000, 0.2613143773)--(0.6425160000, 0.2724935312)--(0.6525160000, 0.2838037720)--(0.6625160000, 0.2952460866)--(0.6725160000, 0.3068216242)--(0.6825160000, 0.3185317013)--(0.6925160000, 0.3303778058)--(0.7025160000, 0.3423616031)--(0.7125160000, 0.3544849406)--(0.7225160000, 0.3667498528)--(0.7325160000, 0.3791585664)--(0.7425160000, 0.3917135039)--(0.7525160000, 0.4044172855)--(0.7625160000, 0.4172727297)--(0.7725160000, 0.4302828495)--(0.7825160000, 0.4434508443)--(0.7925160000, 0.4567800846)--(0.8025160000, 0.4702740864)--(0.8125160000, 0.4839364727)--(0.8225160000, 0.4977709150)--(0.8325160000, 0.5117810495)--(0.8425160000, 0.5259703571)--(0.8525160000, 0.5403419963)--(0.8625160000, 0.5548985714)--(0.8725160000, 0.5696418188)--(0.8825160000, 0.5845721856)--(0.8925160000, 0.5996882780)--(0.9025160000, 0.6149861580)--(0.9125160000, 0.6304584812)--(0.9225160000, 0.6460934987)--(0.9325160000, 0.6618740025)--(0.9425160000, 0.6777763758)--(0.9525160000, 0.6937700070)--(0.9625160000, 0.7098173966)--(0.9725160000, 0.7258752566)--(0.9825160000, 0.7418967135)--(0.9925160000, 0.7578343846)--(1.002516000, 0.7736437495)--(1.012516000, 0.7892860916)--(1.022516000, 0.8047304347)--(1.032516000, 0.8199542712)--(1.042516000, 0.8349432362)--(1.052516000, 0.8496900890)--(1.062516000, 0.8641933779)--(1.072516000, 0.8784560773)--(1.082516000, 0.8924843649)--(1.092516000, 0.9062866094)--(1.102516000, 0.9198725786)--(1.112516000, 0.9332528435)--(1.122516000, 0.9464383453)--(1.132516000, 0.9594400884)--(1.142516000, 0.9722689297)--(1.152516000, 0.9849354385)--(1.162516000, 0.9974498077)--(1.172516000, 1.009821801)--(1.182516000, 1.022060727)--(1.192516000, 1.034175430)--(1.202516000, 1.046174291)--(1.212516000, 1.058065244) node[right] {$(-3,0)$};
\draw[orange,thick,->] (0.2025160000, -0.08203540376)--(0.1925160000, -0.08554850470)--(0.1825160000, -0.08870993936)--(0.1725160000, -0.09149521207)--(0.1625160000, -0.09387643034)--(0.1525160000, -0.09582162230)--(0.1425160000, -0.09729386493)--(0.1325160000, -0.09825015434)--(0.1225160000, -0.09863991735)--(0.1125160000, -0.09840301309)--(0.1025160000, -0.09746699094)--(0.09251600000, -0.09574323108)--(0.08251600000, -0.09312134696)--(0.07251600000, -0.08946076810)--(0.06251600000, -0.08457750966)--(0.05251600000, -0.07822218452)--(0.04251600000, -0.07004071319)--(0.03251600000, -0.05949683693)--(0.02251600000, -0.04569571805)--(0.01251600000, -0.02687723372)--(0.002516000000, 0.001801860594)--(-0.007484000000, 0.02216148560)--(-0.01748400000, 0.02406287056)--(-0.02748400000, 0.01973984130)--(-0.03748400000, 0.01166417041)--(-0.04748400000, 0.0009870574582)--(-0.05748400000, -0.01160991022)--(-0.06748400000, -0.02567377689)--(-0.07748400000, -0.04088231655)--(-0.08748400000, -0.05699584640)--(-0.09748400000, -0.07383041880)--(-0.1074840000, -0.09124157538)--(-0.1174840000, -0.1091138617)--(-0.1274840000, -0.1273537258)--(-0.1374840000, -0.1458845233)--(-0.1474840000, -0.1646428956)--(-0.1574840000, -0.1835760775)--(-0.1674840000, -0.2026398547)--(-0.1774840000, -0.2217969863)--(-0.1874840000, -0.2410159694)--(-0.1974840000, -0.2602700585)--(-0.2074840000, -0.2795364786)--(-0.2174840000, -0.2987957881)--(-0.2274840000, -0.3180313585)--(-0.2374840000, -0.3372289453)--(-0.2474840000, -0.3563763336)--(-0.2574840000, -0.3754630415)--(-0.2674840000, -0.3944800721)--(-0.2774840000, -0.4134197048)--(-0.2874840000, -0.4322753175)--(-0.2974840000, -0.4510412371)--(-0.3074840000, -0.4697126106)--(-0.3174840000, -0.4882852953)--(-0.3274840000, -0.5067557649)--(-0.3374840000, -0.5251210276)--(-0.3474840000, -0.5433785562)--(-0.3574840000, -0.5615262274)--(-0.3674840000, -0.5795622684)--(-0.3774840000, -0.5974852114)--(-0.3874840000, -0.6152938535)--(-0.3974840000, -0.6329872210)--(-0.4074840000, -0.6505645396)--(-0.4174840000, -0.6680252067)--(-0.4274840000, -0.6853687683)--(-0.4374840000, -0.7025948974)--(-0.4474840000, -0.7197033768)--(-0.4574840000, -0.7366940816)--(-0.4674840000, -0.7535669660)--(-0.4774840000, -0.7703220500)--(-0.4874840000, -0.7869594085)--(-0.4974840000, -0.8034791612)--(-0.5074840000, -0.8198814643)--(-0.5174840000, -0.8361665022)--(-0.5274840000, -0.8523344812)--(-0.5374840000, -0.8683856233)--(-0.5474840000, -0.8843201608)--(-0.5574840000, -0.9001383322)--(-0.5674840000, -0.9158403779)--(-0.5774840000, -0.9314265370)--(-0.5874840000, -0.9468970445)--(-0.5974840000, -0.9622521291)--(-0.6074840000, -0.9774920116)--(-0.6174840000, -0.9926169037)node[right] {$(-4,0)$};
\draw[magenta,thick] (-0.4900000000, -0.6232432934)--(-0.4800000000, -0.6068781913)--(-0.4700000000, -0.5903410822)--(-0.4600000000, -0.5736289234)--(-0.4500000000, -0.5567385108)--(-0.4400000000, -0.5396664753)--(-0.4300000000, -0.5224092798)--(-0.4200000000, -0.5049632160)--(-0.4100000000, -0.4873244031)--(-0.4000000000, -0.4694887858)--(-0.3900000000, -0.4514521350)--(-0.3800000000, -0.4332100489)--(-0.3700000000, -0.4147579564)--(-0.3600000000, -0.3960911237)--(-0.3500000000, -0.3772046632)--(-0.3400000000, -0.3580935466)--(-0.3300000000, -0.3387526244)--(-0.3200000000, -0.3191766507)--(-0.3100000000, -0.2993603173)--(-0.3000000000, -0.2792982979)--(-0.2900000000, -0.2589853071)--(-0.2800000000, -0.2384161753)--(-0.2700000000, -0.2175859469)--(-0.2600000000, -0.1964900059)--(-0.2500000000, -0.1751242384)--(-0.2400000000, -0.1534852409)--(-0.2300000000, -0.1315705886)--(-0.2200000000, -0.1093791824)--(-0.2100000000, -0.08691169588)--(-0.2000000000, -0.06417115559)--(-0.1900000000, -0.04116369643)--(-0.1800000000, -0.01789954901)--(-0.1700000000, 0.005605661938)--(-0.1600000000, 0.02932919915)--(-0.1500000000, 0.05323892146)--(-0.1400000000, 0.07729020072)--(-0.1300000000, 0.1014217291)--(-0.1200000000, 0.1255497428)--(-0.1100000000, 0.1495599338)--(-0.1000000000, 0.1732958822)--(-0.09000000000, 0.1965420414)--(-0.08000000000, 0.2189977167)--(-0.07000000000, 0.2402350357)--(-0.06000000000, 0.2596256700)--(-0.05000000000, 0.2761991247)--(-0.04000000000, 0.2883297458)--(-0.03000000000, 0.2929252230)--(-0.02000000000, 0.2828947942)--(-0.01000000000, 0.2375848532)--(0, 0.08364163774)--(0.001855021165, 0.008000000000)--(0.001793269547, 0.01600000000)--(0.001676846075, 0.02400000000)--(0.001522073587, 0.03200000000);
\draw[magenta,thick,->] (-0.4900000000, -0.6232432934)--(-0.5100000000, -0.6554688619)--(-0.5200000000, -0.6713346405)--(-0.5300000000, -0.6870390439)--(-0.5400000000, -0.7025843650)--(-0.5500000000, -0.7179727585)--(-0.5600000000, -0.7332062444)--(-0.5700000000, -0.7482867122)--(-0.5800000000, -0.7632159239)--(-0.5900000000, -0.7779955183)--(-0.6000000000, -0.7926270141)--(-0.6100000000, -0.8071118136)--(-0.6200000000, -0.8214512063)--(-0.6300000000, -0.8356463724)--(-0.6400000000, -0.8496983868)--(-0.6500000000, -0.8636082230)--(-0.6600000000, -0.8773767573)--(-0.6700000000, -0.8910047737)--(-0.6800000000, -0.9044929688)--(-0.6900000000, -0.9178419581)--(-0.7000000000, -0.9310522822)--(-0.7100000000, -0.9441244150)--(-0.7200000000, -0.9570587729)--(-0.7300000000, -0.9698557253)--(-0.7400000000, -0.9825156075)--(-0.7500000000, -0.9950387364) node[left] {$(-1,0)$};
\filldraw[brown] (0,0.0512932*8) circle (0.5pt);
\filldraw[blue] (0.212516,- 0.0097740*8) circle (0.5pt);
\end{tikzpicture}
\end{tiny}
\caption{This figure depicts in \textcolor{orange}{orange} the set $\calC_{p_4,\vartheta}$ for $\vartheta=\arg(-1/\tau)=1.85159$, and $(A,B)=(4,1)$. The \textcolor{green}{green} lines are parallel to the lines $i/\tau$. The \textcolor{magenta}{magenta} curve represents the tail at the intersection of the set $\calC_{p_4,\vartheta}$ with the green lines.}
\label{fig:after_FAKE_saddle.A4B1}
\end{figure}
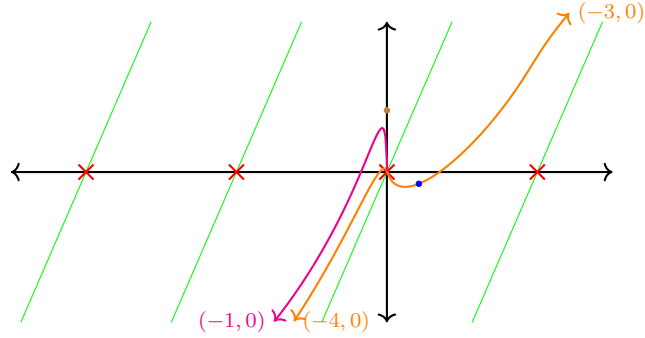
\noindent We can decompose the thimble $\calC_{p_4,\vartheta}$ in terms of state integral contours. We have contribution of $I_{-3,0}$ from the intersection of the orange contour with the interval $(-1,0)$, which has $(n,m,\ell)=(0,-3,0)$. Then, at the second iteration of the algorithm we get the contribution of $-I_{-2,0}$ from the intersection of the magenta curve with the interval $(-1,0)$, which has $(n,m,\ell)=(0,-2,0)$. Therefore, we find 
\be\label{eq:region_Is_A4B1}
\calI_{p_4,V_4,I_\star}=\calI_{-3,0}-\calI_{-2,0}=\calI_{-3,1}
\ee
where in the last equality follows from Equation~\eqref{eq:state_integral_1}. Notice that the set $\calC_{p_4,\vartheta}$ intersects the interval $(0,1)$ twice, thus there is no contribution coming from the state integral $\calI_{-3,1}$. 
\subsubsection{The saddle connection between $\mathrm{I}$ and $\mathrm{II}$}
The saddle connection is depicted in Figure~\ref{fig:saddle.A4B1.I-II}.
\begin{figure}[ht]
\begin{tiny}
\begin{tikzpicture}[scale=2]
\draw[<->,thick] (-2.5,0) -- (1.5,0);
\draw[<->,thick] (0,-1) -- (0,1);
\foreach \x in {-2,-1,0,1}
\draw[green] (-0.966069/8/0.258285+\x,-1)--(0.966069/8/0.258285+\x,1);
\draw[thick,red] (-0.05,-0.05)--(0.05,0.05);
\draw[thick,red] (0.05,-0.05)--(-0.05,0.05);
\draw[thick,red,xshift=-1cm] (-0.05,-0.05)--(0.05,0.05);
\draw[thick,red,xshift=-1cm] (0.05,-0.05)--(-0.05,0.05);
\draw[thick,red,xshift=-2cm] (-0.05,-0.05)--(0.05,0.05);
\draw[thick,red,xshift=-2cm] (0.05,-0.05)--(-0.05,0.05);
\draw[thick,red,xshift=1cm] (-0.05,-0.05)--(0.05,0.05);
\draw[thick,red,xshift=1cm] (0.05,-0.05)--(-0.05,0.05);
\draw[orange,thick,->] (0.212516,-0.009774*8)--(0.22252, -0.074822)--(0.23252, -0.071159)--(0.24252, -0.067218)--(0.25252, -0.063013)--(0.26252, -0.058557)--(0.27252, -0.053860)--(0.28252, -0.048934)--(0.29252, -0.043787)--(0.30252, -0.038428)--(0.31252, -0.032864)--(0.32252, -0.027104)--(0.33252, -0.021153)--(0.34252, -0.015017)--(0.35252, -0.0087020)--(0.36252, -0.0022132)--(0.37252, 0.0044449)--(0.38252, 0.011268)--(0.39252, 0.018252)--(0.40252, 0.025393)--(0.41252, 0.032689)--(0.42252, 0.040135)--(0.43252, 0.047729)--(0.44252, 0.055468)--(0.45252, 0.063351)--(0.46252, 0.071374)--(0.47252, 0.079537)--(0.48252, 0.087836)--(0.49252, 0.096271)--(0.50252, 0.10484)--(0.51252, 0.11354)--(0.52252, 0.12237)--(0.53252, 0.13134)--(0.54252, 0.14043)--(0.55252, 0.14966)--(0.56252, 0.15901)--(0.57252, 0.16849)--(0.58252, 0.17810)--(0.59252, 0.18783)--(0.60252, 0.19770)--(0.61252, 0.20769)--(0.62252, 0.21782)--(0.63252, 0.22807)--(0.64252, 0.23845)--(0.65252, 0.24897)--(0.66252, 0.25962)--(0.67252, 0.27040)--(0.68252, 0.28131)--(0.69252, 0.29236)--(0.70252, 0.30355)--(0.71252, 0.31488)--(0.72252, 0.32636)--(0.73252, 0.33797)--(0.74252, 0.34974)--(0.75252, 0.36165)--(0.76252, 0.37372)--(0.77252, 0.38594)--(0.78252, 0.39833)--(0.79252, 0.41088)--(0.80252, 0.42359)--(0.81252, 0.43648)--(0.82252, 0.44954)--(0.83252, 0.46279)--(0.84252, 0.47622)--(0.85252, 0.48984)--(0.86252, 0.50366)--(0.87252, 0.51767)--(0.88252, 0.53188)--(0.89252, 0.54629)--(0.90252, 0.56090)--(0.91252, 0.57570)--(0.92252, 0.59069)--(0.93252, 0.60584)--(0.94252, 0.62114)--(0.95252, 0.63656)--(0.96252, 0.65205)--(0.97252, 0.66757)--(0.98252, 0.68306)--(0.99252, 0.69847)--(1.0025, 0.71374)--(1.0125, 0.72883)--(1.0225, 0.74370)--(1.0325, 0.75833)--(1.0425, 0.77270)--(1.0525, 0.78680)--(1.0625, 0.80063)--(1.0725, 0.81420)--(1.0825, 0.82751)--(1.0925, 0.84059)--(1.1025, 0.85343)--(1.1125, 0.86605)--(1.1225, 0.87846)--(1.1325, 0.89068)--(1.1425, 0.90271)--(1.1525, 0.91458)--(1.1625, 0.92629)--(1.1725, 0.93785)--(1.1825, 0.94927)--(1.1925, 0.96056)--(1.2025, 0.97173)--(1.2125, 0.98279) node[right] {$(-3,0)$};
\draw[orange,thick,->] (0.212516,-0.009774*8)--(0.20252, -0.081252)--(0.19252, -0.083983)--(0.18252, -0.086362)--(0.17252, -0.088366)--(0.16252, -0.089967)--(0.15252, -0.091132)--(0.14252, -0.091825)--(0.13252, -0.092004)--(0.12252, -0.091617)--(0.11252, -0.090605)--(0.10252, -0.088895)--(0.092516, -0.086400)--(0.082516, -0.083009)--(0.072516, -0.078582)--(0.062516, -0.072938)--(0.052516, -0.065830)--(0.042516, -0.056909)--(0.032516, -0.045650)--(0.022516, -0.031196)--(0.012516, -0.011927)--(-0.0060000, 0.032361)--(-0.016000, 0.036853)--(-0.026000, 0.034511)--(-0.036000, 0.028157)--(-0.046000, 0.019052)--(-0.056000, 0.0079286)--(-0.066000, -0.0047344)--(-0.076000, -0.018598)--(-0.086000, -0.033412)--(-0.096000, -0.048985)--(-0.10600, -0.065166)--(-0.11600, -0.081837)--(-0.12600, -0.098899)--(-0.13600, -0.11627)--(-0.14600, -0.13390)--(-0.15600, -0.15171)--(-0.16600, -0.16967)--(-0.17600, -0.18774)--(-0.18600, -0.20589)--(-0.19600, -0.22408)--(-0.20600, -0.24230)--(-0.21600, -0.26052)--(-0.22600, -0.27873)--(-0.23600, -0.29691)--(-0.24600, -0.31505)--(-0.25600, -0.33314)--(-0.26600, -0.35116)--(-0.27600, -0.36912)--(-0.28600, -0.38700)--(-0.29600, -0.40479)--(-0.30600, -0.42250)--(-0.31600, -0.44012)--(-0.32600, -0.45763)--(-0.33600, -0.47505)--(-0.34600, -0.49237)--(-0.35600, -0.50958)--(-0.36600, -0.52668)--(-0.37600, -0.54368)--(-0.38600, -0.56056)--(-0.39600, -0.57734)--(-0.40600, -0.59400)--(-0.41600, -0.61055)--(-0.42600, -0.62698)--(-0.43600, -0.64331)--(-0.44600, -0.65951)--(-0.45600, -0.67560)--(-0.46600, -0.69158)--(-0.47600, -0.70744)--(-0.48600, -0.72319)--(-0.49600, -0.73882)--(-0.50600, -0.75434)--(-0.51600, -0.76974)--(-0.52600, -0.78503)--(-0.53600, -0.80020)--(-0.54600, -0.81526)--(-0.55600, -0.83020)--(-0.56600, -0.84503)--(-0.57600, -0.85974)--(-0.58600, -0.87434)--(-0.59600, -0.88882)--(-0.60600, -0.90320)--(-0.61600, -0.91745)--(-0.62600, -0.93160)--(-0.63600, -0.94562)--(-0.64600, -0.95954)--(-0.65600, -0.97334)--(-0.66600, -0.98703)--(-0.67600, -1.0006)--(-0.68600, -1.0141)--(-0.69600, -1.0274)--(-0.70600, -1.0406) node[right] {$(-4,0)$};
\draw[magenta, thick,-] (-0.00019916, 0.42000)--(0.0011425, 0.34000)--(0.0023724, 0.26000)--(0.0034922, 0.18000)--(0.0043780, 0.10000)--(0.0045469, 0.020000);
\draw[magenta, thick,->] (0.010000, 0.41962)--(0.020000, 0.42676)--(0.030000, 0.43203)--(0.040000, 0.43578)--(0.050000, 0.43837)--(0.060000, 0.44007)--(0.070000, 0.44114)--(0.080000, 0.44177)--(0.090000, 0.44211)--(0.10000, 0.44227)--(0.11000, 0.44235)--(0.12000, 0.44242)--(0.13000, 0.44254)--(0.14000, 0.44275)--(0.15000, 0.44308)--(0.16000, 0.44355)--(0.17000, 0.44419)--(0.18000, 0.44500)--(0.19000, 0.44600)--(0.20000, 0.44719)--(0.21000, 0.44859)--(0.22000, 0.45018)--(0.23000, 0.45198)--(0.24000, 0.45398)--(0.25000, 0.45618)--(0.26000, 0.45858)--(0.27000, 0.46118)--(0.28000, 0.46398)--(0.29000, 0.46697)--(0.30000, 0.47016)--(0.31000, 0.47353)--(0.32000, 0.47709)--(0.33000, 0.48084)--(0.34000, 0.48476)--(0.35000, 0.48886)--(0.36000, 0.49314)--(0.37000, 0.49759)--(0.38000, 0.50220)--(0.39000, 0.50699)--(0.40000, 0.51193)--(0.41000, 0.51704)--(0.42000, 0.52230)--(0.43000, 0.52773)--(0.44000, 0.53330)--(0.45000, 0.53903)--(0.46000, 0.54490)--(0.47000, 0.55093)--(0.48000, 0.55710)--(0.49000, 0.56341)--(0.50000, 0.56987)--(0.51000, 0.57647)--(0.52000, 0.58321)--(0.53000, 0.59008)--(0.54000, 0.59709)--(0.55000, 0.60424)--(0.56000, 0.61153)--(0.57000, 0.61894)--(0.58000, 0.62649)--(0.59000, 0.63418)--(0.60000, 0.64199)--(0.61000, 0.64994)--(0.62000, 0.65801)--(0.63000, 0.66622)--(0.64000, 0.67456)--(0.65000, 0.68303)--(0.66000, 0.69162)--(0.67000, 0.70035)--(0.68000, 0.70921)--(0.69000, 0.71819)--(0.70000, 0.72731)--(0.71000, 0.73656)--(0.72000, 0.74594)--(0.73000, 0.75545)--(0.74000, 0.76509)--(0.75000, 0.77487)--(0.76000, 0.78478)--(0.77000, 0.79482)--(0.78000, 0.80500)--(0.79000, 0.81532)--(0.80000, 0.82577)--(0.81000, 0.83636)--(0.82000, 0.84708)--(0.83000, 0.85794)--(0.84000, 0.86894)--(0.85000, 0.88007)--(0.86000, 0.89134)--(0.87000, 0.90274)--(0.88000, 0.91426)--(0.89000, 0.92591)--(0.90000, 0.93768)--(0.91000, 0.94956)--(0.92000, 0.96154)--(0.93000, 0.97362)--(0.94000, 0.98577)--(0.95000, 0.99798) node[left] {$(-2,0)$};
\draw[magenta,thick,->] (-0.010000, 0.39863)--(-0.020000, 0.38433)--(-0.030000, 0.36762)--(-0.040000, 0.34878)--(-0.050000, 0.32815)--(-0.060000, 0.30611)--(-0.070000, 0.28304)--(-0.080000, 0.25924)--(-0.090000, 0.23496)--(-0.10000, 0.21041)--(-0.11000, 0.18575)--(-0.12000, 0.16109)--(-0.13000, 0.13653)--(-0.14000, 0.11214)--(-0.15000, 0.087966)--(-0.16000, 0.064041)--(-0.17000, 0.040393)--(-0.18000, 0.017040)--(-0.19000, -0.0060086)--(-0.20000, -0.028746)--(-0.21000, -0.051169)--(-0.22000, -0.073279)--(-0.23000, -0.095078)--(-0.24000, -0.11657)--(-0.25000, -0.13776)--(-0.26000, -0.15865)--(-0.27000, -0.17924)--(-0.28000, -0.19956)--(-0.29000, -0.21959)--(-0.30000, -0.23935)--(-0.31000, -0.25885)--(-0.32000, -0.27809)--(-0.33000, -0.29708)--(-0.34000, -0.31582)--(-0.35000, -0.33432)--(-0.36000, -0.35258)--(-0.37000, -0.37062)--(-0.38000, -0.38844)--(-0.39000, -0.40603)--(-0.40000, -0.42342)--(-0.41000, -0.44060)--(-0.42000, -0.45757)--(-0.43000, -0.47435)--(-0.44000, -0.49093)--(-0.45000, -0.50733)--(-0.46000, -0.52354)--(-0.47000, -0.53956)--(-0.48000, -0.55540)--(-0.49000, -0.57107)--(-0.50000, -0.58657)--(-0.51000, -0.60190)--(-0.52000, -0.61706)--(-0.53000, -0.63205)--(-0.54000, -0.64688)--(-0.55000, -0.66155)--(-0.56000, -0.67607)--(-0.57000, -0.69043)--(-0.58000, -0.70463)--(-0.59000, -0.71868)--(-0.60000, -0.73259)--(-0.61000, -0.74634)--(-0.62000, -0.75994)--(-0.63000, -0.77340)--(-0.64000, -0.78671)--(-0.65000, -0.79988)--(-0.66000, -0.81290)--(-0.67000, -0.82578)--(-0.68000, -0.83851)--(-0.69000, -0.85111)--(-0.70000, -0.86356)--(-0.71000, -0.87587)--(-0.72000, -0.88804)--(-0.73000, -0.90007)--(-0.74000, -0.91196)--(-0.75000, -0.92371)--(-0.76000, -0.93532)--(-0.77000, -0.94679)--(-0.78000, -0.95811)--(-0.79000, -0.96930)--(-0.80000, -0.98035)--(-0.81000, -0.99125) node[left] {$(-3,0)$};
\filldraw[brown] (0,0.0512932*8) circle (0.5pt);
\filldraw[blue] (0.212516,- 0.0097740*8) circle (0.5pt);
\end{tikzpicture}
\end{tiny}
\caption{This figure depicts in \textcolor{orange}{orange} the set $\calC_{p_4,\vartheta}$ for $\vartheta=\arg(-1/\tau)=1.83204$, and $(A,B)=(4,1)$. The \textcolor{green}{green} lines are parallel to the lines $i/\tau$. The \textcolor{magenta}{magenta} curve represents the tail at the intersection of the set $\calC_{p_4,\vartheta}$ with the green lines.}
\label{fig:saddle.A4B1.I-II}
\end{figure}
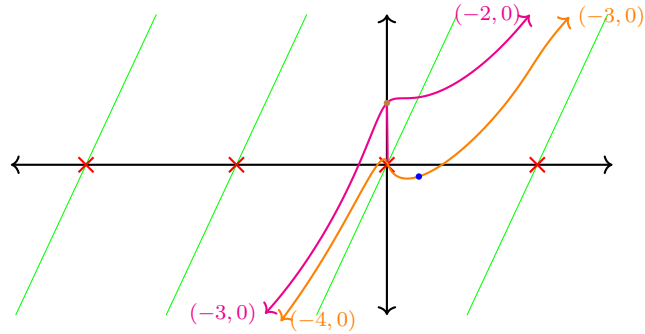
Notice that the magenta line hits the critical point $p_2$. This suggests that in the state integral decomposition we should see the contribution from $\calI_{-2,0}$, as we indeed verify in Section~\ref{sec:region_II}.

\subsubsection{State integrals decomposition in $\mathrm{II}$}\label{sec:region_II}
Let $\vartheta\in\mathrm{II}$. We plot the curve $\calC_{p_4,\vartheta}$ in Figure~\ref{fig:thimb.A4B1.II}.
We can decompose the thimble in terms of state integrals contours. We have contributions of $\calI_{-3,0}$ from the intersection of the orange contour with the interval $(-1,0)$, which has $(n,m,\ell)=(0,-3,0)$. Then, due to the intersection with the green lines we should run the algorithm another time for the magenta contour. The latter intersects a green line, thus we apply the algorithm one more time to produce the violet contours. Both the magenta and the violet contours do not contribute as they do not intersect the reals~\ref{cor:half.thimb.stateint.zero}. Summarising, the algorithm gives
\be\label{eq:region_II_A4B1}
\calI_{p_4,V_4,\mathrm{II}}
\=
\calI_{-3,0}=\calI_{-3,1}+\calI_{-2,0}\,,
\ee
where the second equality follows from Equation~\eqref{eq:state_integral_1}.
\begin{figure}[ht]
\begin{tiny}
\begin{tikzpicture}[scale=2]
\draw[<->,thick] (-2.5,0) -- (1.5,0);
\draw[<->,thick] (0,-1) -- (0,1);
\foreach \x in {-1,0}
\draw[green] (-0.98545/8/0.169967+\x,-1)--(0.98545/8/0.169967+\x,1);
\draw[green] (-0.98545/8/0.169967+1,-1)--(1.5,0.5*8*0.169967/0.98545);
\draw[green] (-0.98545/8/0.169967+2,-1)--(1.5,-0.5*8*0.169967/0.98545);
\draw[green] (-2.5,-0.5*8*0.169967/0.98545)--(0.98545/8/0.169967-2,1);
\draw[thick,red] (-0.05,-0.05)--(0.05,0.05);
\draw[thick,red] (0.05,-0.05)--(-0.05,0.05);
\draw[thick,red,xshift=-1cm] (-0.05,-0.05)--(0.05,0.05);
\draw[thick,red,xshift=-1cm] (0.05,-0.05)--(-0.05,0.05);
\draw[thick,red,xshift=-2cm] (-0.05,-0.05)--(0.05,0.05);
\draw[thick,red,xshift=-2cm] (0.05,-0.05)--(-0.05,0.05);
\draw[thick,red,xshift=1cm] (-0.05,-0.05)--(0.05,0.05);
\draw[thick,red,xshift=1cm] (0.05,-0.05)--(-0.05,0.05);
\draw[orange,thick,->] (0.2225160000, -0.07844283940)--(0.2325160000, -0.07840218376)--(0.2425160000, -0.07808547578)--(0.2525160000, -0.07750643293)--(0.2625160000, -0.07667744544)--(0.2725160000, -0.07560973907)--(0.2825160000, -0.07431351261)--(0.2925160000, -0.07279805484)--(0.3025160000, -0.07107184447)--(0.3125160000, -0.06914263598)--(0.3225160000, -0.06701753368)--(0.3325160000, -0.06470305583)--(0.3425160000, -0.06220519048)--(0.3525160000, -0.05952944411)--(0.3625160000, -0.05668088423)--(0.3725160000, -0.05366417679)--(0.3825160000, -0.05048361905)--(0.3925160000, -0.04714316860)--(0.4025160000, -0.04364646896)--(0.4125160000, -0.03999687218)--(0.4225160000, -0.03619745893)--(0.4325160000, -0.03225105618)--(0.4425160000, -0.02816025287)--(0.4525160000, -0.02392741378)--(0.4625160000, -0.01955469171)--(0.4725160000, -0.01504403824)--(0.4825160000, -0.01039721308)--(0.4925160000, -0.005615792233)--(0.5025160000, -0.0007011749930)--(0.5125160000, 0.004345410085)--(0.5225160000, 0.009522900226)--(0.5325160000, 0.01483039441)--(0.5425160000, 0.02026715025)--(0.5525160000, 0.02583258172)--(0.5625160000, 0.03152625775)--(0.5725160000, 0.03734790165)--(0.5825160000, 0.04329739139)--(0.5925160000, 0.04937476075)--(0.6025160000, 0.05558020126)--(0.6125160000, 0.06191406514)--(0.6225160000, 0.06837686920)--(0.6325160000, 0.07496929964)--(0.6425160000, 0.08169221802)--(0.6525160000, 0.08854666840)--(0.6625160000, 0.09553388565)--(0.6725160000, 0.1026553053)--(0.6825160000, 0.1099125746)--(0.6925160000, 0.1173075660)--(0.7025160000, 0.1248423911)--(0.7125160000, 0.1325194186)--(0.7225160000, 0.1403412928)--(0.7325160000, 0.1483109558)--(0.7425160000, 0.1564316726)--(0.7525160000, 0.1647070586)--(0.7625160000, 0.1731411119)--(0.7725160000, 0.1817382491)--(0.7825160000, 0.1905033455)--(0.7925160000, 0.1994417803)--(0.8025160000, 0.2085594867)--(0.8125160000, 0.2178630075)--(0.8225160000, 0.2273595542)--(0.8325160000, 0.2370570704)--(0.8425160000, 0.2469642952)--(0.8525160000, 0.2570908213)--(0.8625160000, 0.2674471400)--(0.8725160000, 0.2780446540)--(0.8825160000, 0.2888956311)--(0.8925160000, 0.3000130485)--(0.9025160000, 0.3114102430)--(0.9125160000, 0.3231002258)--(0.9225160000, 0.3350944300)--(0.9325160000, 0.3474005226)--(0.9425160000, 0.3600187569)--(0.9525160000, 0.3729362620)--(0.9625160000, 0.3861190073)--(0.9725160000, 0.3995026298)--(0.9825160000, 0.4129864197)--(0.9925160000, 0.4264378371)--(1.002516000, 0.4397122314)--(1.012516000, 0.4526810511)--(1.022516000, 0.4652529838)--(1.032516000, 0.4773788053)--(1.042516000, 0.4890437236)--(1.052516000, 0.5002558375)--(1.062516000, 0.5110362163)--(1.072516000, 0.5214121116)--(1.082516000, 0.5314128629)--(1.092516000, 0.5410676327)--(1.102516000, 0.5504042543)--(1.112516000, 0.5594487160)--(1.122516000, 0.5682249953)--(1.132516000, 0.5767550743)--(1.142516000, 0.5850590470)--(1.152516000, 0.5931552636)--(1.162516000, 0.6010604874)--(1.172516000, 0.6087900486)--(1.182516000, 0.6163579881)--(1.192516000, 0.6237771895)--(1.202516000, 0.6310594971)--(1.212516000, 0.6382158223) node[right] {$(-3,0)$};
\draw[orange,thick,->] (0.2025160000, -0.07763328710)--(0.1925160000, -0.07674706300)--(0.1825160000, -0.07551213284)--(0.1725160000, -0.07390430427)--(0.1625160000, -0.07189607259)--(0.1525160000, -0.06945597050)--(0.1425160000, -0.06654774388)--(0.1325160000, -0.06312929380)--(0.1225160000, -0.05915129881)--(0.1125160000, -0.05455539246)--(0.1025160000, -0.04927170956)--(0.09251600000, -0.04321551772)--(0.08251600000, -0.03628249281)--(0.07251600000, -0.02834193909)--(0.06251600000, -0.01922683492)--(0.05251600000, -0.008718955770)--(0.04251600000, 0.003473263373)--(0.03251600000, 0.01774453162)--(0.02251600000, 0.03460725890)--(0.01251600000, 0.05446770841)--(0.002516000000, 0.07554749425)--(-0.007484000000, 0.09068074973)--(-0.01748400000, 0.09796876395)--(-0.02748400000, 0.09998671936)--(-0.03748400000, 0.09844340234)--(-0.04748400000, 0.09433118230)--(-0.05748400000, 0.08828085388)--(-0.06748400000, 0.08072479005)--(-0.07748400000, 0.07197568567)--(-0.08748400000, 0.06226857375)--(-0.09748400000, 0.05178523466)--(-0.1074840000, 0.04066936281)--(-0.1174840000, 0.02903650490)--(-0.1274840000, 0.01698085500)--(-0.1374840000, 0.004580063474)--(-0.1474840000, -0.008101261660)--(-0.1574840000, -0.02100894450)--(-0.1674840000, -0.03409725625)--(-0.1774840000, -0.04732737979)--(-0.1874840000, -0.06066620459)--(-0.1974840000, -0.07408536894)--(-0.2074840000, -0.08756049009)--(-0.2174840000, -0.1010705395)--(-0.2274840000, -0.1145973316)--(-0.2374840000, -0.1281251016)--(-0.2474840000, -0.1416401554)--(-0.2574840000, -0.1551305763)--(-0.2674840000, -0.1685859791)--(-0.2774840000, -0.1819973015)--(-0.2874840000, -0.1953566279)--(-0.2974840000, -0.2086570382)--(-0.3074840000, -0.2218924791)--(-0.3174840000, -0.2350576532)--(-0.3274840000, -0.2481479228)--(-0.3374840000, -0.2611592278)--(-0.3474840000, -0.2740880134)--(-0.3574840000, -0.2869311675)--(-0.3674840000, -0.2996859661)--(-0.3774840000, -0.3123500255)--(-0.3874840000, -0.3249212603)--(-0.3974840000, -0.3373978464)--(-0.4074840000, -0.3497781884)--(-0.4174840000, -0.3620608906)--(-0.4274840000, -0.3742447318)--(-0.4374840000, -0.3863286421)--(-0.4474840000, -0.3983116831)--(-0.4574840000, -0.4101930295)--(-0.4674840000, -0.4219719530)--(-0.4774840000, -0.4336478078)--(-0.4874840000, -0.4452200174)--(-0.4974840000, -0.4566880628)--(-0.5074840000, -0.4680514714)--(-0.5174840000, -0.4793098080)--(-0.5274840000, -0.4904626649)--(-0.5374840000, -0.5015096545)--(-0.5474840000, -0.5124504010)--(-0.5574840000, -0.5232845338)--(-0.5674840000, -0.5340116810)--(-0.5774840000, -0.5446314628)--(-0.5874840000, -0.5551434864)--(-0.5974840000, -0.5655473397)--(-0.6074840000, -0.5758425868)--(-0.6174840000, -0.5860287624)--(-0.6274840000, -0.5961053676)--(-0.6374840000, -0.6060718646)--(-0.6474840000, -0.6159276726)--(-0.6574840000, -0.6256721634)--(-0.6674840000, -0.6353046570)--(-0.6774840000, -0.6448244175)--(-0.6874840000, -0.6542306493)--(-0.6974840000, -0.6635224934)--(-0.7074840000, -0.6726990237)--(-0.7174840000, -0.6817592443)--(-0.7274840000, -0.6907020867)--(-0.7374840000, -0.6995264080)--(-0.7474840000, -0.7082309900)--(-0.7574840000, -0.7168145395)--(-0.7674840000, -0.7252756901)--(-0.7774840000, -0.7336130069)--(-0.7874840000, -0.7418249933)--(-0.7974840000, -0.7499101027)--(-0.8074840000, -0.7578667547)--(-0.8174840000, -0.7656933590)--(-0.8274840000, -0.7733883477)--(-0.8374840000, -0.7809502202)--(-0.8474840000, -0.7883776028)--(-0.8574840000, -0.7956693283)--(-0.8674840000, -0.8028245407)--(-0.8774840000, -0.8098428296)--(-0.8874840000, -0.8167244031)--(-0.8974840000, -0.8234703032)--(-0.9074840000, -0.8300826680)--(-0.9174840000, -0.8365650428)--(-0.9274840000, -0.8429227280)--(-0.9374840000, -0.8491631457)--(-0.9474840000, -0.8552961830)--(-0.9574840000, -0.8613344529)--(-0.9674840000, -0.8672933939)--(-0.9774840000, -0.8731911275)--(-0.9874840000, -0.8790480144)--(-0.9974840000, -0.8848859021)--(-1.010000000, -0.8921999091)--(-1.020000000, -0.8980757759)--(-1.030000000, -0.9040015130)--(-1.040000000, -0.9099940461)--(-1.050000000, -0.9160672501)--(-1.060000000, -0.9222318263)--(-1.070000000, -0.9284954341)--(-1.080000000, -0.9348629908)--(-1.090000000, -0.9413370606)--(-1.100000000, -0.9479182681)--(-1.110000000, -0.9546056932)--(-1.120000000, -0.9613972245)--(-1.130000000, -0.9682898596)--(-1.140000000, -0.9752799535)--(-1.150000000, -0.9823634172)--(-1.160000000, -0.9895358746)--(-1.170000000, -0.9967927835) node[left] {$(-4,0)$};
\draw[magenta,thick,->](0.02155852790, 0.01000000000)--(0.02245802909, 0.02000000000)--(0.02337971789, 0.03000000000)--(0.02432857291, 0.04000000000)--(0.02530987729, 0.05000000000)--(0.02632929659, 0.06000000000)--(0.02739297450, 0.07000000000)--(0.02850764951, 0.08000000000)--(0.02968079737, 0.09000000000)--(0.03092080632, 0.1000000000)--(0.03223719504, 0.1100000000)--(0.03364088735, 0.1200000000)--(0.03514456353, 0.1300000000)--(0.03676311689, 0.1400000000)--(0.03851425747, 0.1500000000)--(0.04041932548, 0.1600000000)--(0.04250441064, 0.1700000000)--(0.04480192895, 0.1800000000)--(0.04735290445, 0.1900000000)--(0.05021037557, 0.2000000000)--(0.05344467103, 0.2100000000)--(0.05715195092, 0.2200000000)--(0.06146880786, 0.2300000000)--(0.06659900794, 0.2400000000)--(0.07286709632, 0.2500000000)--(0.08084012081, 0.2600000000)--(0.09166042803, 0.2700000000)--(0.1082957490, 0.2800000000)--(0.1100000000, 0.2807588142)--(0.1200000000, 0.2844740892)--(0.1300000000, 0.2871726252)--(0.1400000000, 0.2891175524)--(0.1500000000, 0.2905042528)--(0.1600000000, 0.2914803276)--(0.1700000000, 0.2921588892)--(0.1800000000, 0.2926276584)--(0.1900000000, 0.2929553432)--(0.2000000000, 0.2931962049)--(0.2100000000, 0.2933933879)--(0.2200000000, 0.2935813864)--(0.2300000000, 0.2937878990)--(0.2400000000, 0.2940352396)--(0.2500000000, 0.2943414244)--(0.2600000000, 0.2947210162)--(0.2700000000, 0.2951857881)--(0.2800000000, 0.2957452478)--(0.2900000000, 0.2964070558)--(0.3000000000, 0.2971773616)--(0.3100000000, 0.2980610746)--(0.3200000000, 0.2990620843)--(0.3300000000, 0.3001834406)--(0.3400000000, 0.3014275012)--(0.3500000000, 0.3027960533)--(0.3600000000, 0.3042904141)--(0.3700000000, 0.3059115152)--(0.3800000000, 0.3076599715)--(0.3900000000, 0.3095361403)--(0.4000000000, 0.3115401700)--(0.4100000000, 0.3136720415)--(0.4200000000, 0.3159316033)--(0.4300000000, 0.3183186007)--(0.4400000000, 0.3208327015)--(0.4500000000, 0.3234735171)--(0.4600000000, 0.3262406213)--(0.4700000000, 0.3291335658)--(0.4800000000, 0.3321518941)--(0.4900000000, 0.3352951535)--(0.5000000000, 0.3385629054)--(0.5100000000, 0.3419547346)--(0.5200000000, 0.3454702574)--(0.5300000000, 0.3491091292)--(0.5400000000, 0.3528710506)--(0.5500000000, 0.3567557740)--(0.5600000000, 0.3607631093)--(0.5700000000, 0.3648929289)--(0.5800000000, 0.3691451732)--(0.5900000000, 0.3735198553)--(0.6000000000, 0.3780170662)--(0.6100000000, 0.3826369797)--(0.6200000000, 0.3873798571)--(0.6300000000, 0.3922460527)--(0.6400000000, 0.3972360189)--(0.6500000000, 0.4023503117)--(0.6600000000, 0.4075895967)--(0.6700000000, 0.4129546550)--(0.6800000000, 0.4184463896)--(0.6900000000, 0.4240658325)--(0.7000000000, 0.4298141516)--(0.7100000000, 0.4356926587)--(0.7200000000, 0.4417028169)--(0.7300000000, 0.4478462496)--(0.7400000000, 0.4541247480)--(0.7500000000, 0.4605402806)--(0.7600000000, 0.4670950007)--(0.7700000000, 0.4737912543)--(0.7800000000, 0.4806315863)--(0.7900000000, 0.4876187448)--(0.8000000000, 0.4947556820)--(0.8100000000, 0.5020455498)--(0.8200000000, 0.5094916883)--(0.8300000000, 0.5170976020)--(0.8400000000, 0.5248669204)--(0.8500000000, 0.5328033354)--(0.8600000000, 0.5409105035)--(0.8700000000, 0.5491919025)--(0.8800000000, 0.5576506204)--(0.8900000000, 0.5662890542)--(0.9000000000, 0.5751084863)--(0.9100000000, 0.5841085028)--(0.9200000000, 0.5932862191)--(0.9300000000, 0.6026352940)--(0.9400000000, 0.6121447541)--(0.9500000000, 0.6217977393)--(0.9600000000, 0.6315704134)--(0.9700000000, 0.6414314373)--(0.9800000000, 0.6513424896)--(0.9900000000, 0.6612601783)--(1.000000000, 0.6711392627)--(1.010000000, 0.6809365251)--(1.020000000, 0.6906142817)--(1.030000000, 0.7001426789)--(1.040000000, 0.7095004661)--(1.050000000, 0.7186744956)--(1.060000000, 0.7276584645)--(1.070000000, 0.7364513977)--(1.080000000, 0.7450562001)--(1.090000000, 0.7534784378)--(1.100000000, 0.7617253890)--(1.110000000, 0.7698053423)--(1.120000000, 0.7777270978)--(1.130000000, 0.7854996230)--(1.140000000, 0.7931318228)--(1.150000000, 0.8006323910)--(1.160000000, 0.8080097181)--(1.170000000, 0.8152718391)--(1.180000000, 0.8224264073)--(1.190000000, 0.8294806862)--(1.200000000, 0.8364415527)--(1.210000000, 0.8433155087)--(1.220000000, 0.8501086954)--(1.230000000, 0.8568269123)--(1.240000000, 0.8634756355)--(1.250000000, 0.8700600376)--(1.260000000, 0.8765850067)--(1.270000000, 0.8830551651)--(1.280000000, 0.8894748867)--(1.290000000, 0.8958483143)--(1.300000000, 0.9021793751)--(1.310000000, 0.9084717955)--(1.320000000, 0.9147291149)--(1.330000000, 0.9209546986)--(1.340000000, 0.9271517497)--(1.350000000, 0.9333233203)--(1.360000000, 0.9394723215)--(1.370000000, 0.9456015335)--(1.380000000, 0.9517136140)--(1.390000000, 0.9578111068)--(1.400000000, 0.9638964493)--(1.410000000, 0.9699719796)--(1.420000000, 0.9760399436)--(1.430000000, 0.9821025007)--(1.440000000, 0.9881617297)--(1.450000000, 0.9942196347) node[right] {$(-2,0)$};
\draw[violet,thick,->] (0.2100000000, 0.2915678695)--(0.2200000000, 0.2985029496)--(0.2300000000, 0.3052032895)--(0.2400000000, 0.3116952344)--(0.2500000000, 0.3180023364)--(0.2600000000, 0.3241457174)--(0.2700000000, 0.3301443754)--(0.2800000000, 0.3360154450)--(0.2900000000, 0.3417744208)--(0.3000000000, 0.3474353479)--(0.3100000000, 0.3530109883)--(0.3200000000, 0.3585129633)--(0.3300000000, 0.3639518786)--(0.3400000000, 0.3693374335)--(0.3500000000, 0.3746785159)--(0.3600000000, 0.3799832871)--(0.3700000000, 0.3852592556)--(0.3800000000, 0.3905133428)--(0.3900000000, 0.3957519415)--(0.4000000000, 0.4009809682)--(0.4100000000, 0.4062059089)--(0.4200000000, 0.4114318610)--(0.4300000000, 0.4166635708)--(0.4400000000, 0.4219054667)--(0.4500000000, 0.4271616900)--(0.4600000000, 0.4324361219)--(0.4700000000, 0.4377324087)--(0.4800000000, 0.4430539841)--(0.4900000000, 0.4484040904)--(0.5000000000, 0.4537857971)--(0.5100000000, 0.4592020182)--(0.5200000000, 0.4646555282)--(0.5300000000, 0.4701489771)--(0.5400000000, 0.4756849038)--(0.5500000000, 0.4812657489)--(0.5600000000, 0.4868938668)--(0.5700000000, 0.4925715364)--(0.5800000000, 0.4983009719)--(0.5900000000, 0.5040843327)--(0.6000000000, 0.5099237327)--(0.6100000000, 0.5158212493)--(0.6200000000, 0.5217789318)--(0.6300000000, 0.5277988102)--(0.6400000000, 0.5338829025)--(0.6500000000, 0.5400332227)--(0.6600000000, 0.5462517885)--(0.6700000000, 0.5525406281)--(0.6800000000, 0.5589017875)--(0.6900000000, 0.5653373374)--(0.7000000000, 0.5718493794)--(0.7100000000, 0.5784400525)--(0.7200000000, 0.5851115388)--(0.7300000000, 0.5918660689)--(0.7400000000, 0.5987059258)--(0.7500000000, 0.6056334491)--(0.7600000000, 0.6126510365)--(0.7700000000, 0.6197611442)--(0.7800000000, 0.6269662842)--(0.7900000000, 0.6342690184)--(0.8000000000, 0.6416719478)--(0.8100000000, 0.6491776951)--(0.8200000000, 0.6567888789)--(0.8300000000, 0.6645080762)--(0.8400000000, 0.6723377695)--(0.8500000000, 0.6802802740)--(0.8600000000, 0.6883376379)--(0.8700000000, 0.6965115094)--(0.8800000000, 0.7048029603)--(0.8900000000, 0.7132122578)--(0.9000000000, 0.7217385744)--(0.9100000000, 0.7303796315)--(0.9200000000, 0.7391312781)--(0.9300000000, 0.7479870243)--(0.9400000000, 0.7569375708)--(0.9500000000, 0.7659704114)--(0.9600000000, 0.7750696174)--(0.9700000000, 0.7842159322)--(0.9800000000, 0.7933872835)--(0.9900000000, 0.8025597452)--(1.000000000, 0.8117088587)--(1.010000000, 0.8208110998)--(1.020000000, 0.8298452213)--(1.030000000, 0.8387932382)--(1.040000000, 0.8476409432)--(1.050000000, 0.8563779691)--(1.060000000, 0.8649975098)--(1.070000000, 0.8734958371)--(1.080000000, 0.8818717416)--(1.090000000, 0.8901259817) node[above] {$(-1,0)$};
\filldraw[brown] (0,0.0512932*8) circle (0.5pt);
\filldraw[magenta] (0.02245802909, 0.02000000000) circle (0.5pt);
\filldraw[blue] (0.212516,- 0.0097740*8) circle (0.5pt);
\filldraw[violet] (0.2100000000, 0.2915678695) circle (0.5pt);
\end{tikzpicture}
\end{tiny}
\caption{This figure depicts in \textcolor{orange}{orange} the set $\calC_{p_4,\vartheta}$ for $\vartheta=\arg(-1/\tau)=1.74159$, and $(A,B)=(4,1)$. The \textcolor{green}{green} lines are parallel to the lines $i/\tau$. The \textcolor{magenta}{magenta} curve represents the tail at the intersection of the set $\calC_{p_4,\vartheta}$ with the green lines. The \textcolor{violet}{violet} contour represents half of the contour that appears when we run the algorithm a second time.}
\label{fig:thimb.A4B1.II}
\end{figure}
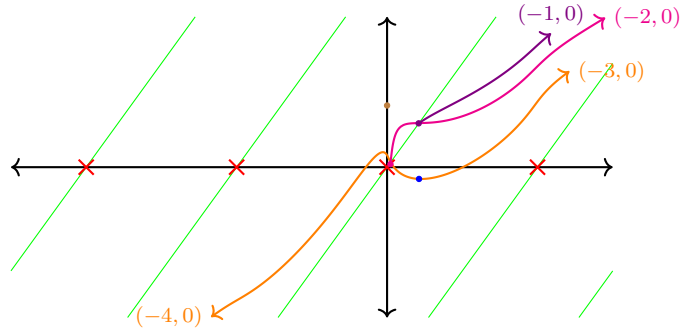
\subsubsection{The saddle connection between $\mathrm{II}$ and $\mathrm{III}$}
The saddle connection is depicted in Figure~\ref{fig:saddle.A4B1.II-III}. Note that it appears in the contours coming from the first iteration of the algorithm.
\begin{figure}[ht]
\begin{tiny}

\end{tiny}
\caption{This figure depicts in \textcolor{orange}{orange} the set $\calC_{p_1,\vartheta}$ for $\vartheta=\arg(-1/\tau)=1.64069$ and $(A,B)=(4,1)$. The \textcolor{green}{green} lines are parallel to the lines $i/\tau$. The \textcolor{magenta}{magenta} curves represent the tails at the intersection of the set $\calC_{p_4,\vartheta}$ with the green lines. The \textcolor{violet}{violet} contour represents half of the contour that appears when we run the algorithm a second time. Notice that the \textcolor{magenta}{magenta} contour splits at the critical point $\textcolor{blue}{p_3}$ as this is a saddle connection.}
\label{fig:saddle.A4B1.II-III}
\end{figure}

\subsubsection{State integrals decomposition in $\mathrm{III}$}
Let $\vartheta\in\mathrm{III}$. We plot the curve $\calC_{p_2,\vartheta}$ in Figure~\ref{fig:thimb.A4B1.III}.
\begin{figure}[ht]
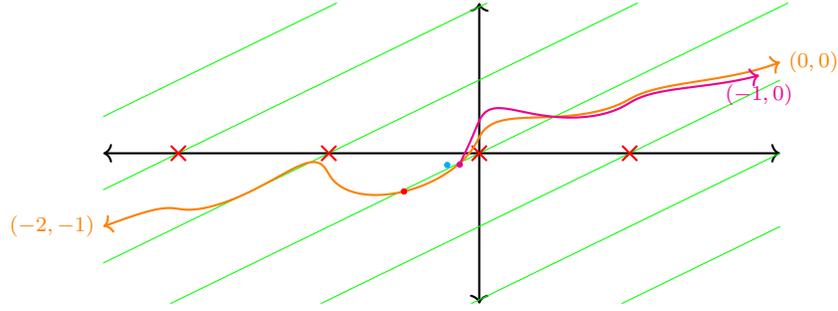

\begin{tiny}

\end{tiny}
\caption{This figure depicts in \textcolor{orange}{orange} the set $\calC_{p_1,\vartheta}$ for $\vartheta=\arg(-1/\tau)=1.63159$, and $(A,B)=(4,1)$. The \textcolor{green}{green} lines are parallel to the lines $i/\tau$. The \textcolor{magenta}{magenta} curve represents the tail at the intersection of the set $\calC_{p_4,\vartheta}$ with the green lines.}
\label{fig:thimb.A4B1.III}
\end{figure}
We can decompose the thimble in terms on state integrals contours. We have contributions from $\calI_{0,0}$ from the intersection of the orange contour with the interval $(-1,0)$, which has $(n,m,\ell)=(0,0,0)$. Then due to the intersection with the green line we should study the contributions from the magenta line. Since the latter crosses the reals in the interval $(-1,0)$, we get the contribution of $\calI_{-1,0}$, which has $(n,m,\ell)=(0,-1,0)$. Any other tails lead to no contribution as they do not cross the reals and head to $\infty$. The algorithm therefore gives
\be\label{eq:region_III_A4B1}
\calI_{p_1,V_1,\mathrm{III}}
\=
\calI_{0,0}-\calI_{-1,0}\,.
\ee
\subsubsection{The saddle connection between $\mathrm{III}$ and $\mathrm{IV}$}
The saddle connection is depicted in Figure~\ref{fig:saddle.A4B1.III-IV}.
\begin{figure}[ht]
\begin{tiny}

\end{tiny}
\caption{This figure depicts in \textcolor{orange}{orange} the set $\calC_{p_4,\vartheta}$ for $\vartheta=\arg(-1/\tau)=1.62047$, and $(A,B)=(4,1)$. The \textcolor{green}{green} lines are parallel to the lines $i/\tau$. The \textcolor{magenta}{magenta} curve represents the tail at the intersection of the set $\calC_{p_4,\vartheta}$ with the green lines. The \textcolor{violet}{violet} contour represents half of the contour that appears when we run the algorithm a second time. While it is hard to see at this scale, the violet contour hits the critical point $\textcolor{blue}{p_3+1}$, which is just above the green line. Therefore, at this critical angle the violet contour splits at this point.}
\label{fig:saddle.A4B1.III-IV}
\end{figure}

\subsubsection{State integrals decomposition in $\mathrm{IV}$}
Let $\vartheta\in\mathrm{IV}$. We plot the set $\calC_{p_4,\vartheta}$ in Figure~\ref{fig:thimb.A4B1.IV}.
\begin{figure}[ht]
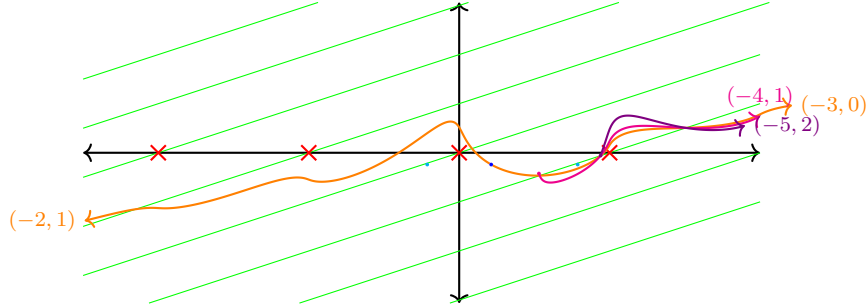

\begin{tiny}

\end{tiny}
\caption{This figure depicts in \textcolor{orange}{orange} the set $\calC_{p_4,\tau}$ for $\vartheta=\arg(-1/\tau)=1.61159$, and $(A,B)=(4,1)$. The \textcolor{green}{green} lines are parallel to the lines $i/\tau$. The \textcolor{magenta}{magenta} curves represent the tails at the intersection of the set $\calC_{p_4,\vartheta}$ with the green lines. The \textcolor{violet}{violet} contour represents half of the contour that appears when we run the algorithm a second time.}
\label{fig:thimb.A4B1.IV}
\end{figure}
We can decompose the thimble in terms of state integrals contours. We have contributions of $\calI_{-3,0}$ from the intersection of the orange contour with the interval $(-1,0)$, which has $(n,m,\ell)=(0,-3,0)$. Then, due to the intersection with the green lines we should run the algorithm another time for the magenta line. The latter intersects a green line, thus we apply the algorithm one more time for the violet line. Then we find a contribution of $-q^2\calI_{-5,1}$ from the intersection of the violet line with the interval $(0,1)$, which has $(n,m,\ell)=(2,-5,1)$. Summarising, the algorithm gives
\be\label{eq:region_IV_A4B1}
\calI_{p_4,V_4,\mathrm{IV}}
\=
\calI_{-3,0}-q^2\calI_{-5,1}=\calI_{-3,0}-q\calI_{-1,0}\,,
\ee
where the second equality follows from Equation~\eqref{eq:state_integral_1}. Notice that the magenta line in Figure~\ref{fig:thimb.A4B1.IV} does not contribute even if it intersects the interval $(0,1)$. Indeed it intersects the green line before intersecting the reals, and this leads to the contribution from the violet tail, as shown in Figure~\ref{fig:thimb.A4B1.IV--no_magenta}.
\begin{figure}[ht]
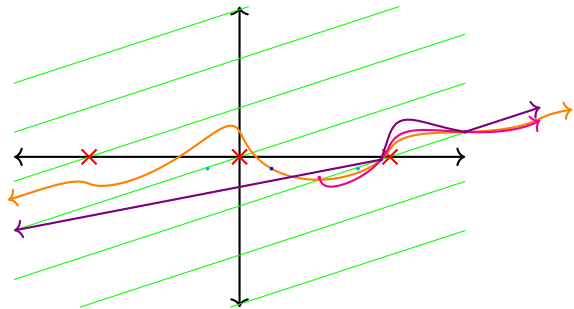

\begin{tiny}

\end{tiny}
\caption{This figure depicts in \textcolor{violet}{violet} the tail that contributes in the decomposition of the set $\calC_{p_4,\vartheta}$ given by state integral contours.}
\label{fig:thimb.A4B1.IV--no_magenta}
\end{figure}
\subsubsection{First few Stokes constants}\label{sec:stokes_A4B1}
We can now describe the Stokes matrices associated to first Stokes rays in Figure~\ref{fig:A4B1.stokes}. From Equations~\eqref{eq:region_I_A4B1} and~\eqref{eq:region_Is_A4B1} we get
\be
\begin{tiny}
\begin{aligned}
  %
\,.
  \end{aligned}
\end{tiny}
\ee
Therefore we learn that there is no Stokes jump between region $I$ and $I_\star$. 
Then, from Equations~\eqref{eq:region_I_A4B1} and~\eqref{eq:region_II_A4B1} we find 
\be
\begin{tiny}
\begin{aligned}
  %
\,.
\end{aligned}
\end{tiny}
\ee
Finally, from Equations~\eqref{eq:region_III_A4B1} and~\eqref{eq:region_IV_A4B1} we find 
\be
\begin{tiny}
\begin{aligned}
  %
\,.
\end{aligned}
\end{tiny}
\ee
Therefore, we see that
\be\label{eq:AB41stokes}
\begin{aligned}
  \mathfrak{S}_{\mathrm{I},\mathrm{I}_\star}(q)
  &=&
  %
.&\,&
  \end{aligned}
\ee

\subsubsection{Computing all Stokes constants}
We can again compute the Stokes matrix $\mathfrak{S}_{\mathrm{I},\mathrm{X}}$ using quantum modularity~\cite{GZ:RQMOD} or equivalently the factorisations of the state integrals~\cite{GGM:I,GK:qser}. Here it is more complicated than $4_1$ or $(A,B)=(1,2)$ from Section~\ref{sec:4_1} as one needs to use an ``untrapping'' of the state integral and various identities between $q$-series and state integrals. This was all done in~\cite[Sec.~10.1]{Wh:thesis}. To give these results we need some notation.

Let $\Theta(q)$ be the matrix whose columns are given by the coefficients of $\epsilon^0,\epsilon^1,\epsilon^2,\epsilon^4$ respectively of the series
\be
    \exp\Big(\frac{1}{2}\epsilon+\frac{1}{4}\epsilon^2E_{2}(q)\Big)
    \begin{pmatrix}
        \theta(-\exp(\epsilon);q^{4})\\
        -q^{5/8}\exp(\epsilon/4)\theta(-q\exp(\epsilon);q^{4})\\
        q^{12/8}\exp(\epsilon/2)\theta(-q^{2}\exp(\epsilon);q^{4})\\
        -q^{21/8}\exp(3\epsilon/4)\theta(-q^{3}\exp(\epsilon);q^{4})
    \end{pmatrix}\,,
\ee
where $\theta(x;q)=(qx;q)_{\infty}(x^{-1};q)_{\infty}(q;q)_{\infty}$ and $E_{2}(q)=-\tfrac{1}{24}+\sum_{k=1}^{\infty}k\tfrac{q^k}{1-q^k}$.
Let
\be
    G(t;q)
    \=
    \sum_{k=0}^{\infty}\frac{q^{3k(4k+1)/2}t^{4k}}{(q;q)_{4k}}
    \begin{pmatrix}
        \frac{1}{(q;q)_{4k}} &
        \frac{q^{3k+\frac{3}{4}}t}{(q;q)_{4k+1}} &
        \frac{q^{6k+\frac{9}{4}}t^{2}}{(q;q)_{4k+2}} &
        \frac{q^{9k+\frac{9}{2}}t^{3}}{(q;q)_{4k+3}}\\
        \frac{q^{4k}}{(q;q)_{4k}} &
        \frac{q^{3k+\frac{3}{4}+4k+1}t}{(q;q)_{4k+1}} &
        \frac{q^{6k+\frac{9}{4}+4k+2}t^{2}}{(q;q)_{4k+2}} &
        \frac{q^{9k+\frac{9}{2}+4k+3}t^{3}}{(q;q)_{4k+3}}\\
        \frac{q^{8k}}{(q;q)_{4k}} &
        \frac{q^{3k+\frac{3}{4}+8k+2}t}{(q;q)_{4k+1}} &
        \frac{q^{6k+\frac{9}{4}+8k+4}t^{2}}{(q;q)_{4k+2}} &
        \frac{q^{9k+\frac{9}{2}+8k+6}t^{3}}{(q;q)_{4k+3}}\\
        \frac{q^{12k}}{(q;q)_{4k}} &
        \frac{q^{3k+\frac{3}{4}+12k+3}t}{(q;q)_{4k+1}} &
        \frac{q^{6k+\frac{9}{4}+12k+6}t^{2}}{(q;q)_{4k+2}} &
        \frac{q^{9k+\frac{9}{2}+12k+9}t^{3}}{(q;q)_{4k+3}}
    \end{pmatrix}\,,
\ee
\be
\begin{tiny}
\begin{aligned}
    P(q)
    \=
    \begin{pmatrix}
    1 & 0 & 0 & 0\\
    -q^{-3} - q^{-2} - q^{-1} &  q^{-3} + q^{-2} + 2q^{-1} + 1 + q &  -q^2 - q^3 - q^4 - q^5 - q^{12} &  q^{12}\\
    q^{-5} + q^{-4} + q^{-3} &  -q^{-5} - 2q^{-4} - 2q^{-3} - 2q^{-2} - q^{-1} + q^3 &  1 + 2q + 2q^2 + q^3 + q^4 + q^{10} + q^{11} &  -q^{10} - q^{11}\\
    -q^{-6} - q^{-1} &  q^{-6} + q^{-5} + q^{-4} - q^2 &  -q^{-1} - 1 - q - q^9 &  q^9
    \end{pmatrix}
\end{aligned}
\end{tiny}
\ee
and
\be
\begin{aligned}
    &F(q)\=P(q)
    \begin{pmatrix}
        1 & 0 & 0 & 0\\
        0 & 1 & 0 & 0\\
        0 & 0 & q^{-4} & 0\\
        0 & 0 & 0 & q^{-12}\\
    \end{pmatrix}
    G(q^{-3/8};q)
    \Theta(q)\,.
\end{aligned}
\ee
The first column of $F(q)$ is given by
\be
  \sum_{k=0}^{\infty}\frac{q^{2k(k+1)}}{(q;q)_{k}}
  \begin{pmatrix}
    q^{-2k}\\q^{-k}\\1\\q^k
  \end{pmatrix}\,.
\ee
The matrix $F(q)$ can be extended to $|q|\neq1$ using the relations $\theta(x;q^{-1})=\theta(x^{-1};q)^{-1}$ and $E_{2}(q^{-1})=-E_{2}(q)$.
Let $f(q)$ be the third row of $F(q)$. Then using the identities of~\cite[Sec.~10.1]{Wh:thesis}, we find that
\be
\begin{aligned}
\begin{pmatrix}
    \calI_{p_1,V_1,\mathrm{I}} & \calI_{p_2,V_2,\mathrm{I}} &\calI_{p_3,V_3,\mathrm{I}} & \calI_{p_4,V_4,\mathrm{I}}
  \end{pmatrix}
  \=
  f(q^{-1})
  \begin{pmatrix}
        1 & 0 & 0 & 0\\
        0 & \tau & 0 & 0\\
        0 & 0 & \tau^2 & 0\\
        0 & 0 & 0 & \tau^4
    \end{pmatrix}^{-1}
\!\!\!\!\!\!
    F(q^{-1})^{-1}\!
    \begin{pmatrix}
        0 & 0 & 1 & 0\\
        1 & 0 & 0 & 0\\
        0 & 1 & 0 & 0\\
        0 & 0 & 0 & q^{-1}
    \end{pmatrix}^{-1}\!\!\!\!\!\!\!.
\end{aligned}
\ee
This exactly agrees with the conjectures~\cite[Ex.~61]{Wh:thesis} for the Borel--Laplace resummation in region $\mathrm{I}$. A similar computation shows that in the region $\mathrm{X}$ we also find agreement with the conjectures and
\be
\begin{aligned}
\!\begin{pmatrix}
    \calI_{p_1,V_1,\mathrm{X}} & \calI_{p_2,V_2,\mathrm{X}} &\calI_{p_3,V_3,\mathrm{X}} & \calI_{p_4,V_4,\mathrm{X}}
  \end{pmatrix}
  \=
  f(q^{-1})\!
  \begin{pmatrix}
        1 & 0 & 0 & 0\\
        0 & -\tau & 0 & 0\\
        0 & 0 & \tau^2 & 0\\
        0 & 0 & 0 & \tau^4
    \end{pmatrix}^{-1}
\!\!\!\!\!\!\!
    F(q^{-1})^{-1}\!\!
    \begin{pmatrix}
        0 & 0 & 1 & 0\\
        1 & 0 & 0 & 0\\
        0 & 0 & 0 & q^{-1}\\
        0 & 1 & 0 & 0
    \end{pmatrix}^{-1}\!\!\!\!\!\!\!.
\end{aligned}
\ee
This implies that
\be
\begin{aligned}
  \mathsf{S}_+(q)
  &\=\begin{pmatrix}
        0 & 0 & 1 & 0\\
        1 & 0 & 0 & 0\\
        0 & 1 & 0 & 0\\
        0 & 0 & 0 & q^{-1}
    \end{pmatrix}
    F(q^{-1})
    \begin{pmatrix}
        1 & 0 & 0 & 0\\
        0 & -1 & 0 & 0\\
        0 & 0 & 1 & 0\\
        0 & 0 & 0 & 1
    \end{pmatrix}
    F(q^{-1})^{-1}
    \begin{pmatrix}
        0 & 0 & 1 & 0\\
        1 & 0 & 0 & 0\\
        0 & 0 & 0 & q^{-1}\\
        0 & 1 & 0 & 0
    \end{pmatrix}^{-1}\\
    &\=
    \left(\begin{pmatrix}
            0 & 0 & 1 & 0\\
            1 & 0 & 0 & 0\\
            0 & 1 & 0 & 0\\
            0 & 0 & 0 & q
        \end{pmatrix}
        F(q)
        \begin{pmatrix}
            1 & 0 & 0 & 0\\
            0 & -1 & 0 & 0\\
            0 & 0 & 1 & 0\\
            0 & 0 & 0 & 1
        \end{pmatrix}
        F(q)^{-1}
        \begin{pmatrix}
            0 & 0 & 1 & 0\\
            1 & 0 & 0 & 0\\
            0 & 0 & 0 & q^{-1}\\
            0 & 1 & 0 & 0
        \end{pmatrix}^{-1}\right)^{-T}\\
    &\=
    \begin{pmatrix}
    1 - q - 2q^2 & q^2 & q & -q - q^2\\
    1 + q + q^2 & 1 - q - q^2 & -1 + q + q^2 & 1 - q^2\\
    -1 - q & q + q^2 & 1 - q - 2q^2 & 2q^2\\
    1 - q^2 & -q & 2q + q^2 & 1 - q - 2q^2
    \end{pmatrix}
    +O(q^3)\\
    &\=
    \begin{pmatrix}
     1 & 0 & 0 & 0\\
     0 & 1 & 0 & 1 \\
     0 & 0 & 1 & 0 \\
     0 & 0 & 0 & 1
    \end{pmatrix}
    \begin{pmatrix}
     1 & 0 & 0 & 0\\
     0 & 1 & 0 & 0 \\
     -1 & 0 & 1 & 0 \\
     0 & 0 & 0 & 1
    \end{pmatrix}
    \begin{pmatrix}
     1 & 0 & 0 & 0\\
     0 & 1 & 0 & 0 \\
     0 & 0 & 1 & -q \\
     0 & 0 & 0 & 1
    \end{pmatrix}
    \cdots\,.
\end{aligned}
\ee
The second equality follows from the duality of~\cite[Conj.~1]{GGM:I} and can be explicitly proved in this example.
This exactly agrees with the computations of Equation~\eqref{eq:AB41stokes}.



\appendix

\section{Critical points and independence of state integrals}~\label{app:crit.pts}
In this appendix we prove that the critical points of the function $V(z)$ are non-degenerate away from the branch points.
\begin{lemma}
For $P(z)=(-z)^A-(1-z)^B$ with $A,B\in\BZ_{>0}$ with $A\neq B$ there are no solutions to the equations
\be
  P(z)\=P(z')\=0\,.
\ee
\end{lemma}
\begin{proof}
We prove this by contradiction. Suppose there exists a degenerate critical point $x_0$ so that
\be
  (-x_0)^A-(1-x_0)^B
  \=P(x_0)
  \=0
  \=P'(x_0)
  \=(-x_0)^A\frac{A}{x_0}+(1-x_0)^B\frac{B}{1-x_0}\,.
\ee
This would imply that
\be
  x_0
  \=
  \frac{A}{A-B}\,,
\ee
which would in turn imply that
\be
  (-1)^A\frac{A^A}{(A-B)^A}
  \=(-1)^B\frac{B^B}{(A-B)^B}\,.
\ee
We can clear the greatest common divisor defining $A=A_0(A,B)$ and $B=B_0(A,B)$ to obtain
\be\label{eq:int.equ.app}
  (-1)^A\frac{A_0^A}{(A_0-B_0)^A}
  \=(-1)^B\frac{B_0^B}{(A_0-B_0)^B}\,.
\ee
Given that $(A_0,B_0)=1$, we see that $A_0$ is invertible modulo $B_0$. Therefore, Equation~\eqref{eq:int.equ.app}
implies that
\be
  (-1)^A\equiv(-1)^AA_0^A(A_0-B_0)^{-A}\equiv(-1)^B\frac{B_0^B}{(A_0-B_0)^B}\equiv0\pmod{B_0}\,,
\ee
which is a contradiction.
\end{proof}

\begin{corollary}\label{cor:phi.is.invert}
The matrix $\widehat{\Phi}_j(\hbar)$ for Equation~\eqref{eq:phimat} is invertible.
\end{corollary}
\begin{proof}
The lemma shows that the roots of $P(z)$ are independent. Moreover, the constant term of $\Phi_{(A,B,p_0),j}(\hbar)$ is equal to the constant term of $\Phi_{(A,B,p_0),0}(\hbar)$ multiplied by $x_0^j$ where $P(x_0)=0$ is the root corresponding to $p_0$. Therefore, the constant term of the matrix $\widehat{\Phi}_j(\hbar)$ gives an invertible diagonal matrix times a Vandermonde matrix with distinct entries and therefore invertible.
\end{proof}

\section{Faddeev's dilogarithm}\label{app:Faddeev}

This appendix is dedicated to a description and proof of the asymptotics of Faddeev's quantum dilogarithm given in the introduction (in particular, Theorem~\ref{thm:fad.asymp}).

\subsection{Asymptotics of the Pochhammer symbol}

The asymptotics of the Pochhammer symbol give rise to a particular branch of the dilogarithm for each argument of $\tau$. The zeros or poles then line up along the branch cuts defining this principle branch. For $\theta\in[0,2\pi)$ define $\mathrm{L}_{\theta}:\BC\backslash(\BZ+e^{i\theta}\BR_{\leq0})\rightarrow\BC$ to be the holomorphic function such that
\be
  \mathrm{L}_{\theta}(z)
  \=
  \int_{z+e^{i\theta}\BR_{\geq0}}
  \int_{w+e^{i\theta}\BR_{\geq0}}
  \frac{\e(\z)}{1-\e(\z)}\,d\z\,dw\,.
\ee
This function $\mathrm{L}_{\theta}(z)$ gives a particular branch of the multivalued  function $z\mapsto\Li_{2}(\e(z))/(2\pi i)^2$ for each $\theta$ and if $\theta\in(0,\pi)$ and $\Im(z)>0$ then $\mathrm{L}_{\theta}(z)=\Li_{2}(\e(z))/(2\pi i)^2$ for the principle branch of $\Li_2$. The domain of $\mathrm{L}_{\theta}$ is pictured in Figure~\ref{fig:prin.bran}. This function also satisfies
\be
  \frac{d^{k+2}}{dz^{k+2}}
  \mathrm{L}_{\theta}(z)
  \=
  (2\pi i)^{k}\Li_{-k}(\e(z))\,.
\ee
\begin{center}
\begin{figure}
\begin{tikzpicture}
\draw[<->,thick] (0,-2) -- (0,2);
\draw[<->,thick] (-4.5,0) -- (4.5,0);
\draw[red] (0,0) -- (2,-2);
\draw[red] (2,0) -- (2+2,-2);
\draw[red] (4,0) -- (4+0.5,-0.5);
\draw[red] (-2,0) -- (-2+2,-2);
\draw[red] (-4,0) -- (-4+2,-2);
\draw[thick,red] (-0.05,-0.05)--(0.05,0.05);
\draw[thick,red] (0.05,-0.05)--(-0.05,0.05);
\draw[thick,red,xshift=-4cm] (-0.05,-0.05)--(0.05,0.05);
\draw[thick,red,xshift=-4cm] (0.05,-0.05)--(-0.05,0.05);
\draw[thick,red,xshift=-2cm] (-0.05,-0.05)--(0.05,0.05);
\draw[thick,red,xshift=-2cm] (0.05,-0.05)--(-0.05,0.05);
\draw[thick,red,xshift=2cm] (-0.05,-0.05)--(0.05,0.05);
\draw[thick,red,xshift=2cm] (0.05,-0.05)--(-0.05,0.05);
\draw[thick,red,xshift=4cm] (-0.05,-0.05)--(0.05,0.05);
\draw[thick,red,xshift=4cm] (0.05,-0.05)--(-0.05,0.05);
\end{tikzpicture}
\caption{Domain of $\mathrm{L}_{3\pi/4}(z)$ given by $\BC\backslash(\BZ+(i-1)\BR_{\leq0})$.}
\label{fig:prin.bran}
\end{figure}
\end{center}
\begin{lemma}\label{lem:asymp.qpoch}
Suppose that $\Im(\tau)>0$ and $\varepsilon\in\BR_{>0}$. Then as $\tau\rightarrow i\infty$ with fixed argument and $z$ is bounded away from the half lines $\{z\,|\,\Im((z+\BZ_{\geq0})\tau)=0,\text{ and }\Im(z)\leq0\}$ by $\varepsilon$ with $\theta=\arg(-1/\tau)\in(0,\pi)$ there exists a constant $C>0$ such that
\be
    \Big|(\e(z);\tq)_{\infty}
    -
    \e\Big(-\mathrm{L}_{\theta}(z)\tau
    -
    \frac{1}{2}\mathrm{L}_{\theta}'(z)
    -
    \sum_{k=2}^{K-1}\frac{B_k}{k!}\frac{(2\pi i)^{k-2}}{\tau^{k-1}}\Li_{2-k}(\e(z))\Big)
    \Big|
    <
    C\,\varepsilon^{-K}K!\,|\tau|^{-K}.
\ee
\end{lemma}
\begin{proof}
Given the positions of the branch cuts of $\mathrm{L}_{\theta}$, we see that
\be
    \frac{1}{2\pi i}\log(\e(z);\e(-1/\tau))_{\infty}
    \=
    -\sum_{n=0}^{\infty}\mathrm{L}_{\theta}'(z-n/\tau)\,.
\ee
Therefore, applying Euler-Maclaurin summation (and the fact that $B_{2k+1}=0$ for $k>0$) we find that
\be
\begin{aligned}
    -\sum_{n=0}^{\infty}\mathrm{L}_{\theta}'(z-n/\tau)
    &\=
    -\int_{0}^{\infty}\mathrm{L}_{\theta}'(z-\eta/\tau)\;d\eta-\frac{\mathrm{L}_{\theta}'(z)}{2}-\sum_{k=2}^{K-1}\frac{B_{k}}{k!}\frac{(2\pi i)^{k-2}}{\tau^{k-1}}\Li_{2-k}(\e(z))\\
    &\qquad+\frac{(-2\pi i)^{K}}{\tau^{K-1}}\int_{0}^{\infty}\Li_{2-K}(\e(z-\eta/\tau))\frac{B_{K}(\eta-\lfloor \eta\rfloor)}{K!}\;d\eta\,.
\end{aligned}
\ee
We see that
\be
    -\int_{0}^{\infty}\mathrm{L}_{\theta}'(z-\eta/\tau)\;d\eta
    \=
    -\tau \mathrm{L}_{\theta}(z)\,.
\ee
The functions $\Li_{k}$ for $k\leq0$ can be expressed in terms of Eulerian numbers
\be
  \Li_{-k}(x)
  \=
  \frac{1}{(1-x)^{k+1}}\sum_{\ell=0}^{k-1}\Big\langle\!\!\begin{array}{c}k\\\ell\end{array}\!\!\Big\rangle\, x^{k-\ell}\,,
\ee
where
\be
  \Big\langle\!\!\begin{array}{c}k\\\ell\end{array}\!\!\Big\rangle
  \=
  \sum_{i=0}^{\ell}(-1)^{i}\binom{k+1}{i}(\ell+1+i)^{k}
  \quad\text{and}\quad
  0\leq
  \Big\langle\!\!\begin{array}{c}k\\\ell\end{array}\!\!\Big\rangle
  \leq
  k!\,.
\ee
Therefore, using the fact that $\Li_{k}(x)=-(-1)^{k}\Li_{k}(1/x)$ for $k>0$ we find that for $|x-1|\geq\varepsilon$
\be
  |\Li_{-k}(x)|
  \leq
  (k+1)!\,\min\{|x|,|x^{-1}|\}\varepsilon^{-k-1}\,.
\ee
Therefore, as $z$ is bounded away from the branch cuts of $L_{\theta}$ by $\varepsilon$, we see that there is a constant $C_1$ such that
\be\label{eq:bound.polylog}
    |\Li_{2-K}(\e(z-\eta/\tau))|
    <
    (K-1)!\varepsilon^{1-K}C_1|\e(-\eta/\tau)|\,.
\ee
Moreover, we have Lehmer's bounds for the periodic Bernoulli polynomials
\be
  \Big|\frac{B_{K}(\eta-\lfloor \eta\rfloor)}{K!}\Big|
  <
  \frac{2\z(K)}{(2\pi)^K}\,.
\ee
Therefore, there exists $C_2$ such that
\be
  \Big|\int_{0}^{\infty}\Li_{2-K}(\e(z-\eta/\tau))B_{K}(\eta-\lfloor \eta\rfloor)\;d\eta\Big|
  \;<\;C_2\,\varepsilon^{1-K}\frac{(K-1)!}{(2\pi)^K}\,.
\ee
Therefore, there exists $C$ such that
\be
\Big|(\e(z);\tq)_{\infty}
    -
    \e\Big(-\mathrm{L}_{\theta}(z)\tau
    -
    \frac{1}{2}\mathrm{L}_{\theta}'(z)
    -
    \sum_{k=2}^{K-1}\frac{B_k}{k!}\frac{(2\pi i)^{k-2}}{\tau^{k-1}}\Li_{2-k}(\e(z))\Big)
    \Big|
    <
    C\,\varepsilon^{1-K}(K-1)!|\tau|^{1-K}.
\ee
\end{proof}
\begin{remark}
An easy way to see that $\mathrm{L}_{\theta}$ is the correct branching of the multivalued function $\Li_{2}(\e(z))/(2\pi i)^2$ is that the zeros of $(\e(z);\e(-1/\tau))_{\infty}$ accumulate to the lines $\{z\,|\,\Im((z+\BZ_{\geq0})\tau)=0,\text{ and }\Im(z)\leq0\}$ in the limit as $\tau$ tends to infinity.
\end{remark}

\subsection{\texorpdfstring{Asymptotics of Faddeev's dilogarithm for $\Im(\tau)\neq0$}{Asymptotics of Faddeev's dilogarithm for Im(tau)=/=0}}

We are interested in Faddeev's quantum dilogarithm, which in the upper half plane has the expression
\be
  \Phi(z;\tau)
  \=
  \frac{(q\e(z);q)_{\infty}}{(\e(z/\tau);\tq)_{\infty}}\,.
\ee
To understand the asymptotics of this function we need to consider a different branch of the dilogarithm function again. For $\theta\in(0,2\pi)$ define the domain $\BC_{\theta}=\BC\backslash\big((\BZ_{\geq0}+e^{i\theta}\BR_{\leq0})\cup(\BZ_{<0}+e^{i\theta}\BR_{\geq0})\big)$ depicted in Figure~\ref{fig:prin.bran.fad}. Then define $\mathrm{D}_{\theta}:\BC_{\theta}\rightarrow\BC$ to be
\be
  \mathrm{D}_{\theta}(z)
  \=
  \int_{z}^{e^{i\theta/2}\infty}
  \int_{w}^{e^{i\theta/2}\infty}
  \frac{\e(\z)}{1-\e(\z)}\,d\z\,dw\,,
\ee
where the contours are contained in $\BC_{\theta}$. Using this function we have the following asymptotics of $\Phi(z;\tau)$ in the upper half plane.
\begin{center}
\begin{figure}
\begin{tikzpicture}
\draw[red] (0,0) -- (2,-2);
\draw[red] (2,0) -- (2+2,-2);
\draw[red] (4,0) -- (4+0.5,-0.5);
\draw[red] (-2,0) -- (-2-2,2);
\draw[red] (-4,0) -- (-4-0.5,0.5);
\draw[<->,thick] (0,-2) -- (0,2);
\draw[<->,thick] (-4.5,0) -- (4.5,0);
\draw[thick,red] (-0.05,-0.05)--(0.05,0.05);
\draw[thick,red] (0.05,-0.05)--(-0.05,0.05);
\draw[thick,red,xshift=-4cm] (-0.05,-0.05)--(0.05,0.05);
\draw[thick,red,xshift=-4cm] (0.05,-0.05)--(-0.05,0.05);
\draw[thick,red,xshift=-2cm] (-0.05,-0.05)--(0.05,0.05);
\draw[thick,red,xshift=-2cm] (0.05,-0.05)--(-0.05,0.05);
\draw[thick,red,xshift=2cm] (-0.05,-0.05)--(0.05,0.05);
\draw[thick,red,xshift=2cm] (0.05,-0.05)--(-0.05,0.05);
\draw[thick,red,xshift=4cm] (-0.05,-0.05)--(0.05,0.05);
\draw[thick,red,xshift=4cm] (0.05,-0.05)--(-0.05,0.05);
\end{tikzpicture}
\caption{Domain of $\mathrm{D}_{3\pi/4}(z)$ given by $\BC\backslash\big((\BZ_{\geq0}+(i-1)\BR_{\leq0})\cup(\BZ_{<0}+(i-1)\BR_{\geq0})\big)$.}
\label{fig:prin.bran.fad.2}
\end{figure}
\end{center}
\begin{lemma}\label{lem:asymp.qpoch.2}
Suppose that $\Im(\tau)>0$ and $\varepsilon\in\BR_{>0}$. Then as $\tau\rightarrow i\infty$ with fixed argument and $z$ is bounded away from the half lines $\BC\backslash\BC_{\theta}$ and is bounded away from at least one of the lines $\Im((z+1)\tau)=0$ and $\Im((z+1)\tau)$ by $\varepsilon$ with $\theta=\arg(-1/\tau)\in(0,\pi)$ there exists a constant $C\in\BR_{>0}$ such that
\be
    \Big|\Phi(z\tau;\tau)
    -
    \e\Big(\mathrm{D}_{\theta}(z)\tau
    +
    \frac{1}{2}\mathrm{D}_{\theta}'(z)
    +
    \sum_{k=2}^{K}\frac{B_k}{k!}\frac{(2\pi i)^{k-2}}{\tau^{k-1}}\Li_{2-k}(\e(z))\Big)
    \Big|
    <
    C\,K!\, \varepsilon^{-K}\,|\tau|^{-K}\,.
\ee
\end{lemma}
\begin{proof}
This follows from the Lemma~\ref{lem:asymp.qpoch} when $\Im((z+1)\tau)>|\tau|\varepsilon$ and the fact that in this domain we have
\be
  |(q\e(z\tau);q)_{\infty}-1|<2\exp(-2\pi\varepsilon|\tau|)\,.
\ee
Then we use the fact that by the modularity of the $\theta$-function and $\eta$-function, and the Jacobi triple product
\be
  \Phi(z;\tau)\Phi(-\tau-z;\tau)
  \=
  \frac{(q\e(z);q)_{\infty}(\e(-z);q)_{\infty}}{(\e(z/\tau);\tq)_{\infty}(\tq\e(z/\tau);\tq)_{\infty}}
  \=
  iq^{1/6}\tq^{-1/6}\e(z^2/2\tau+z/2+z/2\tau).
\ee
Therefore, this implies that for $\Im(z\tau)<|\tau|\varepsilon$
\be
  \Phi(z\tau;\tau)
  \=
  \Phi(-\tau-z\tau;\tau)^{-1}
  iq^{1/6}\tq^{-1/6}\e(z(z+1+1/\tau)\tau/2)\,.
\ee
Therefore, using the first half of the proof and the relations between the polylogarithms and $D_{\theta}$ under the map $z\mapsto-z$ gives the result when $\Im(z\tau)<|\tau|\varepsilon$, which completes the proof.
\end{proof}

\begin{remark}
The previous remark about the function $\mathrm{L}_{\theta}$ also applies to the quantum dilogarithm. An easy way to see that $\mathrm{D}_{\theta}$ is the correct branching of the multivalued function $\Li_{2}(\e(z))/(2\pi i)^2$ is that the zeros and poles of $\Phi(z\tau;\tau)$ accumulate to the lines $\BC\backslash\BC_{\theta}$ in the limit as $\tau$ tends to infinity. Here the most important point is that this domain remains connected for all $\theta\in(0,2\pi)$, which was not true for the lone Pochhammer symbol.
\end{remark}

\subsection{\texorpdfstring{Asymptotics of Faddeev's dilogarithm for $\Im(\tau)$ near $0$}{Asymptotics of Faddeev's dilogarithm for Im(tau) near 0}}

To understand the asymptotics as we cross the real numbers we use the following equality for $\Im(z+\tau)>0$ and $\Im(z/\tau)>0$ along with $\Re(-\sqrt{\tau}-1/\sqrt{\tau})<\Re(z/\sqrt{\tau})<0$
\be\label{eq:fad.expressions.II}
\begin{aligned}
  \Phi(z;\tau)
  &\=
  \frac{(q\e(z);q)_{\infty}}{(\e(z/\tau);\tq)_{\infty}}
  \=
  \exp\Big(
    \sum_{k=1}^{\infty}
    \frac{\e(kz)}{k(q^{-k}-1)}
    -\frac{\e(kz/\tau)}{k(\tq^k-1)}
  \Big)\\
  &\=
  \exp\Big(
    \int_{i\sqrt{\tau}\BR+\varepsilon\sqrt{\tau}}\frac{\e((z+1+\tau)w/\tau)}{(\e(w)-1)(\e(w/\tau)-1)}
    \frac{dw}{w}
  \Big)\,,
\end{aligned}
\ee
for some small $\varepsilon>0$. 
The various regions where the expression for $\Phi(\tau z;\tau)$ as $\tau$ tends to infinity are depicted in Figure~\ref{fig:fad.qpoch.overlap}.
\begin{center}
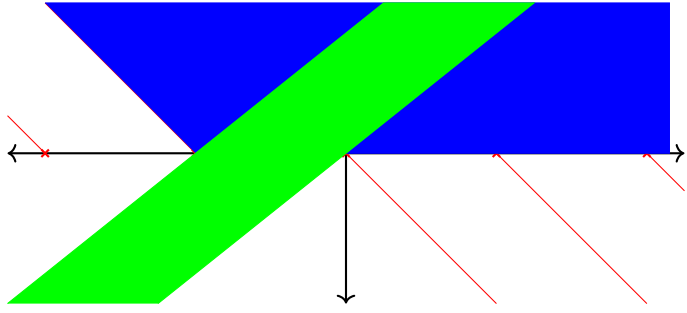
\begin{figure}
\begin{tikzpicture}
\draw[red] (0,0) -- (2,-2);
\draw[red] (2,0) -- (2+2,-2);
\draw[red] (4,0) -- (4+0.5,-0.5);
\draw[red] (-2,0) -- (-2-2,2);
\draw[red] (-4,0) -- (-4-0.5,0.5);
\draw[<->,thick] (0,-2) -- (0,2);
\draw[<->,thick] (-4.5,0) -- (4.5,0);
\draw[thick,red] (-0.05,-0.05)--(0.05,0.05);
\draw[thick,red] (0.05,-0.05)--(-0.05,0.05);
\draw[thick,red,xshift=-4cm] (-0.05,-0.05)--(0.05,0.05);
\draw[thick,red,xshift=-4cm] (0.05,-0.05)--(-0.05,0.05);
\draw[thick,red,xshift=-2cm] (-0.05,-0.05)--(0.05,0.05);
\draw[thick,red,xshift=-2cm] (0.05,-0.05)--(-0.05,0.05);
\draw[thick,red,xshift=2cm] (-0.05,-0.05)--(0.05,0.05);
\draw[thick,red,xshift=2cm] (0.05,-0.05)--(-0.05,0.05);
\draw[thick,red,xshift=4cm] (-0.05,-0.05)--(0.05,0.05);
\draw[thick,red,xshift=4cm] (0.05,-0.05)--(-0.05,0.05);
\filldraw[blue,opacity=0.2] (-2,0) -- (4.3,0) -- (4.3,2) -- (-4,2);
\filldraw[green,opacity=0.2] (-2.5,-2) -- (2.5,2) -- (0.5,2) -- (-4.5,-2);
\end{tikzpicture}
\caption{The blue region depicts the region the first line of Equation~\eqref{eq:fad.expressions.II} considered for $\Phi(\tau z;\tau)$ as $\tau$ is defined and the green region depicts the region the second line is defined.}
\label{fig:fad.qpoch.overlap}
\end{figure}
\end{center}
The asymptotics of Faddeev's quantum dilogarithm in the lower half plane can be carried out by a completely similar analysis as in the upper half plane using the formula for $\Im(\tau)<0$ given by
\be
  \Phi(z;\tau)
  \=
  \frac{(\tq^{-1}\e(z/\tau);\tq^{-1})_{\infty}}{(\e(z);q^{-1})_{\infty}}\,.
\ee
In the upper and lower half planes there is one technicality that when $\tau$ is very close to the reals or argument comparable with $\varepsilon$ (the bound for $z$ from the branch cuts) we find that for $z$ in a neighbourhood of the interval $[-1+2\varepsilon,-2\varepsilon]$ the methods we have previously mentioned fail to produce the asymptotics. We also have the issue that the previous methods do not work when $\tau\in\BR_{>0}$. These two problems can be dealt with by using Faddeev's original formula given in Equation~\eqref{eq:fad.expressions.II}.
\begin{lemma}\label{lem:asymp.qpoch.3}
Suppose that $\tau\in\BC\backslash\BR_{\leq0}$ and $\varepsilon\in\BR_{>0}$. Then as $\tau\rightarrow \infty$ with fixed argument bounded by $\varepsilon$ and $z$ bounded away from $\BC_{\theta}$ by $\varepsilon$ with $\theta=\arg(-1/\tau)\in(\pi-\varepsilon,\pi+\varepsilon)$ there exists a constant $C\in\BR_{>0}$ such that
\be
    \Big|\Phi(z\tau;\tau)
    -
    \e\Big(\mathrm{D}_{\theta}(z)\tau
    +
    \frac{1}{2}\mathrm{D}_{\theta}'(z)
    +
    \sum_{k=2}^{K}\frac{B_k}{k!}\frac{(2\pi i)^{k-2}}{\tau^{k-1}}\Li_{2-k}(\e(z))\Big)
    \Big|
    <
    C\,K!\,\varepsilon^{-K}\,|\tau|^{-K}\,.
\ee
\end{lemma}
\begin{proof}
By deforming the contour to infinity for $\Im(z)>0$ we have
\be\label{eq:fad.int.asymp}
  \int_{i\sqrt{\tau}\BR+\varepsilon_1\sqrt{\tau}}
  \frac{\e(zw)}{(\e(-w)-1)}
  \frac{dw}{w^{2-k}}
  \=
  \sum_{\ell=1}^{\infty}\frac{\e(\ell z)}{\ell^{2-k}}
  \=
  \Li_{2-k}(\e(z))\,,
\ee
while for $\Im(z)<0$ we can deform in the other direction and find that for $k>1$ we have
\be
\begin{aligned}
  \int_{i\sqrt{\tau}\BR+\varepsilon_1\sqrt{\tau}}
  \frac{\e(zw)}{(\e(-w)-1)}
  \frac{dw}{w^{2-k}}
  &\=
  -\delta_{k,2}
  -\sum_{\ell=1}^{\infty}\frac{\e(-\ell z)}{(-\ell)^{2-k}}\\
  \=
  \Li_{2-k}(\e(z))\,,
\end{aligned}
\ee
While for $k=0$ we have
\be
  \int_{i\sqrt{\tau}\BR+\varepsilon_1\sqrt{\tau}}
  \frac{\e(zw)}{(\e(-w)-1)}
  \frac{dw}{w^{2-k}}
  \=
  -(2\pi i)^2(\frac{1}{2}z(z+1)+\frac{1}{12})
  -\sum_{\ell=1}^{\infty}\frac{\e(-\ell z)}{(-\ell)^{2}}
  \=
  D_{\theta}(z)\,,
\ee
and $k=1$ we have
\be\label{eq:fad.int.asymp.2}
  \int_{i\sqrt{\tau}\BR+\varepsilon_1\sqrt{\tau}}
  \frac{\e(zw)}{(\e(-w)-1)}
  \frac{dw}{w^{2-k}}
  \=
  -(2\pi i)(z+\frac{1}{2})
  +\sum_{\ell=1}^{\infty}\frac{\e(-\ell z)}{\ell}
  \=
  D_{\theta}'(z)\,.
\ee
Therefore, these equations hold for all $z$ in our region. We can use these equations to deduce our desired asymptotics. We want to compute the asymptotics of the integral 
\be
  \int_{\mathcal{C}}
  \frac{\e(zw)}{(\e(-w)-1)(\e(-w/\tau)-1)}
  \frac{dw}{w}\,,
\ee
where for $z$ with $\Re((\varepsilon_1-1)\sqrt{\tau}-1/\sqrt{\tau})<\Re(z\sqrt{\tau})<-\varepsilon_1\Re(\sqrt{\tau})$ we can take $\mathcal{C}=i\sqrt{\tau}\BR+\varepsilon_1\sqrt{\tau}$. We can deform the contour of this integral so that for $z$ bounded away from $\BC_{\theta}$ by $\varepsilon$ the integral converges. This is given up to exponentially small terms by the same integral cut off with $|w|<|\sqrt{\tau}|$. Then by Taylor's theorem there exists a constant $C_1$ such that for $|w|<|\sqrt{\tau}|$ we have
\be
  \Big|\frac{1}{\e(-w/\tau)-1}
  -
  \sum_{k=0}^{K}\frac{B_{k}}{k!}\Big(\frac{-2\pi iw}{\tau}\Big)^{k-1}
  \Big|
  \;\leq\;
  C_1\Big|2\pi\frac{w}{\tau}\Big|^{K}
\ee
Therefore, we find that
\be
\begin{aligned}
  &\Big|\int_{\mathcal{C}}
  \Big(\frac{\e(zw)}{(\e(-w)-1)(\e(-w/\tau)-1)}
  -
  \frac{\e(zw)}{(\e(-w)-1)}\sum_{k=0}^{K}\frac{B_{k}}{k!}\Big(\frac{-2\pi iw}{\tau}\Big)^{k-1}
  \Big)
  \frac{dw}{w}
  \Big|\\
  &<C_1\,(2\pi)^{K}|\tau|^{-K}\int_{\mathcal{C}}
  \Big|\frac{\e(zw)}{w^{1-K}(\e(-w)-1)}\Big|
  dw\,.
\end{aligned}
\ee
This final integral can be split into two parts depending on which dominates $\max\{\e(-w),1\}$. The size of one of these integrals is then approximated by a constant times an integral of the form
\be
  \int_{0}^{\infty}
  \exp(-2\pi\varepsilon w)w^{K-1}
  dw
  \=
  (2\pi\varepsilon)^{-K}(k-1)!
\ee
The other integral has a similar bound. Therefore, we see that there is a constant $C$ such that
\be
\begin{aligned}
  &\Big|\int_{i\sqrt{\tau}\BR+\varepsilon_1\sqrt{\tau}}
  \Big(\frac{\e(zw)}{(\e(-w)-1)(\e(-w/\tau)-1)}
  -
  \frac{\e(zw)}{(\e(-w)-1)}\sum_{k=0}^{K}\frac{B_{k}}{k!}\Big(\frac{-2\pi iw}{\tau}\Big)^{k-1}
  \Big)
  \frac{dw}{w}
  \Big|\\
  &<C\,K!\,\varepsilon^{-K}\,|\tau|^{-K}.
\end{aligned}
\ee
\end{proof}

\begin{theorem}\label{thm:fad.asy.app}
Suppose that $\tau\in\BC\backslash\BR_{\leq0}$ and $\varepsilon\in\BR_{>0}$. Then as $|\tau|\rightarrow \infty$ with fixed argument and $z$ bounded away from $\BC_{\theta}$ by $\varepsilon$ with $\theta=\arg(-1/\tau)\in(0,2\pi)$ there exists a constant $C\in\BR_{>0}$ such that
\be
    \Big|\Phi(z\tau;\tau)
    -
    \e\Big(\mathrm{D}_{\theta}(z)\tau
    +
    \frac{1}{2}\mathrm{D}_{\theta}'(z)
    +
    \sum_{k=2}^{K}\frac{B_k}{k!}\frac{(2\pi i)^{k-2}}{\tau^{k-1}}\Li_{2-k}(\e(z))\Big)
    \Big|
    <
    C\,\varepsilon^{-K}\,K!\,|\tau|^{-K}\,.
\ee
\end{theorem}
\begin{proof}
The proof follows by combining the results of Lemma~\ref{lem:asymp.qpoch}, Lemma~\ref{lem:asymp.qpoch.2} and Lemma~\ref{lem:asymp.qpoch.3}.
\end{proof}
\begin{corollary}
If $\Im(z)>0$ and $-\arg(z)+\pi/2<\arg(\tau)<\arg(z+1)+\pi/2$, $\Im(z)<0$ and $-\pi/2-\arg(z+1)<\arg(\tau)<-\arg(z)-\pi/2$, or $z\in(-1,0)$ and $\tau\in\BC\backslash\BR_{\leq0}$, then $\Phi(z\tau;\tau)$ is the Borel--Laplace resummation of its asymptotics.
\end{corollary}
\begin{proof}
As $\tau\to\infty$ we have
\be
    \Big|\Phi(z\tau;\tau)
    -
    \e\Big(\mathrm{D}_{\theta}(z)\tau
    +
    \frac{1}{2}\mathrm{D}_{\theta}'(z)
    +
    \sum_{k=2}^{K}\frac{B_k}{k!}\frac{(2\pi i)^{k-2}}{\tau^{k-1}}\Li_{2-k}(\e(z))\Big)
    \Big|
    <
    C\,\varepsilon^{-K}\,K!\,|\tau|^{-K}\,.
\ee
Therefore, we see that $\Phi(z\tau,\tau)$ is the Borel--Laplace resummation of its asymptotics for $\tau$ in these cones.
\end{proof}

\bibliographystyle{plain}
\bibliography{biblio}
\end{document}